\documentclass[11pt,reqno,final]{amsart}
\usepackage{color}
\usepackage[colorlinks=true,allcolors=blue,backref=page]{hyperref}
\usepackage{amsmath, amssymb, amsthm}
\usepackage{mathrsfs}
\usepackage{mathtools}
\usepackage[noabbrev,capitalize,nameinlink]{cleveref}
\crefname{equation}{}{}
\usepackage{fullpage}
\usepackage[noadjust]{cite}
\usepackage{graphics}
\usepackage{pifont}
\usepackage{tikz}
\usepackage{bbm}
\usepackage[T1]{fontenc}

\usetikzlibrary{arrows.meta}

\usepackage{environ}
\usepackage{framed}
\usepackage{url}
\usepackage[linesnumbered,ruled,vlined]{algorithm2e}
\usepackage[noend]{algpseudocode}
\usepackage[labelfont=bf]{caption}
\usepackage{cite}
\usepackage{framed}
\usepackage[framemethod=tikz]{mdframed}
\usepackage{appendix}
\usepackage{graphicx}
\usepackage[textsize=tiny]{todonotes}
\usepackage{tcolorbox}
\usepackage{enumerate}
\allowdisplaybreaks[4]
\usepackage{enumerate}
\usepackage{stmaryrd}

\usepackage[shortlabels]{enumitem}
\crefformat{enumi}{#2#1#3}
\crefrangeformat{enumi}{#3#1#4 to~#5#2#6}
\crefmultiformat{enumi}{#2#1#3}
{ and~#2#1#3}{, #2#1#3}{ and~#2#1#3}

\DeclareSymbolFont{symbolsC}{U}{pxsyc}{m}{n}
\SetSymbolFont{symbolsC}{bold}{U}{pxsyc}{bx}{n}
\DeclareFontSubstitution{U}{pxsyc}{m}{n}
\DeclareMathSymbol{\medcircle}{\mathbin}{symbolsC}{7}

\usepackage{thmtools}
\usepackage{thm-restate}

\crefname{algocf}{Algorithm}{Algorithms}

\crefname{equation}{}{} 
\AtBeginEnvironment{appendices}{\crefalias{section}{appendix}} 

\usepackage[color,final]{showkeys} 

\colorlet{refkey}{orange!20}
\colorlet{labelkey}{blue!30}

\crefname{algocf}{Algorithm}{Algorithms}

\numberwithin{equation}{section}
\newtheorem{theorem}{Theorem}[section]
\newtheorem{proposition}[theorem]{Proposition}
\newtheorem{lemma}[theorem]{Lemma}
\newtheorem{claim}[theorem]{Claim}

\crefname{claim}{Claim}{Claims}

\newtheorem{corollary}[theorem]{Corollary}

\newtheorem*{question*}{Question}
\newtheorem{fact}[theorem]{Fact}

\theoremstyle{definition}
\newtheorem{definition}[theorem]{Definition}

\newtheorem*{definition*}{Definition}

\theoremstyle{remark}
\newtheorem{remark}[theorem]{Remark}
\newtheorem*{remark*}{Remark}

\newcommand{\norm}[1]{\bigg\lVert#1\bigg\rVert}
\newcommand{\snorm}[1]{\lVert#1\rVert}

\newcommand{\supp}{\operatorname{supp}}

\newcommand{\mbf}{\mathbf}
\newcommand{\mbm}{\mathbbm}
\newcommand{\mc}{\mathcal}
\newcommand{\mf}{\mathfrak}
\newcommand{\mr}{\mathrm}

\newcommand{\ol}{\overline}
\newcommand{\on}{\operatorname}

\newcommand{\wh}{\widehat}
\newcommand{\wt}{\widetilde}
\newcommand{\poly}{\on{poly}}

\newcommand{\eps}{\varepsilon}

\allowdisplaybreaks

\title{Effective bounds for Roth's theorem with shifted square common difference}

\author[A1]{Sarah Peluse}
\address{Department of Mathematics, University of Michigan, East Hall, 530 Church Street, Ann Arbor, MI 48109, USA}
\email{speluse@umich.edu}

\author[A2]{Ashwin Sah}
\author[A3]{Mehtaab Sawhney}
\address{Department of Mathematics, Massachusetts Institute of Technology, 77 Massachusetts Avenue, Cambridge, MA 02139, USA}
\email{\{asah,msawhney\}@mit.edu}

\begin{document}

\maketitle
\begin{abstract}
Let $S$ be a subset of $\{1,\ldots,N\}$ avoiding the nontrivial progressions $x, x+y^2-1, x+ 2(y^2-1)$. We prove that $|S|\ll N/\log_m{N}$, where $\log_m $ is the $m$-fold iterated logarithm and $m\in\mathbf{N}$ is an absolute constant. This answers a question of Green.
\end{abstract}

\section{Introduction}

This paper contributes to the program of proving reasonable bounds for sets lacking polynomial progressions, a problem posed by Gowers \cite[Problem~11.4]{Gow01} after his proof of the first reasonable bounds in Szemer\'edi's theorem on arithmetic progressions~\cite{Gow98,Gow01a}.

In the late 1970's, Furstenberg~\cite{Fur77} and S\'ark\"ozy~\cite{Sar78a} independently proved that any subset of the natural numbers having positive upper density must contain a nontrivial\footnote{Here, \emph{nontrivial} means that both terms of the progression are distinct.} instance of the progression $x,x+y^2$. Furstenberg's proof, which appeared in the same paper in which he introduced his eponymous correspondence principle and used it to give a proof of Szemer\'edi's theorem via ergodic theory, produced no quantitative bounds, but S\'ark\"ozy's proof, which was via the circle method, showed that if $S\subseteq\{1,\ldots,N\}$ contains no nontrivial progressions $x,x+y^2$, then
\begin{equation*}
    |S|\ll\frac{N}{(\log{N})^{1/3-o(1)}}.
\end{equation*}
S\'ark\"ozy~\cite{Sar78b} extended his argument to all progressions of the form $x,x+y^n$ with bounds of the same quality, which were later improved by Balog, Pelik\'an, Pintz, and Szemer\'edi~\cite{BPPS94} and then Bloom and Maynard~\cite{BM22}. Slijep\v{c}evi\'c~\cite{Sli03} further extended S\'ark\"ozy's argument to work for all two-term polynomial progressions $x,x+P(y)$ where $P(0)=0$.

Note that it cannot possibly be the case that the Furstenberg--S\'ark\"ozy theorem holds for every single polynomial progression $x,x+P(y)$ with $P\in\mathbf{Z}[y]$. Indeed, the set of multiples of $3$ have positive density in the integers, but contain no progressions of the form $x,x+y^2+1$ because $y^2+1$ is never divisible by $3$ when $y$ is an integer. Polynomials $P\in\mathbf{Z}[y]$ for which any subset of the natural numbers with positive upper density must contain a nontrivial polynomial progression of the form $x,x+P(y)$ are called \emph{intersective}. Kamae and Mend\'es France \cite{KM78} showed that a polynomial is intersective if and only if it has a root modulo every natural number. Polynomials $P\in\mathbf{Z}[y]$ with $P(0)=0$ clearly satisfy this criterion, and so does $y^2-1$ and, more generally, any other polynomial with an integer root. There also exist polynomials, like $(y^3-19)(y^2+y+1)$, that are intersective but have no rational roots. The argument of Kamae and Mend\'es France produced no quantitative bounds, but Lucier~\cite{Luc06} generalized S\'ark\"ozy's argument to show that if $P\in\mathbf{Z}[y]$ is intersective and $S\subseteq\{1,\ldots,N\}$ contains no nontrivial progressions $x,x+P(y)$, then
\begin{equation*}
    |S|\ll_P\frac{N}{(\log{N})^{1/(\deg{P}-1)-o(1)}}.
\end{equation*}
The bound has since been improved by Rice~\cite{Ric19}.

Bergelson and Leibman~\cite{BL96} proved that if $P_1,\ldots,P_m\in\mathbf{Z}[y]$ are any polynomials satisfying $P_1(0)=\cdots=P_m(0)=0$, then any subset of the natural numbers with positive upper density must contain a nontrivial polynomial progression of the form
\begin{equation}\label{eq:polyprog}
    x,x+P_1(y),\ldots,x+P_m(y).
\end{equation}
Their argument, which was via ergodic theory, produced no quantitative bounds. Gowers's proof of Szemer\'edi's theorem provides quantitative bounds in the case that $P_1,\ldots,P_m$ are all linear. Green \cite{Gre02} proved quantitative bounds for subsets of integers avoiding three-term arithmetic progressions with common difference equal to the sum of two squares. This was substantially generalized in work of Prendiville~\cite{Pre17} to prove the existence of $k$-term arithmetic progressions with common difference a perfect $d$-th power. Both papers \cite{Gre02,Pre17} build on Gowers's seminal work~\cite{Gow98,Gow01a} and, in particular, crucially rely on the homogeneous nature of these polynomial progressions to proceed via the the density increment strategy using the local inverse theorems for the $U^{s}$-norms. The progressions considered by Prendiville are the most general to which Gowers's methods can possibly apply, and no effective results were known for any other progressions of length greater than two until recently.

Progress on effective bounds on the size of sets lacking more polynomial progressions was made first in the finite field setting. Bourgain and Chang~\cite{BC17} proved that any $S\subseteq\mathbf{F}_p$ lacking nontrivial \emph{nonlinear Roth configurations} $x,x+y,x+y^2$ has size $|S|\ll p^{14/15}$. Similar polynomial saving bounds were proven in the case of more general progressions $x, x+P_1(y), x+P_2(y)$ for linearly independent polynomials $P_1(y)$ and $P_2(y)$ by the first author~\cite{Pel18} and, independently, Dong, Li, and Sawin~\cite{DLS20}. While the proofs of these results avoided the use of the inverse theory of the Gowers norms, the arguments did not extend to longer polynomial patterns. The first result in this direction was due to first author \cite{Pel19}, who introduced the degree-lowering method and used it to prove power-saving bounds for sets lacking arbitrarily long progressions~\eqref{eq:polyprog} with linearly independent polynomials $P_1,\ldots,P_m$. Degree-lowering was then used by Kuca~\cite{Kuc21} and Leng~\cite{Len22,Len23} to give effective bounds for subsets of finite fields avoiding various families of polynomial progressions of complexity\footnote{Here, \emph{complexity} refers to true complexity, as defined in~\cite{Kuca23}.} $1$ or greater.

The first author and Prendiville~\cite{PP19, PP22} adapted the degree-lowering method to the integer setting to prove that any subset $S$ of $\{1,\ldots,N\}$ lacking non-linear Roth configurations must satisfy
\begin{equation*}
    |S|\ll\frac{N}{(\log\log{N})^c}
\end{equation*}
for some absolute constant $c>0$. This was extended in work of the first author~\cite{Pel20} to arbitrarily long progressions~\eqref{eq:polyprog} where the polynomials $P_1,\ldots,P_m$ have all distinct degrees. Proving a fully general quantitative polynomial Szemer\'edi theorem remains a very challenging open problem, and effective bounds for sets lacking polynomial progressions~\eqref{eq:polyprog} of complexity at least one where the polynomials $P_1,\ldots,P_m$ are not homogeneous of the same degree are unknown in the integer setting.

Our work establishes the first effective case of the polynomial {S}zemer\'{e}di theorem over the integers where the underlying pattern has complexity higher than one and the polynomials involved are not homogeneous of the same degree.

\begin{theorem}\label{thm:main}
There exists a positive integer $m = m_{\ref{thm:main}}$ such that the following holds. If $S \subseteq \{1,\ldots, N\}$ is such that $S$ does not contain a progression of the form 
\begin{equation}\label{eq:config}
x,~x + y^2-1,~x + 2(y^2-1) \qquad (y\neq \pm 1), 
\end{equation}
then 
\[|S|\ll \frac{N}{\log_{m}{N}}.\]
\end{theorem}
The problem of proving quantitative bounds for sets lacking~\eqref{eq:config} was explicitly raised by Green \cite[Problem~11(i)]{GreOp}. 
\begin{remark*}
By tracing through our proof, and the inputs from \cite{Pel19}, we can take $m = 200$. By inserting some plausible improvements in the quantitative aspects of the theory of nilsequences, our argument would yield a bound of the form $|S|\ll N\exp(-(\log\log{N})^c)$ in the theorem; see the discussion at the end of \cref{sec:sketch}.
\end{remark*}

We give an outline of our key definitions, method, and new techniques in \cref{sec:notation,sec:sketch}, and describe the structure of the paper in \cref{sub:organization}.

\subsection*{Acknowledgments}
The first author thanks Sean Prendiville for helpful conversations. The second and third authors thank James Leng for helpful clarifications regarding \cite[Lemma~6.1]{Len22}. The third author thanks Dmitrii Zakharov for help with computations with nilpotent groups. The authors thank Ben Green for useful comments. The first author was supported by the NSF Mathematical Sciences Postdoctoral Research Fellowship Program under Grant No.~DMS-1903038. The second author was supported by the PD Soros Fellowship. The second and third authors were supported by NSF Graduate Research Fellowship Program DGE-2141064.

\section{Notation and key definitions}\label{sec:notation}

We use standard asymptotic notation throughout, as follows. For functions $f=f(n)$ and $g=g(n)$, we write $f=O(g)$ or $f \ll g$ to mean that there is a constant $C$ such that $|f(n)|\le C|g(n)|$ for sufficiently large $n$. Similarly, we write $f=\Omega(g)$ or $f \gg g$ to mean that there is a constant $c>0$ such that $f(n)\ge c|g(n)|$ for sufficiently large $n$. Finally, we write $f\asymp g$ or $f=\Theta(g)$ to mean that $f\ll g$ and $g\ll f$, and we write $f=o(g)$ or $g=\omega(f)$ to mean that $f(n)/g(n)\to0$ as $n\to\infty$. Subscripts on asymptotic notation indicate quantities that should be treated as constants. Furthermore, throughout the paper, we will use the standard notation $\mathbf{T}=\mathbf{R}/\mathbf{Z}$, $\mathbf{N}=\{1,2,\ldots\}$, $\mathbf{Z} = \{\ldots, -2, -1, 0, 1, 2, \ldots\}$, $[X] = \{1,2\ldots, \lfloor X\rfloor\}$, $[\pm X] = \{-\lfloor X\rfloor, \ldots, \lfloor X\rfloor\}$. Finally given a nonzero real $t$ and a set $Q$ we define $t\cdot Q = \{tq\colon q\in Q\}$.

One nonstandard piece of notation, following work of Tao and Ter{\"a}v{\"a}inen \cite{TT21}, is that we let $\on{poly}_m(Q)$ for $Q\ge 2$ denote a quantity bounded above by $\exp(\exp(m^{O(1)})) Q^{\exp(m^{O(1)})}$. For $0<\delta\le 1/2$, we let $\on{poly}_m(\delta)$ denote a quantity bounded \emph{below} by $\exp(-\exp(m^{O(1)})) \delta^{\exp(m^{O(1)})}$. Throughout the paper we will always assume $\delta\in(0,1/2]$.

 Given a function $f\colon\mathbf{Z}\to\mathbf{C}$ with $\snorm{f}_{\ell^{1}(\mathbf{Z})}<\infty$, we normalize the Fourier transform by defining
\[\wh{f}(\theta) = \sum_{x\in \mathbf{Z}}f(x)e(-x\theta),\]
where $e(x) = \exp(2\pi i x)$. Using this normalization, the Fourier inversion formula for $f$ satisfying $\snorm{f}_{\ell^{1}(\mathbf{Z})} + \snorm{f}_{\ell^{2}(\mathbf{Z})}<\infty$ is
\[f(x) = \int_{\mathbf{T}}\wh{f}(\Theta)e(x\Theta)~d\Theta.\] We define the normalized Fej\'er kernel on $\mathbf{Z}$ to be 
\begin{equation*}
\mu_H(h) = \frac{1}{\lfloor H\rfloor}\bigg(1-\frac{|h|}{\lfloor H\rfloor}\bigg)_{+}
\end{equation*}
and write
\begin{equation}\label{eq:fejer}
\mu_H(h)=\mu_H(h_1,\ldots,h_d)=\prod_{i=1}^d\mu_H(h_i)
\end{equation}
for $h=(h_1,\ldots,h_d)\in\mathbf{Z}^d$. 

We define two types of multiplicative discrete derivatives; for any complex-valued function $f$ on $\mathbf{Z}$ and $h,h_1,h_1'\in\mathbf{Z}$, set
\[\Delta_hf(x):=f(x)\ol{f(x+h)}\qquad\text{ and }\qquad\Delta'_{(h_1,h_1')} := \ol{f(x+h_1)}f(x+h_1').\]
We will occasionally write $\Delta_h^{(x)}f(x,y)$ and $\Delta_{(h,h')}'^{(x)}f(x,y)$, for example, if there are multiple possible variables to choose from, so that, for example, $\Delta_{h}^{(x)}f(x,y)=f(x,y)\overline{f(x+h,y)}$. With the definition of $\Delta'$ in hand, we can now define the Gowers box and uniformity norms. We will write expressions such as $\Delta_{(h_1,h_1'),(h_2,h_2')}'f(x)$ as shorthand for $\Delta_{(h_1,h_1')}'\Delta_{(h_2,h_2')}'f(x)$, and so on, where convenient; note the order of these operators does not matter.
\begin{definition}\label{def:box-norm}
Let $d\in\mathbf{N}$, $Q_1,\ldots,Q_d\subseteq\mathbf{Z}$ be finite subsets, and $f\colon\mathbf{Z}\to\mathbf{C}$. We define the Gowers box-norm of $f$ with respect to $Q_1,\ldots,Q_d$ to be
\[\snorm{f}_{\square_{Q_1,\ldots,Q_d}^{d}}^{2^{d}} :=\sum_{x\in \mathbf{Z}}\mathbf{E}_{\substack{h_i,h_i'\in Q_i\\i=1,\ldots,d}}\Delta_{(h_1,h_1'),\ldots,(h_d,h_d')}'f(x).\]
For $Q = Q_1 = \cdots = Q_d$ define
\[\snorm{f}_{U_{Q}^{d}}:=\snorm{f}_{\square_{Q,\ldots,Q}^{d}}.\]
\end{definition}
Note that our definition differs from that in \cite[Definition~2.1]{Pel20}, as the sum is not normalized.

For the entirety of the paper, we define 
\begin{equation}\label{eq:Wdef}
W = \prod_{\substack{2\le p\le w\\p\text{ prime}}}p,\qquad M = \lfloor N^{1/2}W^{-1/2}\rfloor,\qquad\text{and}\qquad P(y) = Wy^2 + y,
\end{equation}
for some parameter $w$. Eventually, $w$ will be chosen to be a sufficiently slowly growing function of $N$; throughout the paper, we ensure that various implied constants are independent of $W$. It is elementary to prove that $W\le 4^{w}$. As stated, $M$ is a function of a floating parameter $N$; $N$, up to a constant factor, will always denote the size of the support of the sets or functions under consideration. 

Next, we define the critical counting operators to be used throughout the paper.
\begin{definition}\label{def:W-operator}
Given $N$ and given finitely supported functions $f_1,f_2,f_3\colon\mathbf{Z}\to\mathbf{C}$, we define the trilinear operators $\Lambda^{W}$ and $\Lambda^{\mr{Model}}$ by
\[
\Lambda^{W}(f_1,f_2,f_3) = \sum_{\substack{x\in \mathbf{Z}\\|k|\le M}}f_1(x)f_2(x+P(k))f_3(x+2P(k))\\
\]
and
\[
\Lambda^{\mr{Model}}(f_1,f_2,f_3) = \sum_{\substack{x\in \mathbf{Z}\\d\in \mathbf{Z}}}f_1(x)f_2(z+d)f_3(z+2d)\nu(d),
\]
where 
\begin{equation}\label{eq:nu-def}
\nu(d) = \sqrt{\frac{N}{d}}\mbm{1}_{1\le d\le N}.
\end{equation}
We also define the ``difference'' counting operator
\begin{equation*}
\wt{\Lambda}(f_1,f_2,f_3) := (NW)^{1/2}\Lambda^{W}(f_1,f_2,f_3) - \Lambda^{\mr{Model}}(f_1,f_2,f_3).    
\end{equation*}
\end{definition}

Finally, we will repeatedly encounter the following \emph{dual functions} when carrying out our degree-lowering argument. 
\begin{definition}\label{def:dual}
Given functions $f_1,f_2,f_3\colon\mathbf{Z}\to\mathbf{C}$, we define 
\[\mc{D}^1(f_2,f_3)(x) = \mathbf{E}_{y\in [\pm M]}f_2(x+P(y))f_3(x+2P(y)),\]
\[\mc{D}^2(f_3,f_1)(x) = \mathbf{E}_{y\in [\pm M]}f_1(x-P(y))f_3(x+P(y)),\]
and
\[\mc{D}^3(f_1,f_2)(x) = \mathbf{E}_{y\in [\pm M]}f_1(x-2P(y))f_2(x-P(y)).\]
\end{definition}
These dual functions arise in a key maneuver in the degree-lowering method known as \emph{stashing}, a term coined by Manners. More discussion on stashing can be found in~\cite{Man21}, but for us it will almost always refer to the procedure of noting that if $f_1,f_2,f_3\colon\mathbf{Z}\to\mathbf{C}$ are $1$-bounded functions supported in $[N]$ and
\begin{equation*}
\left|\Lambda^W(f_1,f_2,f_3)\right|\geq 2\delta NM,
\end{equation*}
then
\begin{equation*}
 \left| \Lambda^W\left(\mathcal{D}^1(f_2,f_3),\overline{f_2},\overline{f_3}\right)\right|,\left| \Lambda^W\left(\overline{f_1},\mathcal{D}^2(f_1,f_3),\overline{f_3}\right)\right|,\left| \Lambda^W\left(\overline{f_1},\overline{f_2},\mathcal{D}^3(f_1,f_2)\right)\right|\gg\delta^2 NM.
\end{equation*}
This is a simple consequence of the Cauchy--Schwarz inequality. For example, we have
\begin{equation*}
    \Lambda^W(f_1,f_2,f_3)= \sum_{x\in \mathbf{Z}}f_1(x)\cdot\left(\sum_{k\in[\pm M]}f_2(x+P(k))f_3(x+2P(k))\right),
\end{equation*}
which is bounded above by
\begin{align*}
    &N^{1/2}\left(\sum_{x\in \mathbf{Z}}\sum_{k,k'\in[\pm M]}f_2(x+P(k))f_3(x+2P(k))\overline{f_2(x+P(k'))f_3(x+2P(k'))}\right)^{1/2} \\
    &=N^{1/2}\left(\sum_{\substack{x\in \mathbf{Z}\\ k'\in[\pm M]}}\left(\sum_{k\in[\pm M]}f_2(x+P(k))f_3(x+2P(k))\right)\overline{f_2(x+P(k'))f_3(x+2P(k'))}\right)^{1/2}
\end{align*}
by the Cauchy--Schwarz inequality. Rearranging now yields $\left| \Lambda^W(\mathcal{D}^1(f_2,f_3),\overline{f_2},\overline{f_3})\right|\gg\delta^2 NM$, and the other two inequalities are proved similarly.

\section{Proof sketch}\label{sec:sketch}
The starting point of our work is to use the $W$-trick of Green \cite{Gre05} to compare the count of certain three-term arithmetic progressions with shifted square common difference to the count of all three-term arithmetic progressions in a set, and then apply quantitative lower bounds for the number of three-term arithmetic progressions coming from Roth's theorem. This is closely motivated by work of Wooley and Ziegler \cite{WZ12}, who proved a version of the polynomial {S}zemer\'{e}di theorem with $y$ restricted to the set of shifted primes via such an approach. 

We will show for any $1$-bounded functions $f_1,f_2,f_3\colon\mathbf{Z}\to\mathbf{C}$ with support in $[N]$ that 
\begin{equation}\label{eq:transference}
\bigg|(NW)^{1/2}\Lambda^{W}(f_1,f_2,f_3) - \Lambda^{\mr{Model}}(f_1,f_2,f_3)\bigg|\ll\frac{N^2}{\log_m{N}}
\end{equation}
for some absolute constant $m\in\mathbf{N}$ (recall that $W$ will ultimately be chosen to be slowly growing with $N$). \cref{thm:main} then follows by dividing $S$ into classes modulo $4W$, shifting an appropriately dense congruence class of $S$ and scaling by $(4W)^{-1}$, noting that differences in this rescaled set of the form $Wy^2 + y$ correspond to differences of the form $y^2-1$ in the original set, and applying supersaturation results for Roth's theorem. 

The crux of our proof of \cref{thm:main} is establishing that the ``difference'' counting operator 
\[\wt{\Lambda}(f_1,f_2,f_3) = (NW)^{1/2}\Lambda^{W}(f_1,f_2,f_3) - \Lambda^{\mr{Model}}(f_1,f_2,f_3)\]
is controlled by the $U^{2}$-norm of the functions $f_i$ (or, more precisely, the $U^2_{W\cdot[N/W]}$-norm). Given such norm control, combining a variant of stashing with the $U^2$-inverse theorem implies that there exist linear phase functions $\psi_1,\psi_2,\psi_3$ such that the counting operator $\wt{\Lambda}\left(\psi_11_{[N]},\psi_21_{[N]},\psi_31_{[N]}\right)$ is large. The existence of such phase functions is ruled out by a direct Fourier analytic computation. Indeed, the weight function $\nu(d)$ is chosen so that the corresponding exponential sums closely matches that of $P(k)$, and the $W$-trick serves to remove the major arc contributions initially present in the Fourier transform of the squares. 

It follows from the triangle inequality that in order to establish $U^{2}$-norm control of the counting operator $\wt{\Lambda}$, it suffices to establish the result for the counting operators $\Lambda^{W}$ and $\Lambda^{\mr{Model}}$ separately. The (far) simpler of these two tasks is establishing $U^{2}$-norm control for $\Lambda^{\mr{Model}}(f_1,f_2,f_3)$. Note that if $\nu(d)$ were absent, then this is precisely the fact that the $U^2$-norm controls the count of three-term arithmetic progressions weighted by $f_1,f_2,$ and $f_3$. The result for $\Lambda^{\mr{Model}}$ follows by noting that the Fourier transform of $\nu(d)$ (after a bit of smoothing) is appropriately bounded in $L^1$. 

The vast majority of the paper, therefore, is devoted to establishing $U^2$-control of the operator $\Lambda^{W}$. We will do this by using the degree-lowering method, following work of the first author~\cite{Pel19} and the first author and Prendiville \cite{PP19, PP22}. This method, in our setting, can be broken down into two steps. First, we establish that $\Lambda^{W}$ is controlled by some high degree Gowers $U^s$-norm, and then we show (essentially) that $U^{t}$-norm control of $\Lambda^W$ implies $U^{t-1}$-norm control of $\Lambda^W$ whenever $t\geq 3$. These two steps taken together imply the desired $U^2$-norm control. The first step is proven by combining the PET induction scheme of Bergelson and Leibman~\cite{BL96} with the quantitative concatenation results of~\cite{Pel20}, which we can use as a black box. The majority of our effort, therefore, is concentrated on the second step of the argument.

Via an application of stashing, the key to the second step of our argument is establishing that 
\begin{equation}\label{eq:degree-lower}
\snorm{\mc{D}^{1}(f_2,f_3)}_{U^{k}_{W\cdot[N/W]}}^{2^k}\ge \delta N \Longrightarrow \snorm{f_i}_{U^{k-1}_{W\cdot[N/W]}}^{2^{k-1}}\ge \delta' N   
\end{equation}
for $k \ge 3$ and $i\in \{2,3\}$ (and the analogous statement for $\mc{D}^{3}(f_1,f_2)$). This, combined with further applications of stashing, implies $U^2$-norm control of $\Lambda^W$. By dual-difference interchange (\cref{lem:dual-interchange}), it is essentially sufficient to prove the result when $k=3$, and for the remainder of the sketch we will focus on this special case. 

First, let us pretend, for the sake of illustration, that the $U^3$-inverse theorem implied large correlation with a global quadratic form $e(\alpha x^2 + \beta x)$. This is, of course, a lie due to the existence of bracket-polynomials, but will help to motivate the main technical considerations. Furthermore, suppose for the sake of discussion that $\snorm{\mc{D}^{1}(f_2,f_3)}_{U^{k}_{[N]}}$ is large; this is a rather minor technical point that can be handled by splitting into congruence classes modulo $W$. It then follows from our ``fake'' $U^3$-inverse theorem that 
\begin{equation}\label{eq:fake1}
\bigg|\frac{1}{N}\sum_{x\in \mathbf{Z}}e(\alpha x^2 + \beta x)\mathbf{E}_{y\in [\pm M]}f_2(x+P(y))f_3(x + 2P(y))\bigg|
\end{equation}
is large. Setting
\begin{equation*}
\wt{f}_2(x) := f_2(x)\cdot e(2\alpha x^2 + 2\beta x)
\end{equation*}
and
\begin{equation*}
\wt{f}_3(x) := f_3(x)\cdot e(-\alpha x^2 - \beta x),
\end{equation*}
and, as in work of Leng~\cite{Len22}, using the polynomial identities
\begin{equation*}
x^2 = 2(x+P(y))^2 - (x+2P(y))^2 + 2P(y)^2
\end{equation*}
and
\begin{equation*}
x = 2(x+P(y)) - (x+2P(y)),
\end{equation*}
we get, by rearranging~\eqref{eq:fake1}, that
\[\bigg|\frac{1}{N}\sum_{x\in \mathbf{Z}}\mathbf{E}_{y\in [\pm M]}\wt{f}_2(x+P(y))\wt{f}_3(x+2P(y))e(2\alpha P(y)^2)\bigg|\]
is large. By applying Fourier inversion to $\wt{f}_2$ and $\wt{f}_3$ and then using orthogonality of characters and Parseval's identity, it follows that
\[\sup_{\kappa\in \mathbf{T}}\bigg|\mathbf{E}_{y\in [\pm M]}e(2\alpha P(y)^2 + \kappa P(y))\bigg|\]
is large. Using Weyl's inequality, and carefully analyzing various terms in the expansion of $P(y)^2$, shows that $\alpha$ and $\kappa$ are essentially major arc. More precisely, there exists a positive integer $q$ such that $q\le \delta^{-O(1)}$ and $\snorm{q\alpha}_{\mathbf{T}}\le \delta^{-O(1)}/N^2$ and $\snorm{q\kappa}_{\mathbf{T}}\le \delta^{-O(1)}/N$. This computation is a bit delicate; one needs that the coefficients of $P(y)$ are coprime in order to avoiding sacrificing factors of $W$. To simplify the rest of our discussion, we will pretend that, in fact, Weyl's inequality implies that $\alpha, \kappa = 0$; by passing to intervals of length $\delta^{O(1)}N$ and spacing at most $\delta^{-O(1)}$, one can turn this fantasy into a reality. 

Note that if $\alpha = 0$, we would have that 
\[\bigg|\frac{1}{N}\sum_{x\in \mathbf{Z}}e(\beta x)\mathbf{E}_{y\in [\pm M]}f_2(x+P(y))f_3(x + 2P(y))\bigg|\]
is large. Applying the second of our two identities, we may rewrite the above quantity as 
\[\bigg|\frac{1}{N}\sum_{x\in \mathbf{Z}}\mathbf{E}_{y\in [\pm M]}f_2(x+P(y))e(2\beta(x+P(y))f_3(x + 2P(y))e(-\beta(x+2P(y))\bigg|,\]
which, by making the change of variables $x\mapsto x-P(y)$, equals
\[\bigg|\frac{1}{N}\sum_{x\in \mathbf{Z}}\mathbf{E}_{y\in [\pm M]}f_2(x)e(2\beta x)f_3(x + P(y))e(-\beta(x+P(y))\bigg|.\]
That $f_2$ and $f_3$ must have large $U^2$-norms then follows by $U^2$-control for the configuration $(x, x+P(y))$, which is implicit in work of S\'{a}rk\"{o}zy \cite{Sar78a}; this is a simple consequence of Fourier inversion, orthogonality of characters, and the Gowers--Cauchy--Schwarz inequality.

To rigorously prove the implication~\eqref{eq:degree-lower} we must use the $U^3$-inverse theorem of Green and Tao~\cite{GT08b} in place of our ``fake'' $U^3$-inverse theorem. The Green--Tao inverse theorem produces a Lipschitz function $F$ on a degree $2$ nilmanifold $G/\Gamma$ and a polynomial sequence $g\colon\mathbf{Z}\to G$ (in the sense of \cref{def:polyseq}) such that 
\[
\bigg|\sum_{x\in \mathbf{Z}}F(g(x))\mathbf{E}_{y\in [\pm M]}f_2(x+P(y))f_3(x + 2P(y))\bigg|\ge \exp(-\delta^{-O(1)}) N.
\]
Mimicking our simplified sketch above, we now want to ``factor'' $F(g(x))$ into terms involving $x+P(y)$, $x+2P(y)$, and $P(y)$. Leng~\cite{Len22} accomplishes such a maneuver for the pattern $(x,x+P(y),x+Q(y),x+P(y)+Q(y))$ over finite fields via a vertical Fourier expansion of $F$ and noting that the Host-Kra cube of dimension $3$ has a constrained orbit for any degree $2$ polynomial sequence on a nilmanifold. 

In our case, however, the constraints coming from the Host-Kra cube are insufficient, and to proceed directly one would require a suitable understanding of the orbits of the linear forms $(x,y,x+y,x+2y)$ for a degree $2$ polynomial sequence on a nilmanifold. The understanding of such an orbit is rather delicate, as this set of forms does not satisfy the \emph{flag condition}, and the underlying equidistribution theory has only recently been addressed in work of Altman~\cite{Alt22b}. However, by using an earlier ``lifting'' trick of Altman \cite{Alt22}, which amounts to a simple change of variables in our setting, it instead suffices to constrain the orbit of $6x$ given the images of $(6y,3(x+y),2(x+2y))$. As the pattern $(6x,6y,3(x+y),2(x+2y))$ is translation invariant, the flag-equidistribution theory developed in work of Green and Tao \cite{GT10b} applies, and one can then derive the necessary constraint. We do this by following \cite[Section~14]{GT08b}, which establishes the analogous result for $k$-term arithmetic progressions, although various related results appear earlier in the ergodic theory literature \cite{BHK05, Fur96, Zie05}.

Having obtained a suitable constraint, we will next require a suitable analogue of Weyl's inequality for nilsequences. This can be found in the seminal paper of Green and Tao on the equidistribution of polynomial orbits on nilmanifolds \cite{GT12}. The main technical result of this work \cite[Theorem~1.9]{GT12} essentially proves that if a polynomial sequence $g(\cdot)$ fails to equidistribute on a nilmanifold, one can identify an abelian reason for it. Using this result we will prove that if the polynomial sequence $g(P(6y))$ fails to equidistribute, then one can factor the polynomial sequence $g$. By tracking carefully with Mal'cev coordinates (analogously to the sketch with Weyl's inequality earlier), one can prove that the factorization is of the same quality as if one knew that instead $g(y)$ failed to equidistribute. While such a factorization itself is not immediately useful, via iterating the factorization (as in the factorization results of \cite{GT12}), one can prove that instead of correlating with a degree two nilsequence, one, in fact, correlates with a degree one nilsequence. The form of our result is closely motivated by earlier work of Leng \cite[Lemma~6.1]{Len22}. Given such a result, and then Fourier expanding the degree one nilsequence, we can reduce to dealing with pure polynomial phases, and the analysis follows as sketched earlier. 

We end our discussion with a brief remark on bounds in the implication~\eqref{eq:degree-lower}. Our bounds are of iterated logarithmic type, as $\delta'$ is ultimately doubly-exponentially small in $\delta$. The first of these exponential terms is derived from the fact that we use the the $U^3$-inverse theorem of Green and Tao \cite{GT08b}; given more recent work of Sanders \cite{San12} the correlation could be improved to quasi-polynomial. The second source of exponentials comes from the double-exponential dependence on dimension implicit in \cite[Theorem~7.1]{GT12}; this dependence was quantified explicitly in recent work of Tao and Ter{\"a}v{\"a}inen \cite{TT21}. Therefore, even using the results of Sanders \cite{San12}, our bounds involve a large number of logs. Recently, however, the dimension dependence in results of Green and Tao \cite{GT12} have been improved to exponential for periodic nilsequences in work of Leng \cite{Len23}; Leng has also announced analogous results for all nilsequences, and, by inputting such results into our work (along with the necessary quantitative versions of results in \cite[Appendix~A]{GT12}), a substantially reduced number of logs would be achieved (likely yielding $\ll N\exp(-(\log\log N)^{c})$ in \cref{thm:main}). 

\subsection{Organization of the paper}\label{sub:organization}
In \cref{sec:model-control}, we prove $U^2$-control for $\Lambda^{\mr{Model}}$. In \cref{sec:nil-main}, we prove the constraints for degree $2$ nilmanifold orbits (in \cref{sub:leib}) and the necessary factorization theorem (in \cref{sub:factor}). In \cref{sec:deg-lower} we prove the main degree-lowering statement in this work. In \cref{sec:main}, we complete proof of \cref{thm:main}. In \cref{sec:nilmanifold}, we collect various definitions and basic properties regarding nilmanifolds. In \cref{sec:circle}, we collect various standard exponential sum estimates for the polynomial $P(y) = Wy^2 + y$. Finally in \cref{sec:random}, we collect various basic estimates regarding changing parameters in the box-norm.

\section{Control for \texorpdfstring{$\Lambda^{\mr{Model}}$}{LambdaModel}}\label{sec:model-control}
In this section, we will establish $U^2$-norm control of $\Lambda^{\mr{Model}}$ and deduce a uniform lower bound for $\Lambda^{\mr{Model}}$ from the best known bounds in Roth's theorem.

\begin{lemma}\label{lem:model-control}
Let $f_1,f_2,f_3\colon\mathbf{Z}\to\mathbf{C}$ be $1$-bounded functions supported on $[\pm \delta^{-1} N]$. If $N\gg \delta^{-O(1)}$ and
\[\left|\Lambda^{\mr{Model}}(f_1,f_2,f_3)\right|\ge \delta N^2,\]
then
\[\min_{i\in [3]}~\snorm{f_i}^{4}_{U^{2}_{[N]}}\gg \delta^{O(1)}N.\]
\end{lemma}
\begin{proof}
By adjusting implicit constants, we may assume that $\delta$ is smaller than an absolute constant. Define 
\begin{equation*}
\nu^{(1)}(d) = \sqrt{\frac{N}{d}}\mbm{1}_{\delta^{5}N\le d\le N}
\end{equation*}
and
\begin{equation*}
\tau(d) = \frac{\mbm{1}_{|d|\le \delta^{10}N}}{2\delta^{10}N}.
\end{equation*}
Noting that $\nu^{(1)}(d)$ is $\delta^{-8}/N$-Lipschitz away from the boundary of its support and recalling the definition~\eqref{eq:nu-def} of $\nu$, we have
\[\sum_{d\in \mathbf{Z}}|(\tau\ast \nu^{(1)})(d) - \nu(d)|\le \sum_{d\in \mathbf{Z}}\left(|(\tau\ast \nu^{(1)})(d) - \nu^{(1)}(d)| + |\nu^{(1)}(d) - \nu(d)|\right)\le \delta^{2}N.\]
Therefore, since the $f_i$ are $1$-bounded,
\[\sum_{x,d\in \mathbf{Z}}f_1(x)f_2(x+d)f_3(x+2d)(\tau\ast \nu^{(1)})(d)\ge \delta N^2/2.\]
Furthermore, we have by orthogonality of characters, Cauchy--Schwarz, and Parseval that
\begin{align*}
&\bigg|\sum_{x,d\in \mathbf{Z}}f_1(x)f_2(x+d)f_3(x+2d)(\tau\ast \nu^{(1)})(d)\bigg|\\
&=\bigg|\int_{\mathbf{T}^2}\wh{f_1}(\Theta_1)\wh{f_2}(-2\Theta_1 + \Theta_2) \wh{f_3}(\Theta_1 -\Theta_2)(\wh{\tau\ast \nu^{(1)}}(\Theta_2))~d\Theta_1d\Theta_2\bigg|\\
&\le \int_{\mathbf{T}}|\wh{\tau\ast \nu^{(1)}}(\Theta_2)|~d\Theta_2 \cdot \sup_{\Theta_2\in \mathbf{T}}\int_{\mathbf{T}}|\wh{f_1}(\Theta_1)|\cdot |\wh{f_2}(-2\Theta_1 + \Theta_2)|\cdot |\wh{f_3}(\Theta_1 -\Theta_2)|~d\Theta_1\\
&\le \int_{ \mathbf{T}}|\wh{\tau}(\Theta_2)| |\wh{\nu^{(1)}}(\Theta_2)|~d\Theta_2 \cdot \sup_{\mathbf{T}} |\wh{f_1}(\Theta_1)| \cdot \sup_{\Theta_2\in \mathbf{T}}\int_{\Theta_1\in \mathbf{T}}|\wh{f_2}(-2\Theta_1 + \Theta_2)|\cdot |\wh{f_3}(\Theta_1 -\Theta_2)|~d\Theta_1\\
&\le \snorm{\tau}_{\ell^2(\mathbf{Z})}\snorm{\nu^{(1)}}_{\ell^2(\mathbf{Z})}\sup_{\Theta_1\in \mathbf{T}} |\wh{f_1}(\Theta_1)|\cdot \snorm{f_2}_{\ell^2(\mathbf{Z})}\snorm{f_3}_{\ell^2(\mathbf{Z})}\\
&\ll (\delta^{-10}N^{-1})^{1/2} \cdot (N\log(1/\delta))^{1/2} \cdot \sup_{\Theta_1\in \mathbf{T}} |\wh{f_1}(\Theta_1)| \cdot N \\
&\ll \delta^{-5}(\log(1/\delta))^{1/2} N\sup_{\Theta\in\mathbf{T}}\left|\widehat{f_1}(\Theta)\right|.
\end{align*}
An analogous inequality holds for $f_2$ and $f_3$, and therefore
\[\inf_{i\in [3]}\sup_{\Theta\in \mathbf{T}}|\wh{f_i}(\Theta)|\gg \delta^{O(1)}N;\]
the result now follows from the converse of the $U^{2}$-inverse theorem (see, e.g., \cref{lem:converse}).
\end{proof}

We next establish a uniform lower bound on $\Lambda^{\mr{Model}}$ using recent breakthrough work of Kelley and Meka \cite{KM23}.
\begin{lemma}\label{lemma:lower-bound}
Suppose that $f\colon\mathbf{Z}\to[0,1]$ with $\on{supp}(f)\in [N]$ and $\sum_{x\in \mathbf{Z}}f(x)\ge \delta N$. Then
\[\Lambda^{\mr{Model}}(f,f,f)\gg \exp\left(-\log(2/\delta)^{O(1)}\right)N^2.\]
\end{lemma}
\begin{proof}
Noting that $w(d)\ge 1$ for all $1\le d\le N$, the result follows from \cite[Theorem~1.2]{KM23}.
\end{proof}

\section{Nilmanifold considerations}\label{sec:nil-main}
Throughout this section, we will assume familiarity with standard terminology related to nilsequences and nilmanifolds. All terminology used is defined in \cref{sec:nilmanifold}; our conventions match those in \cite{GT12,TT21}. Furthermore, throughout this section, we will require various quantitative rationality claims from~\cite[Appendix~A]{GT12}, but with explicit dimensional dependencies. As stated in \cite[pg.~52]{TT21}, all bounds of the form $Q^{O_m(1)}$ in \cite[Appendix~A]{GT12} may in fact be taken to be $\poly_m(Q)$ (where $m$ is the dimension of the underlying nilmanifold). We will cite bounds from \cite[Appendix~A]{GT12}, but assume this more explicit dimensional quantification.

\subsection{Leibman group considerations}\label{sub:leib}
Throughout this subsection, define 
\[\tau(x,y) := (2(x+2y), 3(x+y), 6y, 6x)\]
for all $x,y\in\mathbf{Z}$. We will write $\tau_i(x,y)$, for $i=1,\ldots,4$, to refer to the $i$-th coordinate of $\tau(x,y)$. The key output of this subsection will be \cref{lem:dual-expansion}, which relates the values of a degree $2$ polynomial sequence at the first three coordinates of $\tau(x,y)$ to the value at the fourth coordinate.

\begin{lemma}\label{lem:dual-expansion}
Let $G/\Gamma$ be a filtered nilmanifold of dimension $m$, degree $2$, and complexity at most $L$. Let $F$ be a function on $G/\Gamma$ with vertical frequency $\xi$ with $|\xi|\le L$ and $\snorm{F}_{\mr{Lip}}\le L$. Let $g(\cdot)$ be a polynomial sequence with respect to $G/\Gamma$ (and the corresponding degree $2$ filtration). There exist $G$ and $F_{j,\alpha}$ such that for all $x,y\in \mathbf{Z}$,
\[F(g(\tau_4(x,y))\Gamma) = \sum_{\alpha} \prod_{j\in [3]} F_{j,\alpha}(g(\tau_j(x,y))\Gamma) + G(x,y),\]
where
\begin{itemize}
    \item $\snorm{G}_{\infty}\le L^{-1}$;
    \item for all $\alpha$, $F_{1,\alpha}$ has vertical frequency $-9\xi$, $F_{2,\alpha}$ has vertical frequency $8\xi$, and $F_{3,\alpha}$ has vertical frequency $2\xi$;
    \item there are $\on{poly}_{m}(L)$ summand indices $\alpha$; and
    \item we have $\snorm{F_{j,\alpha}}_{\mr{Lip}}\le \on{poly}_{m}(L)$ for all $\alpha$ and $j\in\{1,2,3\}$.
\end{itemize}
\end{lemma}

The key input into \cref{lem:dual-expansion} is that the image of $\tau(x,y)$ under a polynomial sequence on a nilmanifold is constrained. An analogous result $k$-term arithmetic progressions appears in \cite[Lemma~12.7]{GT08b}, and for the Host-Kra cube in \cite[Proposition~11.5]{GT10}. Our proof is essentially identical to that of \cite[Lemma~12.7]{GT08b} modulo certain algebraic issues regarding the Leibman group \cite{Lei10}.

We first require the notion of being continuous right-invertible. 
\begin{definition}\label{def:right-invertible}
Let $N_1, N_2$ be compact topological spaces, let $\pi\colon N_1\to N_2$ be a continuous map, and let $\Sigma\subseteq N_1$. We say that $\pi$ is \emph{continuously right-invertible} on $\Sigma$ if, for all $w\in \ol{\pi(\Sigma)}$, there exists a neighborhood $V_{w}\subseteq N_2$ of $w$ and a continuous map $\pi_w^{-1}\colon V_{w}\to N_1$ such that $\pi_w^{-1}\circ \pi$ is the identity map on $\Sigma \cap \pi^{-1}(V_w)$.
\end{definition}

Now we can precisely described the aforementioned constraint.

\begin{lemma}\label{lem:constraint}
Let $G/\Gamma$ be a be a filtered nilmanifold of dimension $m$, degree $2$, and complexity at most $L$. Let $G_{\bullet}$ denote the degree two filtration $G_0 = G_1\geqslant G_2\geqslant \mr{Id}_G$ on $G$ and $\mc{X}$ denote the chosen Mal'cev basis for $G/\Gamma$. Furthermore, define 
\[G^{\tau} := \{(g_0,g_0g_1,g_0g_1^{-2}g_2,g_0g_1^{4}g_2^{2})\colon g_i\in G_i\}.\]
Let $\pi\colon(G/\Gamma)^{4}\to(G/\Gamma)^{3}$ denote the standard projection onto the first three coordinates. Then there exists a compact set $\Sigma\subseteq(G/\Gamma)^{3}$ and a continuous function $Q\colon\Sigma\to G/\Gamma$ such that
\begin{itemize}
    \item $\pi(G^{\tau}\Gamma^{4})\subseteq\Sigma$;
    \item $Q(\pi(g\Gamma^{4})) = g_4\Gamma$ for all $g = (g_1, g_2, g_3, g_4)\in G^{\tau}$; and
    \item $Q$ is $\poly_m(L)$-Lipschitz, where the metric on $\Sigma$ is given by restricting \[\wt{d}((x_1,x_2,x_3),(z_1,z_2,z_3)) = \sum_{i\in [3]}d_{\mc{X}}(x_i\Gamma,z_i\Gamma)\] to $\Sigma$.
\end{itemize}
\end{lemma}
\begin{remark}
By $G^\tau\Gamma^4$ we mean the image of $G^\tau$ under taking $\Gamma$-cosets.
\end{remark}
\begin{proof}
Take $\Sigma = \ol{\pi(G^{\tau}\Gamma^{4})}$, and define 
\begin{align*}
G^{\tau}_0 &= G^{\tau}; \\
G_1^{\tau} &= \{(\mr{Id}_{G},g_1,g_1^{-2}g_2,g_1^{4}g_2^{2})\colon g_i\in G_i\};\\
G_2^{\tau} & = \{(\mr{Id}_{G},\mr{Id}_{G},g_2,g_2^{2})\colon g_2\in G_2\};\text{ and }\\
G_3^{\tau} & = \{(\mr{Id}_{G},\mr{Id}_{G},\mr{Id}_{G},\mr{Id}_{G})\}.
\end{align*}
Our argument is identical to \cite[Section~14]{GT08b}, aside from verifying that $G_i^{\tau}$ are groups. This can be verified using general results of Green and Tao \cite{GT10b}; we provide a short argument specialized to our case. It is trivial to verify that $G_2^{\tau}$ and $G_3^{\tau}$ are groups. That $G_0^{\tau}$ is a group follows from noting that
\begin{align*}
    \left\{(g(0),g(1),g(-2),g(4))\colon g\in\on{Poly}(\mathbf{Z},G_\bullet)\right\} &= \left\{(g_0,g_0g_1,g_0g_1^{-2}g_2^{3},g_0g_1^{4}g_2^{6})\colon g_i\in G_i\text{ for }i=0,1,2\right\} \\
    &= \left\{(g_0,g_0g_1,g_0g_1^{-2}g_2,g_0g_1^{4}g_2^2)\colon g_i\in G_i\text{ for }i=0,1,2\right\} \\
    &= G_0^\tau,
\end{align*}
where we have used that $G_2$ is divisible (since $G_2$, being a connected nilpotent Lie group, has surjective exponential map),
and recalling that $\on{Poly}(\mathbf{Z},G_\bullet)$ is a group. That $G_1^{\tau}$ is a group simply follows from noting that it is the intersection of two groups:
\begin{equation*}
 G_1^{\tau}=G_0^{\tau} \cap (\mr{Id}_G\times G_0\times G_0\times G_0).
\end{equation*}
Finally, observe that the groups $G_i^\tau$ have the nesting property
\[G_3^{\tau}\subseteq G_2^{\tau}\subseteq G_1^{\tau}\subseteq G^{\tau}.\]

Next, we will prove inductively that the restriction of $\pi$ is continuously right-invertible on $G_i^{\tau}/\Gamma^{4}$, starting at $i = 3$ and proceeding downwards. The crucial point is that the first non-identity coordinate in a generic element of $G_i^\tau$ is $g_i$ and, by inverting the quotient map $G_i\to G_i/\Gamma$ locally, we can ``remove'' $g_i$ and proceed inductively. We now give a formal proof following \cite[Section~14]{GT08b}.

Note that $G_3^{\tau}$ is isomorphic to the trivial group, and therefore $\pi$ is trivially continuously right invertible on $G_3^{\tau}/\Gamma^4$. Suppose that the restriction of $\pi$ to $G_{i+1}^{\tau}/\Gamma^4$ is continuously right invertible for some $0\leq i\leq 2$; we will show that the same holds for the restriction of $\pi$ to $G_{i}^{\tau}/\Gamma^4$.

Since $\Gamma$ acts freely and properly on the manifold $G$ (on the right) and the quotient $G/\Gamma$ is compact, the quotient maps $\rho_i\colon G_i\to G_i/\Gamma$ are covering maps. Therefore, for any point $z_i\in G_i/\Gamma$, there exists a neighborhood $V_{z_i}\subseteq G_i/\Gamma$ and a continuous function $f\colon V_{z_i}\to G_i$ such that $\rho_i\circ f$ is the identity map on $V_{z_i}$.

Now, consider a point $\pi(z)\in \overline{\pi(G_{i}^{\tau}/\Gamma^{4})}$, with $z=(z_1,z_2,z_3,z_4)$. Note that the first $i$ coordinates of $\pi(z)$ are $\mr{Id}_G\Gamma$. Consider the $(i+1)$-st coordinate, $z_{i+1}\in G_i/\Gamma$, of $\pi(z)$, and let $x = (x_1,x_2,x_3,x_4)\in G_{i}^{\tau}/\Gamma^{4}$ be such that $x_{i+1}\in V_{z_{i+1}}$ (with $V_{z_{i+1}}$ defined as in the previous paragraph). This implies that $(\rho_{i+1}\circ f)(x_{i+1}) = x_{i+1}$, i.e., $x_{i+1} = f(x_{i+1})\Gamma$. Define
\[F_i(x_{i+1}) = 
\begin{cases} 
(f(x_1), f(x_1), f(x_1), f(x_1)) &\text{if } i = 0 \\
(\mr{Id}_G, f(x_2), f(x_2)^{-2}, f(x_2)^{4}) &\text{if } i = 1 \\
(\mr{Id}_G, \mr{Id}_G, f(x_3), f(x_3)^{2}) &\text{if } i = 2 \\
\end{cases} \]
for all such $x$. We write $F=F_i$ as shorthand. Note that $F(x_{i+1})$ is continuous as a function of $x_{i+1}$, and hence of $\pi(x)$ (as $0\le i\le 2$), which means that $F$ defines a continuous function in an open neighborhood of $\pi(z)$. By definition, if $x\in G_i^{\tau}/\Gamma^{4}$, then there exists $g\in G_i^{\tau}$ such that $g\Gamma^{4} = (g_1,g_2,g_3,g_4)\Gamma^{4}=x$ (which we choose arbitrarily). Observe that $f(x_{i+1})^{-1}g_{i+1}$ is an element of both $G_{i}$ and $\Gamma$. Now, let 
\[\wt{g}_i = 
\begin{cases} 
(f(x_1)^{-1}g_1, f(x_1)^{-1}g_1, f(x_1)^{-1}g_1, f(x_1)^{-1}g_1) &\text{if } i = 0 \\
(\mr{Id}_G, f(x_2)^{-1}g_2, (f(x_2)^{-1}g_2)^{-2}, (f(x_2)^{-1}g_2)^{4}) &\text{if } i = 1 \\
(\mr{Id}_G, \mr{Id}_G, f(x_3)^{-1}g_3, (f(x_3)^{-1}g_3)^{2}) &\text{if } i = 2
\end{cases} \]
if $x$ is in a sufficiently small open neighborhood of $z$. Again let $\wt{g}=\wt{g}_i$ as shorthand. Observe that $\wt{g}\in G_{i}^{\tau}$ and $\wt{g}\in \Gamma^{4}$ by construction. Define $h$ to be such that 
\[g = F(x_{i+1})h\wt{g}.\]
The $(i+1)$-st coordinate of $h$ is the identity (as are the first $i$ coordinates). So, $h$ must lie in $G_{i+1}^{\tau}$. Therefore, 
\[x = g\Gamma^{4} = F(x_{i+1})h\wt{g}\Gamma^{4} = F(x_{i+1})h\Gamma^{4},\]
since $\wt{g}\in \Gamma^{4}$. Thus, 
\[F(x_{i+1})^{-1}x = h\Gamma^{4}.\]
Note that, as $F(x_{i+1})$ depends continuously on $\pi(x)$ in a neighborhood of $\pi(z)$, and is defined via a local continuous right-inverse, we have that $\pi(F(x_{i+1})^{-1}x)$ is within a neighborhood of $\pi(F(z_{i+1})^{-1}z)$. Furthermore, note that, as $h\in G_{i+1}^{\tau}$, the $(i+1)$-st coordinate of $F(x_{i+1})^{-1}x$ is $\mr{Id}_G\Gamma$, and therefore we are in position to apply induction. By induction, we may write 
\begin{equation}\label{eq:inductive-right-inverse}
F(x_{i+1})^{-1}x = h\Gamma^{4} = \pi_{(F(z_{i+1})^{-1}z)}^{-1}(\pi(F(x_{i+1})^{-1}x)),
\end{equation}
and therefore 
\[x = F(x_{i+1})\pi_{F(z_{i+1})^{-1}z}^{-1}(\pi(F(x_{i+1})^{-1}x)),\]
where $\pi_{F(z_{i+1})^{-1}z}^{-1}$ is the (localized) continuous right-inverse we have constructed for $G_{i+1}^{\tau}/\Gamma^{4}$. Note that $\pi(F(x_{i+1})x) = \wt{\pi}(F(x_{i+1}))\pi(x)$, where $\wt{\pi}$ is the projection onto the first three coordinates in $G^{4}$. By the previous discussion, the right-hand-side of \eqref{eq:inductive-right-inverse} depends continuously on $\pi(x)$ and is defined in a sufficiently small neighborhood of $\pi(x)$. Thus, the right-hand-side of \eqref{eq:inductive-right-inverse} provides the desired continuous right-inverse and the result follows.

We now glue these local right-inverses into a global continuous right-inverse $\Pi\colon\Sigma\to G^\tau\Gamma^4$ satisfying $(\Pi\circ\pi)(x)=x$ for all $x\in G^\tau\Gamma^4$. We can perform such gluing as long as all our local right-inverses agree on intersections. To see this, it suffices to show that $\pi$ is injective on $G^\tau\Gamma^4$. Suppose $\pi(x)=\pi(y)$ for $x,y\in G^\tau\Gamma^4$. We can find $g\in G^\tau$ such that $g^{-1}y\in\Gamma^4$, so $\pi(g^{-1}x)=(\mr{Id}_G\Gamma,\mr{Id}_G\Gamma,\mr{Id}_G\Gamma)$ since the right-action of $G$ on $G/\Gamma$ is compatible with $\pi$. Now $g^{-1}x\in G^\tau\Gamma^4$ and has first three coordinates being $\mr{Id}_G\Gamma$. This implies that if we write $g^{-1}x=(g_0\Gamma,g_0g_1\Gamma,g_0g_1^{-2}g_2\Gamma,g_0g_1^4g_2^2\Gamma)$, then $g_0\in\Gamma$ hence $g_1\in\Gamma$ hence $g_2\in\Gamma$ (since $\Gamma$ is a subgroup of $G$). Thus the final coordinate of $g^{-1}x$ is also $\mr{Id}_G\Gamma$, and hence $g^{-1}x,g^{-1}y\in\Gamma^4$. This implies $x=y$ as cosets, completing the proof of injectivity (and hence existence of a global inverse).

We now define $Q$ to be the fourth coordinate of this global right-inverse of $\pi$ on $G^\tau\Gamma^4$. By the above arguments, the first two bullet points are satisfied.

We finally briefly sketch how to obtain the necessary Lipschitz bound on $Q$. First, note from above that $Q$ is unique and, as $(G/\Gamma)^{4}$ has diameter bounded by $\on{poly}_m(L)$ \cite[Lemma~A.16]{GT12}, it suffices to consider points which are within distance $\on{poly}_m(L^{-1})$ of each other to prove Lipschitz bounds on $Q$. Furthermore, looking at our inductive construction, it suffices to show we can invert $\rho_i$ for $x'\in G_i/\Gamma$ such that the preimage in $G_i$ is suitably bounded and in a Lipschitz manner. The remainder of the analysis then consists of left multiplication by bounded group elements, which is Lipschitz by \cite[Lemma~A.5]{GT12} for left-multiplication and right-multiplication is always Lipschitz due to right-invariance of the metric on $G$. Note here the fact that if $g\in G$ is bounded, then $g^k$ for bounded $k$ and $g^{-1}$ are as well since $d(g^k,\mr{Id}_G)\le \sum_{i=1}^{k}d(g^{i},g^{i-1}) = kd(g,\mr{Id}_G)$ and $d(g,\mr{Id}_G) = d(g^{-1},\mr{Id}_G)$.

To invert $\rho_i$ in the neighborhood of a point $x'\in G_i/\Gamma$, first note that $G_i$ is a closed rational subgroup of $G$ and the last $\dim(G_i)$ elements of $\mathcal{X}$ are a valid Mal'cev basis for $G_i$. Therefore, by combining \cite[Lemmas~A.16 and A.17]{GT12}, there exists $g'\in G_i$ such that $g'\Gamma = x'$ and $d(g',\mr{Id}_G)\le \poly_m(L)$. Taking a sufficiently small neighborhood around $g'$, of size $\poly_m(L^{-1})$, for any points $g^{(1)}$ and $g^{(2)}$ in this neighborhood, we have
\begin{align*}
\inf_{\gamma \in \Gamma\setminus\{0\}}d(g^{(1)},g^{(2)}\gamma)&\ge \poly_m(L^{-1})\cdot \inf_{\gamma \in \Gamma\setminus\{0\}}d((g^{(2)})^{-1}g^{(1)},\gamma)\\
&\ge \poly_m(L^{-1})\cdot \bigg(\inf_{\gamma \in \Gamma\setminus\{0\}}d(\mr{Id}_G,\gamma) - d((g^{(2)})^{-1}g^{(1)},\mr{Id}_G)\bigg)\\
&\ge \poly_m(L^{-1}).
\end{align*}
The last inequality comes from the fact that $\psi(\gamma)\in \mathbf{Z}^m$, where $\psi$ are Mal'cev coordinates of the second kind (with respect to an implicit Mal'cev basis $\mc{X}$ giving the complexity bound). As $\psi(\gamma)$ is nonzero, \cite[Lemma~A.4]{GT12} then gives the lower bound.

Thus, in a small neighborhood of $g'$, we have that $d_G(g^{(1)},g^{(2)}) = d_{G/\Gamma}(g^{(1)}\Gamma,g^{(2)}\Gamma)$. Furthermore, the pushforward under $\rho_i$ of the neighborhood of $g'$ in $G_i$ surjects onto a small neighborhood of $x' = g'\Gamma$ in $G_i/\Gamma$. Therefore, given $z'\in G_i/\Gamma$ near $x'$, the map $f(z')$ can be defined by taking the closest point to $g_z'$ to $g'$ in $G_i$ such that $g_z'\Gamma = z'$. This gives the desired inverse map in the neighborhood of $x'$ which is Lipschitz by the above equality of metrics and, furthermore, we have that the inverse image of $x'$ in $G_i$ is $\poly_m(L)$-bounded, as desired. 
\end{proof}

We are now in position to prove \cref{lem:dual-expansion}.
\begin{proof}[Proof sketch of \cref{lem:dual-expansion}]
Let the filtration $G_{\bullet}$ be denoted by $G = G_0 = G_1 \geqslant G_2 \geqslant\mr{Id}_{G}$. Let $\tau^{[i]}$ denote the span in $\mathbf{R}^4$ of the set of vectors 
\begin{equation*}
 \left\{(\tau_1(x,y)^i,\tau_2(x,y)^i,\tau_3(x,y)^i,\tau_4(x,y)^i)\colon x,y\in\mathbf{Z}\right\}.
\end{equation*}
We find that 
\begin{align*}
\tau^{[1]} &= \mathbf{R}(1,1,1,1) \oplus \mathbf{R}(0,1,-2,4)\\
\tau^{[2]} &= \mathbf{R}(1,1,1,1) \oplus \mathbf{R}(0,1,-2,4) \oplus \mathbf{R}(0,0,1,2)\\
\tau^{[3]} &= \mathbf{R}(1,1,1,1) \oplus \mathbf{R}(0,1,-2,4) \oplus \mathbf{R}(0,0,1,2) \oplus \mathbf{R}(0,0,0,1)=\mathbf{R}^4.
\end{align*}
Therefore $\tau$ satisfies the \emph{flag condition} (which also follows from the fact that $\tau(x,y)$ is translation-invariant) and by \cite[Lemma~3.2]{GT10b} we have that $g(\tau(x,y))$ takes values within $G^{\tau}$ (abusively extending $g$ to vectors coordinate-wise).

Furthermore, by \cref{lem:constraint}, for $(g_1,g_2,g_3,g_4)\in G^\tau$ we have 
\[F(g_4\Gamma) = F(Q(g_1\Gamma,g_2\Gamma,g_3\Gamma)),\]
with $Q$ as in~\cref{lem:constraint}. Using the partition of unity argument suggested in \cite[Footnote~10]{TT21} and the quantitative bounds on $Q$ proven in \cref{lem:constraint}, we have that for $(g_1,g_2,g_3,g_4)\in G^{\tau}$,
\[F(g_4\Gamma) = \sum_{\alpha\in A}\prod_{j\in [3]} F_{j,\alpha}(g_j\Gamma) + G((g_1,g_2,g_3,g_4)\Gamma),\]
where
\begin{itemize}
    \item $\snorm{G((g_1,g_2,g_3,g_4)\Gamma)}_{\infty}\le L^{-1}/2$ for all $(g_1,g_2,g_3,g_4)\in G^{\tau}$;
    \item there are $\on{poly}_{m}(L)$ terms in the sum over $\alpha$; and
    \item the functions $F_{j,\alpha}(g_j\Gamma)$ are $\on{poly}_{m}(L)$-Lipschitz.
\end{itemize}
Note here that a qualitative version follows simply by applying the Stone--Weierstrass theorem (and noting that $(G/\Gamma)^3$ is compact).

This procedure, however, does not immediately yield that the $F_{j,\alpha}$'s have the desired vertical frequencies. Applying \cref{lem:vertical-expansion} (vertical expansion), we have may assume that the $F_{j,\alpha}$'s each have vertical frequencies $\xi_{j,\alpha}$ bounded by $\on{poly}_{m}(L)$. The crucial idea at this point (due to Leng \cite[Lemma~A.3]{Len22}) is noting that for $(g_1,g_2,g_3,g_4)\in G^{\tau}$ and $g'\in G_2$, we have
\[(g_1g',g_2g',g_3g',g_4g')\in G^{\tau},~ (g_1,g_2g',g_3g'^{-2},g_4g'^{4})\in G^{\tau}, \text{ and }(g_1,g_2,g_3g',g_4g'^{2})\in G^{\tau}.\]
Thus, if $g_1',g_2',g_3'\in G_2$, then 
$\wt{g}=(g_1g_1',g_2g_1'g_2',g_3g_1'g_2'^{-2}g_3',g_4g_1'g_2'^{4}g_3'^{2})\in G^{\tau}$, and
\begin{align*}
F(g_4\Gamma) &= e(-\xi(g_1'))e(-4\xi(g_2'))e(-2\xi(g_3'))F(g_4g_1'g_2'^{4}g_3'^{2}\Gamma)\\
&= e(-\xi(g_1'))e(-4\xi(g_2'))e(-2\xi(g_3'))\sum_{\alpha}\prod_{j\in [3]} F_{j,\alpha}(\wt{g}_j\Gamma) + \wt{G}((g_1,g_2,g_3,g_4)\Gamma,g_1',g_2',g_3').
\end{align*}
We now integrate over each $g_i'\in G_2/(\Gamma\cap G_2)$ (this is well-defined because $G_2/(\Gamma\cap G_2)$ is a torus onto which $e(\cdot)$ descends). Note that the integral of a nontrivial character $\xi$ over $G_2/(\Gamma\cap G_2)$ is zero, and therefore a term $\alpha$ only remains if the vertical frequencies solve the following system of linear equations:
\begin{align*}
0 &= -\xi + \xi_{1,\alpha} + \xi_{2,\alpha} + \xi_{3,\alpha},\\
0 &= -4\xi + \xi_{2,\alpha} -2\xi_{3,\alpha},\\
0 &= -2\xi + \xi_{3,\alpha},
\end{align*}
using the formulas for $\wt{g}_j$ and the vertical frequencies of the $F_{j,\alpha}$. The unique solution is $\xi_{1,\alpha} = -9\xi$, $\xi_{2,\alpha} = 8\xi$, and $\xi_{3,\alpha} = 2\xi$. Thus, after performing this integration, we find
\[F(g_4\Gamma)=\sum_{\alpha\in A^\ast}\prod_{j\in[3]}F_{j,\alpha}(g_j\Gamma)+\int_{(G_2/(\Gamma\cap G_2))^3}\wt{G}((g_1,g_2,g_3,g_4)\Gamma,g_1',g_2',g_3')~dg_1'dg_2'dg_3'.\]
where all $\alpha\in A^\ast$ are such that $F_{j,\alpha}$ has vertical frequencies $-9\xi,8\xi,2\xi$ for $j=1,2,3$ respectively. This is valid for all $(g_1,g_2,g_3,g_4)\in G^\tau$, hence it applies to $g(\tau(x,y))\in G^\tau$ and we have the desired expression.
\end{proof}

\subsection{Factorization result}\label{sub:factor}
The next lemma serves as the crucial analogue of Weyl's inequality for degree $2$ nilsequences. Although the statement is motivated by work of Leng \cite[Lemma~6.1]{Len22}, our proof mimics the factorization of polynomial sequences on nilmanifolds due to Green and Tao \cite[Theorem~1.19]{GT12}. However, our analogue of the basic decomposition result \cite[Proposition~9.2]{GT12} assumes that the polynomial sequence $g(P(6y))$ is not equidistributed, instead of the sequence $g(y)$. The crucial point, analogous to the case of polynomial phases sketched in \cref{sec:sketch}, is that one can still deduce a useful factorization of $g(y)$ from this.  

The key input into our argument is the following result on equidistribution of polynomial orbits in nilmanifolds due to Green and Tao \cite[Theorem~2.9]{GT12} with the explicit dimension dependencies given in work of Tao and Ter{\"a}v{\"a}inen \cite{TT21}.

\begin{theorem}[{\cite[Theorem~A.3]{TT21}}]\label{thm:equi-nilmanifold}
Let $m\ge 0$, $\delta\in (0,1/2)$, and $N\ge 1$. Let $G/\Gamma$ be a filtered nilmanifold of degree $d$ with complexity at most $1/\delta$. Let $g\colon\mathbf{Z}\to G$ be a polynomial sequence. If $(g(n)\Gamma)_{n\in [N]}$ is not $\delta$-equidistributed (\cref{def:equidistributed}), then there exists a horizontal character $0<|\eta|\le \delta^{-\exp((2m)^{O_d(1)})}$ such that 
\[\snorm{\eta\circ g}_{C^{\infty}[N]}\le \delta^{-\exp((2m)^{O_d(1)})},\]
where the implicit constant $O_d(1)$ only depends on $d$. 
\end{theorem}

Recall the $C^\infty[N]$-norm from \cref{def:smooth}. We now state our analogue of \cite[Proposition~9.2]{GT12}.

\begin{proposition}\label{prop:Malcev-bash}
Fix $\delta \in (0,1/2)$ and $P$ and $W$ as in \eqref{eq:Wdef}. Let $G/\Gamma$ be an $m$-dimensional filtered nilmanifold of degree $2$ and complexity $L$ with filtration $G = G_0 = G_1\geqslant G_2\geqslant \mr{Id}_{G}$ denoted by $G_\bullet$. Furthermore, let $\mc{X}$ denote the Mal'cev basis of $G/\Gamma$ and let $F\colon G/\Gamma\to\mathbf{C}$ be such that $\snorm{F}_{\mr{Lip}}\le L$ and $F$ has a nonzero vertical frequency $2\cdot \xi$ such that $\snorm{\xi}_{\infty}\le L$. Let $g\colon\mathbf{Z}\to\mathbf{G}$ be a polynomial sequence with respect to $G_{\bullet}$ with $g(0) = \mr{Id}_{G}$. For all $r\in [W]$, define 
\[P_r(y) := \frac{P(Wy + r) - P(r)}{W}.\]

Let $\mc{I}\subseteq[\pm \delta^{-1}N^{1/2}W^{-1}]$ be an arithmetic progression of difference at most $\delta^{-1}$. Suppose that $W\le N^{1/10^4}$, $N\ge \poly_m(\delta^{-1}L)$, and 
\[\bigg|\sum_{y\in \mc{I}}F(g(6P_r(y))\Gamma)\bigg|\ge \delta N^{1/2}W^{-1}.\]

Then, there exists a factorization $g = \eps g' \gamma$ with polynomial sequences $\eps,g',\gamma\colon\mathbf{Z}\to G$ such that
\begin{itemize}
    \item for all $n\in [\pm \delta^{-1} N]$, $d(\eps(n),\eps(n-1))\le \on{poly}_m(L\delta^{-1})/N$ and $d(\eps(n),\mr{Id}_G)\le \on{poly}_m(L\delta^{-1})$;
    \item $\gamma$ is $\on{poly}_m(L\delta^{-1})$-rational and $\gamma(n)\Gamma$ is periodic with period at most $\on{poly}_m(L\delta^{-1})$; and
    \item $g'$ takes values only in $G'$, a simply connected proper $\on{poly}_m(L\delta^{-1})$-rational subgroup with respect to $\mc{X}$, and may be viewed as a polynomial sequence with respect to the filtration $G_{\bullet}'$ where $G_i = G' \cap G_i$.
\end{itemize}
\end{proposition}
Here the condition $N\ge\on{poly}_m(\delta^{-1}L)$ is used abusively to express that there is some such polynomial expression such that this condition on $N$ is sufficient; we use similar conventions later without comment. Now, we first state the following explicit binomial coefficient identities. While the precise constant coefficients are unimportant, various powers of $A$ and $B$ will indeed be used in our analysis. 
\begin{claim}\label{clm:identity}
We have 
\begin{align*}
\binom{An^2 + Bn}{1} &= 2A\binom{n}{2} + (A+B)\binom{n}{1}
\end{align*}
and
\begin{align*}
\binom{An^2 + Bn}{2} &= 12A^2\binom{n}{4} + (18A^2 + 6 AB)\binom{n}{3} + (7A^2 + 6AB - A + B^2)\binom{n}{2} \\
&+ \bigg(AB + \binom{A}{2} + \binom{B}{2} \bigg)\binom{n}{1}.
\end{align*}
\end{claim}

We also require the following claim regarding polynomial sequences and the $C^{\infty}[N]$-norm.
\begin{claim}\label{clm:shifting}
Fix a constant $C\ge 1$. Suppose that $S$ is a nonzero integer such that $|S|\le C$ and $I$ is an integer such that $|I|\le CN$. If $p$ is a polynomial of degree at most $d$, there exists $S' \le C^{O_d(1)}$ such that 
\[\snorm{S'p(x)}_{C^{\infty}[N]}\le C^{O_d(1)}\snorm{p(Sx+I)}_{C^{\infty}[N]}.\]
\end{claim}
\begin{proof}
Let $q(x) = p(Sx + I)$ and $I'\in [S]$ be such that $I' \equiv I \mod S$. Note that 
\[\snorm{p(Sx+I')}_{C^{\infty}[N]} = \snorm{q(x + (I'-I)/S)}_{C^{\infty}[N]}.\]
Vandermonde's identity implies
\[\binom{n+I}{j} = \sum_{0\le t\le j}\binom{n}{t}\binom{I}{j-t}.\]
As $|(I-I')/S|\le CN$, using Vandermonde's identity we have by expansion that 
\[\snorm{q(x + (I-I')/S)}_{C^{\infty}[N]}\le C^{O_d(1)}\snorm{q(x)}_{C^{\infty}[N]}.\]
Putting it together, we have
\[\snorm{p(Sx+I')}_{C^\infty[N]}\le C^{O_d(1)}\snorm{p(Sx+I)}_{C^\infty[N]}.\]
Finally, by \cite[Lemma~8.4]{GT12} (applicable since the heights of $S,I'$ are bounded by $C$) we can find appropriate $S'$ so that $\snorm{S'p(x)}_{C^\infty[N]}\le C^{O_d(1)}\snorm{p(Sx+I')}$. This completes the proof.
\end{proof}

\begin{proof}[Proof of \cref{prop:Malcev-bash}]
Let $\min(\mc{I}) = I$, $S$ denote the difference of the progression $\mc{I}$, and $T$ the length of $\mc{I}$. By assumption, we have that
\[\bigg|\sum_{y\in[T]}F(g(6P_r(Sy + I))\Gamma)\bigg|\ge \delta N^{1/2}W^{-1}.\]
Next, as $2\cdot\xi$ is a nonzero vertical frequency for $F$, we have
\[\int_{y\in G/\Gamma}F(y)~dy = 0.\] 
Therefore, by definition we see the polynomial sequence $g(6P_r(Sy + I))$ is not $3\delta^2L^{-1}$-equidistributed. (Notice that $g(6P_r(Sy + I))$ is a polynomial sequence with respect to the filtration $\wt{G}_{\bullet}$ defined by $G = \wt{G}_0 = \wt{G}_1 = \wt{G}_2 = \wt{G}_3 = \wt{G}_4 \geqslant \wt{G}_5 = G_2 \geqslant \mr{Id}_{G}$.)

Let $\psi(g)$ denote the Mal'cev coordinates of $g\in G$ with respect to $\mc{X}$. By the classification of polynomial sequences in terms of Mal'cev coordinates \cite[Lemma~6.7]{GT10b} and the assumption that $g(0) = \mr{Id}_G$, we have
\[\psi(g(n)) = \binom{n}{2}t_2 + \binom{n}{1}t_1,\]
where $t_i\in \mathbf{R}^m$ and the first $m-\dim(G_2)$ coordinates of $t_2$ are zero. As $g(6P_r(Sy + I))$ is not $3\delta^2L^{-1}$-equidistributed, by \cref{thm:equi-nilmanifold} there exists a nonzero horizontal character $\eta$ such that 
\[\snorm{(\eta \circ g)(6P_r(Sy + I))}_{C^{\infty}[N^{1/2}W^{-1}]}\le \poly_m(\delta^{-1}L).\] 
The implied constants in $\poly_m(\cdot)$ are absolute, as the degree of the filtration under consideration is always bounded by $5$. By \cref{clm:shifting} there exists a positive integer $Q\le\delta^{-O(1)}$ such that 
\begin{equation}\label{eq:horiz-Pr-small}
\snorm{Q(\eta \circ g)(6P_r(y))}_{C^{\infty}[N^{1/2}W^{-1}]}\le \poly_m(\delta^{-1}L).
\end{equation}

By a direct computation, we have 
\[6P_r(y) = Ay^2 + By,\]
where 
\begin{align*}
A = 6W^2\quad\text{ and }\quad B= 6(2Wr+1).
\end{align*}

Now let the horizontal character $\eta$ be represented by $k\in\mathbf{Z}^m$ in Mal'cev coordinates. Thus $(\eta\circ g)(n)=k\cdot(\binom{n}{2}t_2+\binom{n}{1}t_1)$. Plugging into \eqref{eq:horiz-Pr-small} and using \cref{clm:identity}, and unwrapping the definition of the $C^\infty[N^{1/2}W^{-1}]$-norm, we can initially deduce that 
\begin{equation*}
\snorm{12QA^2 \cdot (t_2\cdot k)}_{\mathbf{T}}\le\frac{\poly_m(\delta^{-1}L)W^{4}}{N^2},\qquad\snorm{(18A^2 + 6AB)Q \cdot (t_2\cdot k)}_{\mathbf{T}}\le\frac{\poly_m(\delta^{-1}L)W^{3}}{N^{3/2}}.
\end{equation*}
Therefore, there exists a positive integer $Q_1\le\delta^{-O(1)}$ such that 
\begin{equation}\label{eq:W4}
\snorm{Q_1W^4 \cdot (t_2\cdot k)}_{\mathbf{T}}\le\frac{\poly_m(\delta^{-1}L)W^4}{N^2},\qquad\snorm{Q_1(3W^4 + W^2(2Wr + 1)) \cdot (t_2\cdot k)}_{\mathbf{T}}\le\frac{\poly_m(\delta^{-1}L)W^3}{N^{3/2}}.
\end{equation}
Combining these bounds yields
\[\snorm{Q_1W^2(2Wr + 1) \cdot (t_2\cdot k)}_{\mathbf{T}}\le\frac{\poly_m(\delta^{-1}L)W^{3}}{N^{3/2}}.\]
We now, crucially, use that $\gcd(W,2Wr+1) = 1$. Note that
\[Q_1 (t_2\cdot k) = \frac{T_1}{W^4} + E_1 = \frac{T_2}{W^2(2Wr+1)} + E_2\]
with $|E_1|,|E_2|\le N^{-1}$ (say) and $T_1,T_2\in \mathbf{Z}$. However, 
\[\bigg|\frac{T_1}{W^4} - \frac{T_2}{W^2(2Wr+1)}\bigg|\ge N^{-1/2}\]
unless $T_1\cdot W^2(2Wr+1) - T_2\cdot(W^4) = 0$. It follows that $W^2\mid T_1$ and, therefore,
\[\snorm{Q_1W^2 \cdot (t_2\cdot k)}_{\mathbf{T}}\le\frac{\poly_m(\delta^{-1}L)W^{2}}{N^{2}},\]
using the first bound in~\eqref{eq:W4}. Noting that $W^2\mid A$ and using that $|B|\ll W^2$, we have
\begin{equation}\label{eq:horiz-t2-small}
\snorm{Q_1(7A^2 + 6AB - A) (t_2\cdot k)}_{\mathbf{T}}+\norm{2Q_1 \bigg(AB + \binom{A}{2}\bigg) (t_2\cdot k)}_{\mathbf{T}}\le  \frac{\poly_m(\delta^{-1}L)W^4}{N^2}.
\end{equation}
Now we use \eqref{eq:horiz-Pr-small} again but applied to the lower coefficients, and we appropriately cancel out the contributions from the terms in \eqref{eq:horiz-t2-small}. We find
\begin{equation*}
\snorm{2Q_1(B^2(t_2\cdot k) + 2A (t_1\cdot k))}_{\mathbf{T}}\le \frac{\poly_m(\delta^{-1}L)W^{2}}{N}
\end{equation*}
and
\begin{equation*}
\norm{2Q_1\bigg(\binom{B}{2}(t_2\cdot k) + (A+B) (t_1\cdot k)\bigg)}_{\mathbf{T}}\le \frac{\poly_m(\delta^{-1}L)W}{N^{1/2}}.
\end{equation*}
Multiplying the first equation by $A+B$ and the second equation by $2A$ and subtracting, we find that 
\begin{align*}
\snorm{2Q_1(B^2(A+B) - AB(B-1))(t_2\cdot k)}_{\mathbf{T}}&\le \frac{\poly_m(\delta^{-1}L)W^{3}}{N^{1/2}}.
\end{align*}
As $W^2 \mid A$, we find by similar argumentation that
\begin{align*}
\snorm{2Q_1 B^3(t_2\cdot k)}_{\mathbf{T}}&\le \frac{\poly_m(\delta^{-1}L)W^{3}}{N^{1/2}}.
\end{align*}
Again, crucially, $\gcd((2Wr+1)^3,W) = 1$. As $Q_1(t_2\cdot k)$ is near a fraction with denominator $W^2$, repeating fraction comparison arguments similar to above we find that for $Q_2=4Q_1\le \delta^{-O(1)}$ we have
\[\snorm{Q_2 (t_2\cdot k)}_{\mathbf{T}}\le\frac{\poly_m(\delta^{-1}L)}{N^2}.\]
We may substitute this bound into earlier equations, and using the size bounds on $B$ deduce that 
\begin{equation*}
\snorm{Q_2 A (t_1\cdot k)}_{\mathbf{T}}\le \frac{\on{poly}_m(\delta^{-1}L)W^2}{N},\qquad\snorm{Q_2 (A+B) (t_1\cdot k)}_{\mathbf{T}}\le \frac{\on{poly}_m(\delta^{-1}L)W}{N^{1/2}}.
\end{equation*}
As $\gcd(A,A+B) = \gcd(A,B) = 6$, another fraction comparison argument shows 
\begin{align*}
\snorm{6Q_2 (t_1\cdot k)}_{\mathbf{T}}\le \frac{\on{poly}_m(\delta^{-1}L)}{N}. 
\end{align*}
Thus for $Q_3=6Q_2\le \delta^{-O(1)}$ we have
\begin{equation*}
\snorm{Q_3 (t_2\cdot k)}_{\mathbf{T}}\le \frac{\on{poly}_m(\delta^{-1}L)}{N^2},\qquad\snorm{Q_3 (t_1\cdot k)}_{\mathbf{T}}\le \frac{\on{poly}_m(\delta^{-1}L)}{N}.
\end{equation*}
Note that the quality of major arc control here is comparable to a situation where we knew that $g(y)$ itself were poorly equidistributed on $[N]$.

The remaining proof is now essentially identical to the argument in \cite[Proposition~9.2]{GT12}, as we are in the same essential position. We will define $G'$ to be the connected component of $\on{ker}(\eta)$ (as a subgroup of $G$) and, due to the size bounds on $\eta$, we have that $G'$ is a $\on{poly}_m(\delta^{-1}L)$-rational subgroup. It is seen to be simply connected by considering the Mal'cev coordinate representation for $\eta$.

We choose vectors $u_1,u_2\in\mathbf{R}^m$ such that $\snorm{t_j-u_j}_{\infty}\le \on{poly}_m(\delta^{-1}L)N^{-j}$ for $j=1,2$, such that $Q_3(u_1\cdot k)$ and $Q_3(u_2\cdot k)$ are integers, and such that the first $m-\dim(G_2)$ coordinates of $u_2$ are zero. We then choose vectors $v_1$ and $v_2$ with coordinates rationals with denominator bounded by $\on{poly}_m(\delta^{-1}L)$ and such that $k\cdot u_j = k\cdot v_j$ for $j=1,2$. 

Let $\eps$ and $\gamma$ be the polynomial sequences $\mathbf{Z}\to G$ for which
\begin{equation*}
\psi(\eps(n)) = \binom{n}{2}(t_2-u_2) + \binom{n}{1}(t_1-u_1)
\end{equation*}
and
\begin{equation*}
\psi(\gamma(n)) = \binom{n}{2}v_2 + \binom{n}{1}v_1,
\end{equation*}
and set 
\begin{equation*}
g' = \eps^{-1}g\gamma^{-1}.    
\end{equation*}
By construction, $g'$ takes values in $G'$ since $\eta$ is a horizontal character. We have that $\gamma$ is rational, as the denominators of $v_i$ are $\on{poly}_m(\delta^{-1}L)$-bounded, and therefore by \cite[Lemma~A.11(iv),~A.12(ii)]{GT12} we have that $\gamma(\cdot)$ is $\on{poly}_m(\delta^{-1}L)$-rational and periodic of period at most $\on{poly}_m(\delta^{-1}L)$. The claimed smoothness bounds for $\eps$ follow using that $\snorm{t_j-u_j}_{\infty}\le \on{poly}_m(\delta^{-1}L)N^{-j}$ and \cite[Lemma~A.4]{GT12}, which converts between distance in the metric $d_{\mc{X}}$ and differences in Mal'cev coordinates. This completes the proof. 
\end{proof}

Note that subgroup $G'$ obtained from \cref{prop:Malcev-bash} is not dependent on the vertical character $\xi$ in any manner; we only needed that the mean of $F$ on $G/\Gamma$ is $0$. However, we may iterate \cref{prop:Malcev-bash} until $\xi$ is trivial on $G_2'=G_2\cap G'$.

\begin{lemma}\label{lem:poly-sequence}
Fix $\delta \in (0,1/2)$ and $P$ and $W$ as in \eqref{eq:Wdef}. Let $G/\Gamma$ be an $m$-dimensional filtered nilmanifold of degree $2$ and complexity $L$. Furthermore, let $\mc{X}$ denote the Mal'cev basis of $G/\Gamma$ and let $F\colon G/\Gamma\to\mathbf{C}$ be such that $\snorm{F}_{\mr{Lip}}\le L$ and $F$ has vertical frequency $2\cdot \xi$ such $\snorm{\xi}_{\infty}\le L$. Let $g\colon\mathbf{Z}\to\mathbf{G}$ be a polynomial sequence with respect to the filtration $G = G_0 = G_1\geqslant G_2\geqslant \mr{Id}_{G}$, denoted by $G_{\bullet}$, and $g(0) = \mr{Id}_{G}$. Finally, for $r\in [W]$, define 
\[P_r(y) = \frac{P(Wy + r) - P(r)}{W}.\]

Suppose that $W\le N^{1/10^{4}}$, $N\ge \on{poly}_m(\delta^{-1}L)$, and that 
\[\bigg|\sum_{y\in[\pm T]} F(g(6P_{r}(y))\Gamma)\bigg| \ge \delta N^{1/2}W^{-1}\]
for some $T\in [\delta^{-1} N^{1/2}W^{-1}]$. Then there exists a factorization $g=\eps g'\gamma$ and subgroup $G'$ with polynomial sequences $\eps,g',\gamma\colon\mathbf{Z}\to G$ such that
\begin{itemize}
    \item for all $n\in [\pm \delta^{-2} N]$, $d(\eps(n),\eps(n-1))\le \on{poly}_m(L\eps^{-1})/N$ and $d(\eps(n),\mr{Id}_G)\le \on{poly}_m(L\delta^{-1})$;
    \item $\gamma$ is $\on{poly}_m(L\delta^{-1})$-rational and $\gamma(n)\Gamma$ is periodic with period at most $\on{poly}_m(L\delta^{-1})$;
    \item $g'$ takes values in a connected $\on{poly}_m(L\delta^{-1})$-rational subgroup $G'$ and is a polynomial sequence with respect to the filtration $G_{\bullet}'$, where $G_j' = G_j \cap G'$; and
    \item $\xi$ is trivial on $G_2'=G_2\cap G'$.
\end{itemize}
\end{lemma}
\begin{proof}
We first handle the trivial case where $\xi=0$. This case is dispatched via setting $G' = G$, $g' = g$, and $\eps$ and $\gamma$ to both be identically $\mr{Id}_G$. 

Otherwise, we iteratively define a sequence of parameters $(\delta_i)$ with $\delta_1^{-1} = \delta^{-1} L$ and $\delta_{i+1}^{-1} = \on{poly}_m(\delta_i^{-1}L)$ and a sequence of rational connected subgroups $(G^{(i)})$ with $G^{(1)} = G$ and $G^{(i)}$ being $\delta_i^{-1}$-rational with respect to $G$. We write $G_j^{(i)}=G_j\cap G^{(i)}$. At each stage, we have the factorization 
\[g = \eps_i g_i \gamma_i\]
with $g_i(0) = \mr{Id}_G$, $g_i$ taking values in $G^{(i)}$, $\eps_i$ satisfying, for $n\in [\pm \delta^{-2}N]$, that $d(\eps_i(n),\eps_i(n-1))\le \delta_i^{-1}N^{-1}$ and $d(\eps_i(n),\mr{Id}_G)\le \delta_i^{-1}$, and $\gamma_i$ being $\delta_i^{-1}$-rational and periodic with period at most $\delta_i^{-1}$. We let $\eps_1=\gamma_1=\mr{Id}_G$ and $g_1=g$ to start.

Now given $i$, we define the next factorization data. If $\xi$ is trivial on $G_2^{(i)}$ then we terminate, providing our desired final factorization. Else, decompose $[\pm T]$ into arithmetic progressions that are $N^{1/2}W^{-1}\on{poly}_m(\delta_iL^{-1})$ in length and with common difference divisible by the period of $\gamma_i$. Then, by the pigeonhole principle, there exists such a progression $Q$ for which 
\[\bigg|\sum_{y\in Q}F(g(6P_r(y))\Gamma)\bigg|\ge \on{poly}_m(\delta_iL^{-1})N^{1/2}W^{-1}.\]
By the smoothness of $\eps_i$, the rationality of $\gamma_i$, and the Lipschitz bound for $F$, there exist group elements $\eps_Q$ and $\gamma_Q$, each of size $\on{poly}_m(\delta_i^{-1}L)$, with $\gamma_Q$ being $\on{poly}_m(\delta_i^{-1}L)$-rational, such that 
\[\bigg|\sum_{y\in Q}F(\eps_Qg_i(6P_r(y))\gamma_Q\Gamma)\bigg|\ge \on{poly}_m(\delta_iL^{-1})N^{1/2}W^{-1}.\]
Note here that $\gamma_Q$ is essentially a ``representative'' for $\gamma_i$ in this modular class that is bounded, and not the value of $\gamma_i$ itself. Such a representative exists, as any group element can be made bounded by right-multiplying by an element of the $\Gamma$ \cite[Lemma~A.14]{GT12} and the product of two rational elements is rational with appropriate height bounds \cite[Lemma~A.11]{GT12}. 

Set $F_i(x) = F(\eps_Q\gamma_Qx)$. Note that $F_i$ is $\on{poly}_m(\delta_i^{-1}L)$-Lipschitz, as left-multiplication by bounded elements approximately preserves the metric \cite[Lemma~A.5]{GT12}. Furthermore, letting $g_i' = \gamma_Q^{-1}g_i\gamma_Q$, we have 
\[\bigg|\sum_{y\in Q}F_i(g_i'(6P_r(y))\Gamma)\bigg|\ge \on{poly}_m(\delta_iL^{-1})N^{1/2}W^{-1}.\]
Since $G^{(i)}$ is a $\on{poly}_m(\delta_i^{-1}L)$-rational subgroup of $G$, the conjugate subgroup $\gamma_Q^{-1}G^{(i)}\gamma_Q$ is similarly rational by \cite[Lemma~A.13]{GT12}. Furthermore, note that $\gamma_Q^{-1}G_2^{(i)}\gamma_Q = G_2^{(i)}$, as $G_2^{(i)}\subseteq G_2$ is in the center of $G$ because we have a degree $2$ filtration on $G$. Therefore, as $\xi$ is nonzero on $\gamma_Q^{-1}G_2^{(i)}\gamma_Q = G_2^{(i)}$, and since $G_2^{(i)}$ being simply connected implies that if $\xi$ is nonzero then $2\cdot \xi$ is nonzero, we can apply \cref{prop:Malcev-bash} to obtain
\[g_i' = \wt{\eps}_{i+1} g_{i+1} \wt{\gamma}_{i+1}\]
where $\wt{\gamma}_{i+1}$ is $\on{poly}_m(\delta_i^{-1}L)$-rational and periodic, $d(\wt{\eps}_{i+1}(n),\wt{\eps}_{i+1}(n-1))\le \on{poly}_m(\delta_i^{-1}L)N^{-1}$ and $d(\wt{\eps}_{i+1}(n),\mr{Id}_G)\le \on{poly}_m(\delta_i^{-1}L)$ for $n\in [\pm \delta^{-2}N]$, and $g_{i+1}$ lives in a subgroup $G^{(i+1)}$ that is $\on{poly}_m(\delta_i^{-1}L)$-rational with respect to $G^{(i)}$. Thus,
\[g_i = \gamma_Q \wt{\eps}_{i+1} g_{i+1} \wt{\gamma}_{i+1} \gamma_Q^{-1}\]
and, so, 
\[g = \eps_i \gamma_Q \wt{\eps}_{i+1} g_{i+1} \wt{\gamma}_{i+1} \gamma_Q^{-1}\gamma_i.\]
Taking $\eps_{i+1} = \eps_i \gamma_Q \wt{\eps}_{i+1}$ and $\gamma_{i+1} = \wt{\gamma}_{i+1} \gamma_Q^{-1}\gamma_i$ completes the iteration. In particular, $\gamma_Q \wt{\eps}_{i+1}$ is seen to be sufficiently smooth as left-multiplication by bounded elements approximately preserves distances \cite[Lemma~A.5]{GT12}, and $\eps_i$ is sufficiently smooth as the product of smooth sequences is sufficiently smooth by \cite[Lemma~10.1]{GT12}. The rationality claims for $\gamma_{i+1}$ follow immediately from \cite[Lemma~A.11,~A.12]{GT12}.

Note that at each step of the iteration we have $\delta_{i+1}^{-1} = \on{poly}_{m}(L\delta_{i}^{-1})$, where the implied constants in $\on{poly}_m(\cdot)$ are absolute. Note also that there are at most $m$ iterations, as each iteration decreases the dimension of $G^{(i)}$ (since the $G'$ produced by \cref{prop:Malcev-bash} is a connected proper subgroup), and therefore we obtain the desired result (up to slightly increasing the implicit constants in the underlying notation). 
\end{proof}

\section{Degree-lowering}\label{sec:deg-lower}
The main purpose of this section is to deduce the following key degree-lowering result. 
\begin{proposition}\label{prop:degree-lower}
Fix a positive integer $k\ge 3$, let $w$, $W$, and $P$ be as in \eqref{eq:Wdef}, and let $\delta\in(0,1/2)$. Let $f_1,f_2,f_3\colon\mathbf{Z}\to\mathbf{C}$ be $1$-bounded functions such that $\on{supp}(f_i)\subseteq [\pm \delta^{-1} N]$ for $i=1,2,3$. Suppose that 
\[\snorm{\mc{D}^{1}(f_2,f_3)}_{U^{k}_{W\cdot[N/W]}}^{2^k}\ge \delta N.\]
Furthermore, suppose that $N\ge W^{\Omega(1)}\cdot \exp(\exp(\delta^{-\Omega_k(1)}))$. Then,
\[\min_{i=2,3}\snorm{f_i}_{U^{k-1}_{W\cdot[N/W]}}^{2^{k-1}}\gg\exp(-\exp(\delta^{-O_k(1)}))\cdot N.\]
\end{proposition}

We will also require the following variant of the above result; the proof is identical, just replacing the polynomial $P(y) = Wy^2 + y$ with $-Wy^2-y$. 
\begin{proposition}\label{prop:degree-lower-2}
Fix a positive integer $k\ge 3$, and let $w$, $W$, and $P$ be as in \eqref{eq:Wdef}, and let $\delta\in (0,1/2)$. Let $f_1,f_2,f_3\colon\mathbf{Z}\to\mathbf{C}$ be $1$-bounded functions such that $\on{supp}(f_i)\subseteq [\delta^{-1} N]$ for $i=1,2,3$. Suppose that 
\[\snorm{\mc{D}^{3}(f_1,f_2)}_{U^{k}_{W\cdot[N/W]}}^{2^k}\ge \delta N.\]
Furthermore, suppose that $N\ge W^{\Omega(1)}\cdot \exp(\exp(\delta^{-\Omega_k(1)}))$. Then,
\[\min_{i=1,2}\snorm{f_i}_{U^{k-1}_{W\cdot[N/W]}}^{2^{k-1}}\gg\exp(-\exp(\delta^{-O_k(1)}))\cdot N.\]
\end{proposition}
\begin{remark*}
The methods in this paper do \emph{not} prove the analogous statement for $\mc{D}^{2}(f_3,f_1)$, as our methods do not prove the needed statement corresponding to \cref{lem:constraint}. By symmetry, the constraints required for $\mc{D}^{1}(f_2,f_3)$ and $\mc{D}^{3}(f_1,f_2)$ are identical.
\end{remark*}

\subsection{\texorpdfstring{$U^2$}{U2}-control for S\'{a}rk\"{o}zy-type configurations}
We first require $U^2$-control for S\'{a}rk\"{o}zy-type configurations. The proof we give is identical to that of Green \cite[Section~3]{Gre02}, modulo standard circle method computations that we place in \cref{sec:circle}. 

\begin{lemma}\label{lem:sarkozy}
There exists a constant $c = c_{\ref{lem:sarkozy}}>0$ such that the following holds. Let $W$ be as in \eqref{eq:Wdef} with $W\le N^{c}$ and let $f_i\colon\mathbf{Z}\to\mathbf{C}$ be $1$-bounded with $\on{supp}(f_i)\subseteq [\pm \delta^{-1} N]$. Define 
\[P_k(y) = \frac{P(Wy + k) - P(k)}{W}\]
for $k\in [W]$, and suppose that 
\[\bigg|\sum_{x\in \mathbf{Z}}\mathbf{E}_{y\in [\pm N^{1/2}W^{-1}]}f_1(x+P_k(y))f_2(x+2P_k(y))\bigg|\ge \delta N.\]
Then,
\[\min_{i\in \{1,2\}}\sup_{\Theta\in \mathbf{T}}|\wh{f_i}(\Theta)|\gg \delta^{O(1)}N.\]
\end{lemma}
\begin{proof}
We have 
\[\bigg|\sum_{x\in \mathbf{Z}}\sum_{y\in[\pm N^{1/2}W^{-1}]}f_1(x+P_k(y))f_2(x+2P_k(y))\bigg|\ge \delta N^{3/2}W^{-1}.\]
Let $F(t)$ denote the indicator of the set $\{P_k(y)\colon y\in [\pm N^{1/2}W^{-1}]\}$\footnote{Note that $P(y_1) = P(y_2)$ implies $(y_1-y_2)(W(y_1+y_2)+1) = 0$. Therefore, every element in the set occurs with multiplicity $1$.}, and thus we have
\[\bigg|\sum_{x\in \mathbf{Z}}\sum_{y\in \mathbf{Z}}f_1(x+t)f_2(x+2t)F(t)\bigg|\ge \delta N^{3/2}W^{-1}.\]
Applying Fourier inversion, this is equivalent to 
\[\bigg|\int_{\mathbf{T}}\wh{f_1}(\Theta)\wh{f_2}(-\Theta)\wh{F}(\Theta)~d\Theta\bigg|\ge \delta N^{3/2}W^{-1}.\]
We now prove the result for $i = 1$; the result for $i = 2$ is analogous. Note that
\begin{align*}
\delta N^{3/2}W^{-1} &\le \bigg|\int_{\mathbf{T}}\wh{f_1}(\Theta)\wh{f_2}(-\Theta)\wh{F}(\Theta)~d\Theta\bigg|\le \sup_{\Theta\in \mathbf{T}}|\wh{f_1}(\Theta)|^{1/3}\cdot\int_{\mathbf{T}}|\wh{f_1}(\Theta)|^{2/3}|\wh{f_2}(-\Theta)||\wh{F}(\Theta)|~d\Theta\\
&\le \sup_{\Theta\in \mathbf{T}}|\wh{f_1}(\Theta)|^{1/3}\cdot\bigg(\int_{\mathbf{T}}|\wh{f_1}(\Theta)|^{2}~d\Theta\bigg)^{1/3}\bigg(\int_{\mathbf{T}}|\wh{f_2}(-\Theta)|^{2}~d\Theta\bigg)^{1/2}\bigg(\int_{\mathbf{T}}|\wh{F}(\Theta)|^{6}~d\Theta\bigg)^{1/6}\\
&\ll \delta^{-O(1)}N^{5/6}\sup_{\Theta\in \mathbf{T}}|\wh{f_1}(\Theta)|^{1/3}\bigg(N^{2}W^{-6}\bigg)^{1/6}\\
&\ll \delta^{-O(1)}N^{7/6}W^{-1}\sup_{\Theta\in \mathbf{T}}|\wh{f_1}(\Theta)|^{1/3},
\end{align*}
where we have used \cref{lem:L6-bound} (with $N$ replaced by $N^{1/2}W^{-1}$) to bound the $L^6$-norm of $\wh{F}$.
\end{proof}

\subsection{Dual-difference interchange}
The version of dual-difference interchange we use is a minor variant of \cite[Lemma~7.4]{Pel20}; we include a proof for completeness. 
\begin{lemma}\label{lem:dual-interchange}
Consider a $1$-bounded function $f\colon\mathbf{Z}\times S\to\mathbf{C}$ such that $\supp{f(\cdot,y)}\subseteq [-CN,CN]$ for all $y\in S$, and integers $T_1, T_2$ such that $T_1\cdot T_2\le CN$. Set $F(x) := \mathbf{E}_{y\in S}f(x,y)$, fix integers $1\le \ell\le k$, and suppose that 
\[\sum_{x\in \mathbf{Z}}\mathbf{E}_{\substack{h_i,h_i'\in T_1 \cdot [T_2]\\1\le i\le k}}\Delta'_{(h_i,h_i')_{i=1}^{k}}F(x)\ge \delta N.\]
Then, we have that 
\[\mathbf{E}_{\substack{h_i,h_i'\in T_1 \cdot [T_2] \\ 1\le i\le \ell}}\snorm{\mathbf{E}_{y\in S}\Delta'^{(x)}_{(h_i,h_i')_{i=1}^{\ell}}f(x,y)}_{U_{T_1 \cdot [T_2]}^{k-\ell}}^{2^{k-\ell}} \gg (C^{-1}\delta)^{O_k(1)}N.\]
\end{lemma}
\begin{proof}
The proof is exactly as in \cite[Lemma~7.4]{Pel20}, noting that the properties of the dual function are used only in the form of $F$ given above. For the computation below, let $\vec{h} = (h_1,\ldots,h_{k-1})$ and $\vec{h}' = (h_1',\ldots,h_{k-1}')$, and let $\mc{C}^t$ denote complex conjugation $t$ times (which depends only on the parity of $t$). We have, using Cauchy--Schwarz to duplicate $h_k'$ in the middle,
\begin{align*}
&\bigg(\sum_{x\in \mathbf{Z}}\mathbf{E}_{\substack{h_i,h_i'\in T_1 \cdot [T_2] \\ 1\le i\le k}}\Delta'_{(h_i,h_i')_{i=1}^{k}}F(x)\bigg)^2\\
&=\bigg(\mathbf{E}_{\substack{y_{\omega0},y_{\omega1}\in S\\\omega\in \{0,1\}^{k-1}}}\sum_{x\in \mathbf{Z}}\mathbf{E}_{\substack{h_i,h_i'\in T_1 \cdot [T_2] \\ 1\le i\le k}}\prod_{\omega\in \{0,1\}^{k-1}}\mc{C}^{|\omega|-1}(f(x + \vec{h}\cdot \omega + \vec{h'}\cdot (1-\omega) + h_k ,y_{\omega0})\\
&\qquad\qquad\qquad\qquad\qquad\qquad\qquad\qquad\qquad\qquad\times \ol{f(x + \vec{h}\cdot \omega + \vec{h'}\cdot (1-\omega) + h_k' ,y_{\omega1})})\bigg)^2\\
&\le\bigg(\mathbf{E}_{\substack{y_{\omega0},y_{\omega1}\in S\\\omega\in \{0,1\}^{k-1}}}\sum_{x\in \mathbf{Z}}\mathbf{E}_{\substack{h_i,h_j'\in T_1 \cdot [T_2]\\1\le i\le k\\1\le j\le k-1}}\prod_{\omega\in \{0,1\}^{k-1}}|f(x + \vec{h}\cdot \omega + \vec{h'}\cdot (1-\omega) + h_k ,y_{\omega0})|^2\bigg)\\
&\ \ \cdot \bigg(\mathbf{E}_{\substack{y_{\omega0},y_{\omega1}\in S\\\omega\in \{0,1\}^{k-1}}}\sum_{x\in \mathbf{Z}}\mathbf{E}_{\substack{h_i,h_j'\in T_1 \cdot [T_2]\\1\le i\le k\\1\le j\le k-1}}\mathbf{E}_{h_{k,1}',h_{k,2}'\in T_1\cdot [T_2]}\prod_{\omega\in \{0,1\}^{k-1}}\mc{C}^{|\omega|-1}(f(x + \vec{h}\cdot \omega + \vec{h'}\cdot (1-\omega) + h_{k,1}' ,y_{\omega1})\\
&\qquad\qquad\qquad\qquad\qquad\qquad\qquad\qquad\qquad\qquad\qquad\qquad\qquad\quad \times\ol{f(x + \vec{h}\cdot \omega + \vec{h'}\cdot (1-\omega) + h_{k,2}' ,y_{\omega1})})\bigg)\\
&\ll CN\cdot \bigg(\mathbf{E}_{\substack{y_{\omega0},y_{\omega1}\in S\\\omega\in \{0,1\}^{k-1}}}\sum_{x\in \mathbf{Z}}\mathbf{E}_{\substack{h_i,h_i'\in T_1 \cdot [T_2] \\ 1\le i\le k}}\prod_{\omega\in \{0,1\}^{k-1}}\mc{C}^{|\omega|-1}\Delta'^{(x)}_{(h_k,h_k')}f(x + \vec{h}\cdot \omega + \vec{h'}\cdot (1-\omega) ,y_{\omega1})\bigg)\\
&\ll CN\cdot \bigg(\mathbf{E}_{\substack{y_{\omega}\in S\\\omega\in \{0,1\}^{k-1}}}\sum_{x\in \mathbf{Z}}\mathbf{E}_{\substack{h_i,h_i'\in T_1 \cdot [T_2]\\1\le i\le k}}\prod_{\omega\in \{0,1\}^{k-1}}\mc{C}^{|\omega|-1}\Delta'^{(x)}_{(h_k,h_k')}f(x + \vec{h}\cdot \omega + \vec{h'}\cdot (1-\omega),y_{\omega})\bigg).
\end{align*}
The result follows by replacing replacing $f$ by $\Delta'^{(x)}_{(h_k,h_k')}f$ and applying iterating, for a total of $\ell$ times. We use that $T_1 \cdot T_2$ is smaller than $CN$ in order to guarantee appropriate support conditions and bounds.
\end{proof}

\subsection{Hensel's lemma}
We will also require an elementary result number-theoretic result; this is ultimately why the $W$-trick can be used to treat arithmetic progressions with common difference of the form $y^2-1$, but not $y^2$. 
\begin{proposition}\label{prop:Hensel}
Let $Q(y) = ay^2 + by$ and fix a prime $p$ such that $p\mid a$ but $p\nmid b$. Then, for all $k\ge 1$, $Q(y)$ gives a bijective map $\mathbf{Z}/p^k\mathbf{Z}\to\mathbf{Z}/p^k\mathbf{Z}$.  
\end{proposition}
\begin{proof}
Note that for $k = 1$ this is immediate, as $P(y)$ reduces to a nontrivial linear function on $\mathbf{Z}/p\mathbf{Z}$. Furthermore, note that $P'(y) = 2ay + b$ is always nonzero when viewed modulo $p$. Therefore, the desired result follows from Hensel's lemma. 
\end{proof}

\subsection{Completing the proof of \texorpdfstring{\cref{prop:degree-lower}}{Proposition 6.1}}
Before proceeding with the main proof, we require the $U^{3}$-inverse theorem. The result stated follows by embedding the interval $[N]$ into a slightly larger cyclic group and using the $U^3$-inverse theorem of Green and Tao \cite[Theorem~12.8]{GT08b}\footnote{Note that the theorem stated in \cite[Theorem~12.8]{GT08b} produces correlation of a shifted version of $f$ with a nilsequence but, as remarked after the theorem, the shift can be removed.}. We give a brief deduction of the inverse theorem stated below from \cite[Theorem~12.8]{GT08b}, since the definition of $U^3$-norm we use is slightly different from the standard version. 
\begin{theorem}\label{thm:U^3}
Suppose that $f\colon\mathbf{Z}\to\mathbf{C}$ is a $1$-bounded function such that $\on{supp}(f)\subseteq [\pm N]$ and 
\[\snorm{f}_{U^{3}_{[5N]}}^{8}\ge \delta N.\]
Then, there exists a degree $2$ nilmanifold $G/\Gamma$ with dimension $\delta^{-O(1)}$ and complexity $\exp(\delta^{-O(1)})$, a function $F\colon G/\Gamma\to\mathbf{C}$ with $\snorm{F}_{\mr{Lip}}\le \exp(\delta^{-O(1)})$, and a polynomial sequence $g\colon\mathbf{Z}\to G$ such that 
\[\bigg|\sum_{n\in\mathbf{Z}} f(n) F(g(n)\Gamma)\bigg|\ge \exp(-\delta^{O(1)})N.\]
\end{theorem}
\begin{proof}
Note that 
\begin{equation*}
\delta N\le\snorm{f}_{U^{3}_{[5N]}}^{8} = \sum_{h}\mu_{5N}\snorm{\Delta_h f}_{U^2_{[5N]}}^2\asymp \mathbf{E}_{h\in [\pm 2N]}\snorm{\Delta_h f}_{U^2_{[5N]}}^2,
\end{equation*}
where we have used that $\Delta_h f = 0$ for $|h|>2N$ and that there exist absolute constants $c,C>0$ such that $cN^{-1}\leq\mu_{5N}(h)\leq C N^{-1}$ for all $|h|\le 2N$. By Markov and \cref{lem:U2-inver}, we find that 
\[\mathbf{E}_{h\in [\pm 2N]}\sup_{\beta\in \mathbf{T}}\bigg|\sum_{x\in \mathbf{Z}} \Delta_hf(x)e(\beta x)\bigg|\gg \delta^{O(1)}N.\]
Note that if the $\beta$ Fourier sum is large then so will the $[\beta\pm\delta^{O(1)}/N]$ Fourier sums. So, by the identity
\[\sum_{x,h_1,h_2}f(x)\ol{f(x+h_1)}\ol{f(x+h_2)}f(x+h_1+h_2) = \int_{\mathbf{T}}|\wh{f}(\Theta)|^4~d\Theta\]
and Markov, it follows that 
\[\mathbf{E}_{h_3\in [\pm 2N]}\sum_{x,h_1,h_2}(\Delta_{h_3}f)(x)\ol{(\Delta_{h_3}f)(x+h_1)}\ol{(\Delta_{h_3}f)(x+h_2)}(\Delta_{h_3}f)(x+h_1+h_2)\gg \delta^{O(1)}N^3.\]
This implies
\begin{equation*}
\sum_{x,h_1,h_2,h_3}\Delta_{h_1,h_2,h_3}f(x)\gg \delta^{O(1)}N^4.
\end{equation*}
Now treat $f$ as a function on the cyclic group $\mathbf{Z}/(2L+1) \mathbf{Z}$, where $L\in [25N,50N]$, $2L+1$ is prime, and we identify $\mathbf{Z}/(2L+1) \mathbf{Z}$ with $[-L,L]$. The above lower bound implies that $f$, viewed as a function on $\mathbf{Z}/(2L+1) \mathbf{Z}$, has large $U^3$-norm in the sense of \cite[Theorem~12.8]{GT08b}, and therefore the desired result follows from \cite[Theorem~12.8]{GT08b}.
\end{proof}

We now perform a preliminary transformation of \cref{thm:U^3} that allow us to assume that $g(0) = \mr{Id}_{G}$ and that $F$ has a vertical frequency.
\begin{theorem}\label{thm:U^3-mod}
Suppose that $f\colon\mathbf{Z}\to\mathbf{C}$ is a $1$-bounded function such that $\on{supp}(f)\subseteq [\pm N]$ and 
\[\snorm{f}_{U^{3}_{[5N]}}^{8}\ge \delta N.\]
Then there exists a degree $2$ nilmanifold $G/\Gamma$ with dimension $\delta^{-O(1)}$ and complexity $\exp(\delta^{-O(1)})$, a function $F\colon G/\Gamma\to\mathbf{C}$ with $\snorm{F}_{\mr{Lip}}\le \poly_{\delta^{-1}}(\delta^{-1})$ possessing a vertical frequency $\xi$ with $\snorm{\xi}\le \poly_{\delta^{-1}}(\delta^{-1})$, and a polynomial sequence $g\colon\mathbf{Z}\to G$ with $g(0) = \mr{Id}_{G}$ such that
\[\bigg|\sum_{n\in\mathbf{Z}} f(n) F(g(n)\Gamma)\bigg|\ge \poly_{\delta^{-1}}(\delta)N.\]
\end{theorem}
\begin{proof}
First apply \cref{thm:U^3} to find some $G/\Gamma,F,g$ which appropriately correlated with $f$. We may replace $F$ by some $F'$ which has a vertical frequency $\snorm{\xi}\le\on{poly}_{\delta^{-1}}(\delta^{-1})$ by applying \cref{lem:vertical-expansion} with error parameter $\eps$ taken to be $\exp(-\delta^{-O(1)})$ and using the pigeonhole principle. The Lipschitz constant is now of quality $\on{poly}_{\delta^{-1}}(\delta^{-1})$. Note here we are using that $\poly_{\delta^{-1}}(\exp(\delta^{-O(1)}))\le\poly_{\delta^{-1}}(\delta^{-1})$ up to changing the implicit constants.

To force $g(0) = \mr{Id}_{G}$, by using \cite[Lemma~A.14]{GT12} we can factor $g(0) = \{g(0)\}[g(0)]$ with $\snorm{\psi(\{g(0)\})}_{\infty}\le 1$ and $[g(0)]\in \Gamma$. Then, we have that 
\begin{align*}
F'(g(n)\Gamma) &= F'(g(n)g(0)^{-1}g(0)\Gamma) \\
&= F'(g(n)g(0)^{-1}\{g(0)\}\Gamma) \\
&= F'(\{g(0)\}(\{g(0)\}^{-1}g(n)g(0)^{-1}\{g(0)\})\Gamma),
\end{align*}
and taking $\wt{F}(x) = F'(\{g(0)\}^{-1}x)$ and $\wt{g}(n) = \{g(0)\}^{-1}g(n)g(0)^{-1}\{g(0)\}$ gives the desired. 
\end{proof}

We are now in position to complete the proof of \cref{prop:degree-lower}.
\begin{proof}[{Proof of \cref{prop:degree-lower}}]
Throughout the proof $\delta$ will be assumed to be smaller than an appropriate absolute constant. 

\noindent\textbf{Step 1: Applying dual-difference interchange.} By the definition of the box-norm, we have that 
\[\sum_{x\in \mathbf{Z}} \mathbf{E}_{\substack{h_i,h_i'\in W\cdot[ N/W]\\ 1\le i\le k}}\Delta'_{(h_i,h_i')_{i=1}^{k}} \mc{D}^{1}(f_2,f_3) \ge \delta N.\]
Recall that $\mc{D}^{1}(f_2,f_3)(x) = \mathbf{E}_{y\in [\pm M]}f_2(x+P(y))f_3(x+2P(y))$ with $M=\lfloor N^{1/2}W^{-1/2}\rfloor$, and define 
\[g(x,y) = f_2(x+P(y))f_3(x+2P(y))\mbm{1}_{|y|\le M}\mbm{1}_{|x|\le 100\delta^{-1} N}.\]
It follows via the support conditions on the $f_i$ that we immediately have  
\[\sum_{x\in \mathbf{Z}} \mathbf{E}_{\substack{h_i,h_i'\in W\cdot[ N/W] \\ 1\le i\le k}}\Delta'_{(h_i,h_i')_{i=1}^{k}} (\mathbf{E}_{y\in [\pm M]}g(x,y)) \ge \delta N.\]
Applying \cref{lem:dual-interchange}, we deduce that 
\[\mathbf{E}_{\substack{h_i,h_i'\in W\cdot[ N/W] \\ 1\le i\le k-3}}\snorm{\mathbf{E}_{y\in [\pm M]} \Delta'^{(x)}_{(h_i,h_i')_{i=1}^{k-3}}g(x,y)}_{U^{3}_{W\cdot[ N/W]}}^{8}\gg \delta^{O_k(1)}N.\]
Therefore, there are at least $\delta^{O_k(1)}(N/W)^{2(k-3)}$ shifts $(h_i,h_i')_{i=1}^{k-3}\in (W\cdot[N/W])^{2\times(k-3)}$ such that 
\[\snorm{\mathbf{E}_{y\in [\pm M]} \Delta'^{(x)}_{(h_i,h_i')_{i=1}^{k-3}}g(x,y)}_{U^{3}_{W\cdot[ N/W]}}^{8}\gg \delta^{O_k(1)}N.\]

\noindent\textbf{Step 2: Setup for applying the $U^3$-inverse theorem.}
For the next few labeled steps, we fix shifts $(h_i,h_i')_{i=1}^{k-3}\in (W\cdot[N/W])^{2\times(k-3)}$ such that 
\[\snorm{\mathbf{E}_{y\in [\pm M]} \Delta'^{(x)}_{(h_i,h_i')_{i=1}^{k-3}}g(x,y)}_{U^{3}_{W\cdot[N/W]}}^{8}\gg \delta^{O_k(1)}N.\]
Furthermore, denote
\[f_j^{(1)}(x) = \Delta'^{(x)}_{(h_i,h_i')_{i=1}^{k-3}}f_j(x)\]
for $j\in \{2,3\}$. 

Since all the differences defining the box-norm are divisible by $W$, we have 
\[\sum_{j\in [W]}\snorm{\mathbf{E}_{y\in [\pm M]}f_2^{(1)}(Wx + j + P(y))f_3^{(1)}(Wx + j + 2P(y))}_{U^{3}_{[N/W]}}^{8}\gg \delta^{O_k(1)}N.\]
By the triangle inequality, we have
\[\sum_{\substack{j\in [W]\\k\in [W]}}\snorm{\mathbf{E}_{\substack{y\in [\pm MW^{-1}]}}f_2^{(1)}(Wx + j + P(Wy + k))f_3^{(1)}(Wx + j + 2P(Wy + k)))}_{U^{3}_{[N/W]}}^{8}\gg \delta^{O_k(1)}NW.\]
Defining
\[f_{i,t}^{(2)}(x) = f_i^{(1)}(Wx + t),\quad\text{and}\quad P_r(y) = \frac{P(Wy+r)-P(r)}{W},\]
we therefore have 
\[\sum_{\substack{j\in [W]\\k\in [W]}}\snorm{\mathbf{E}_{\substack{y\in [\pm MW^{-1}]}}f_{2,j+P(k)}^{(2)}(x + P_k(y))f_{3,j+2P(k)}^{(2)}(x + 2P_k(y))}_{U^{3}_{[N/W]}}^{8}\gg \delta^{O_k(1)}NW.\]
Thus, for at least a $\delta^{O_k(1)}$ fraction of pairs $(j,k)\in [W]^2$, we have that 
\[\snorm{\mathbf{E}_{\substack{y\in [\pm MW^{-1}]}}f_{2,j+P(k)}^{(2)}(x + P_k(y))f_{3,j+2P(k)}^{(2)}(x + 2P_k(y))}_{U^{3}_{[N/W]}}^{8}\gg \delta^{O_k(1)}NW^{-1}.\]

We fix such $j$ and $k$ for the next few labeled steps within the argument. We now perform a certain set of artificial changes of variables; this change of variable is directly inspired by work of Altman \cite{Alt22}, and is used to reduce to considering the ``flag'' set of forms $\{2(x+2y), 3(x+y), 6y, 6x\}$ that was considered in \cref{sec:nil-main}.

Define $f_{2,t}^{(3)}(x) = f_{2,t}^{(2)}(x/3)\mbm{1}_{3|x}$, $f_{3,t}^{(3)}(x) = f_{3,t}^{(2)}(x/2)\mbm{1}_{2|x}$, 
\[H(x) = \mathbf{E}_{\substack{y\in [\pm MW^{-1}]}}f_{2,j+P(k)}^{(3)}(3(x + P_k(y)))f_{3,j+2P(k)}^{(3)}(2(x + 2P_k(y))),\]
and
\[H^{\ast}(x) = H(x/6)\mbm{1}_{6|x}.\]
Note that
\[\snorm{H^{\ast}(x)}_{U^{3}_{[6N/W]}}^{8}\gg \delta^{O_k(1)}NW^{-1};\]
this follows via expanding the definition of the box-norm and noting that $H^{\ast}$ is only supported on multiples of $6$.

\noindent\textbf{Step 3: Applying the $U^3$-inverse theorem and reduction to \cref{lem:poly-sequence}.}
Note that, by \cref{cor:rescale-loss}, we have
\[\snorm{H^{\ast}(x)}_{U^{3}_{[10^{3}\delta^{-1}N/W]}}^{8}\gg \delta^{O_k(1)}NW^{-1}.\]
Therefore, by \cref{thm:U^3-mod} (applied noting that $H^{\ast}(x)$ has support contained in $[\pm 30\delta^{-1} N/W]$), we have 
\[\bigg|\sum_{x\in \mathbf{Z}}H^{\ast}(x) F(g(x)\Gamma)\bigg|\ge \poly_{\delta^{-1}}(\delta^{O_k(1)})NW^{-1},\]
where $G/\Gamma$, $F\colon G/\Gamma\to\mathbf{C}$, and $g$ are as in \cref{thm:U^3-mod}. Let the vertical character of $F$ be $\xi$. Unwinding the definition of $H^{\ast}(x)$, we in fact that have that 
\[\bigg|\sum_{x\in \mathbf{Z}}H(x) F(g(6x)\Gamma)\bigg|\ge \poly_{\delta^{-1}}(\delta^{O_k(1)})NW^{-1}.\]
Inserting the definition of $H(x)$ yields
\begin{equation}\label{eq:correl}
\bigg|\sum_{x\in \mathbf{Z}}F(g(6x)\Gamma)\mathbf{E}_{\substack{y\in [\pm MW^{-1}]}}f_{2,j+P(k)}^{(3)}(3(x + P_k(y)))f_{3,j+2P(k)}^{(3)}(2(x + 2P_k(y)))\bigg|\ge \poly_{\delta^{-1}}(\delta^{O_k(1)})NW^{-1}.  
\end{equation}
We will return to \eqref{eq:correl} eventually; we first deduce a series of structural claims regarding the polynomial sequence $g(\cdot)$. 

Applying \cref{lem:dual-expansion} with $\eps = \poly_{\delta^{-1}}(\delta^{O_k(1)})$ and using the pigeonhole principle to choose a single $\alpha$, there exist functions $F_1$ with vertical character $-9\xi$, $F_2$ with vertical character $8\xi$, and $F_3$ with vertical character $2\xi$ such that 
\begin{align*}
&\bigg|\sum_{x\in \mathbf{Z}}\mathbf{E}_{\substack{y\in [\pm MW^{-1}]}}F_3(g(6P_k(y))\Gamma)f_{2,j+P(k)}^{(3)}(3(x + P_k(y)))F_2(g(3(x+P_k(y)))\Gamma)\\
&\qquad\qquad\qquad\qquad f_{3,j+2P(k)}^{(3)}(2(x + 2P_k(y)))F_1(g(2(x+2P_k(y)))\Gamma)\bigg|\ge \poly_{\delta^{-1}}(\delta^{O_k(1)})NW^{-1}
\end{align*}
and $\snorm{F_i}_{\mr{Lip}}\le \poly_{\delta^{-1}}(\delta^{-O_k(1)})$ for each $i=1,2,3$. Set $\wt{F}_1(z):=f_{2,j+P(k)}^{(3)}(3z)F_2(g(3z)\Gamma)$ and $\wt{F}_2(z):=f_{3,j+2P(k)}^{(3)}(2z)F_1(g(2z)\Gamma)$. By Parseval's identity, we have
\begin{align*}
&\bigg|\mathbf{E}_{\substack{y\in [\pm MW^{-1}]}}F_3(g(6P_k(y))\Gamma)\int_{\mathbf{T}}\widehat{\wt{F}_1}(\eta)\widehat{\wt{F}_2}(\eta)e(\eta P_k(y))d\eta\bigg|\ge \poly_{\delta^{-1}}(\delta^{O_k(1)})NW^{-1}
\end{align*}
or equivalently
\begin{align*}
&\bigg|\int_{\mathbf{T}}\widehat{\wt{F}_1}(\eta)\widehat{\wt{F}_2}(\eta)\mathbf{E}_{\substack{y\in[\pm MW^{-1}]}}F_3(g(6P_k(y))\Gamma)e(\eta P_k(y))d\eta\bigg|\ge \poly_{\delta^{-1}}(\delta^{O_k(1)})NW^{-1}.
\end{align*}

Thus, there exists $\beta \in \mathbf{T}$ such that
\[\bigg|\mathbf{E}_{\substack{y\in [\pm MW^{-1}]}}F_3(g(6P_k(y))\Gamma)e(6\beta P_k(y))\bigg|\ge \poly_{\delta^{-1}}(\delta^{O_k(1)}),\]
since $\max\{\snorm{\wt{F}_1}_2^2,\snorm{\wt{F}_2}_2^2\}\ll \delta^{-1} NW^{-1}$. Now, fix a choice of $\beta^{\ast}\in \mathbf{T}$ such that $2\beta^{\ast} = \beta$.

\noindent\textbf{Step 4: Applying \cref{lem:poly-sequence}.}
We now use \cref{lem:poly-sequence} to reduce the degree of the polynomial sequence $g$. The argument splits into two cases. In the case when $\xi$ is zero, we will be able to directly reduce the degree of the nilsequence; we defer this case until later. 

If $\xi$ is nonzero, let $G_{\bullet}$ denote the degree $2$ filtration $G = G_0 = G_1\geqslant G_2\geqslant \mr{Id}_G$ relative to which $g$ is a polynomial sequence. As $\xi$ is nonzero, there exists $h\in G_2$ such that $\xi(h) = \beta^{\ast}$. Define $\wt{g}(n) = g(n) h^{n}$. Then $\wt{g}(0) = \mr{Id}_G$ and $\wt{g}$ is a polynomial sequence with respect to $G_{\bullet}$. By construction, we have that 
\[\bigg|\mathbf{E}_{\substack{y\in [\pm MW^{-1}]}}F_3(\wt{g}(6P_k(y))\Gamma)\bigg|\ge \poly_{\delta^{-1}}(\delta^{O_k(1)}).\]

This is exactly the setup of \cref{lem:poly-sequence}. We may thus factor $\wt{g}(n)$ as $\wt{g} = \eps \cdot g' \cdot \gamma$ with $\eps, g',\gamma\in\on{Poly}(\mathbf{Z},G_\bullet)$, where
\begin{itemize}
    \item for all $t\in [\pm 100\delta^{-1}\cdot N/W]$, $d(\eps(t),\eps(t-1))\le W\on{poly}_{\delta^{-1}}(\delta^{-O_k(1)})/N$ and $d(\eps(t),\mr{Id}_G)\le \on{poly}_{\delta^{-1}}(\delta^{-O_k(1)})$,
    \item $\gamma$ is $\on{poly}_{\delta^{-1}}(\delta^{-O_k(1)})$-rational and $\gamma(n)\Gamma$ is periodic with period at most $\on{poly}_{\delta^{-1}}(\delta^{-O_k(1)})$, 
    \item $g'$ takes values only $G'$, a connected proper $\on{poly}_{\delta^{-1}}(\delta^{-O_k(1)})$-rational subgroup with respect to $\mc{X}$, and may be viewed as a polynomial sequence with respect to the filtration $G_{\bullet}'$, where $G_i'
    = G' \cap G_i$,
    \item $\xi$ is trivial on $G_2' = G' \cap G_2$.
\end{itemize}
\noindent\textbf{Step 5: Setup for degree-reduction.}
Recall from \eqref{eq:correl} that
\[\bigg|\sum_{x\in \mathbf{Z}}F(g(6x)\Gamma)\mathbf{E}_{\substack{y\in [\pm MW^{-1}]}}f_{2,j+P(k)}^{(3)}(3(x + P_k(y))f_{3,j+2P(k)}^{(3)}(2(x + 2P_k(y)))\bigg|\ge \poly_{\delta^{-1}}(\delta^{O_k(1)})NW^{-1}.\]
By the definitions of $f_{2,t}^{(3)}$, $f_{3,t}^{(3)}$, and $\wt{g}$ and since $F$ has vertical character $\xi$, it follows that 
\begin{align*}
&\bigg|\sum_{x\in \mathbf{Z}}e(-6\beta^{\ast}x)F(\wt{g}(6x)\Gamma)\mathbf{E}_{\substack{y\in [\pm MW^{-1}]}}f_{2,j+P(k)}^{(2)}(x + P_k(y))f_{3,j+2P(k)}^{(2)}(x + 2P_k(y))\bigg|\\
&\qquad\qquad\ge \poly_{\delta^{-1}}(\delta^{O_k(1)})NW^{-1}.
\end{align*}
We next break $[\pm 100\delta^{-1}N/W]$ into nearly-equal length arithmetic progressions $\{Q_1,\ldots, Q_t\}$ of length $\on{poly}_{\delta^{-1}}(\delta^{O_k(1)})N/W$, with difference equal to the period of $\gamma$. By the pigeonhole principle, there exists $Q\in\{Q_1,\ldots, Q_t\}$ such that 
\begin{align*}
&\bigg|\sum_{x\in Q}e(-6\beta^{\ast}x)F(\wt{g}(6x)\Gamma)\mathbf{E}_{\substack{y\in [\pm MW^{-1}]}}f_{2,j+P(k)}^{(2)}(x + P_k(y))f_{3,j+2P(k)}^{(2)}(x + 2P_k(y))\bigg|\\
&\qquad\qquad\ge \poly_{\delta^{-1}}(\delta^{O_k(1)})|Q|.
\end{align*}
By choosing the implicit constant in the length of $Q$ sufficiently large (so that the length is small), we get, in fact, that 
\begin{align*}
&\bigg|\sum_{x\in Q}e(-6\beta^{\ast}x)F(\eps_Qg'(6x)\gamma_Q\Gamma)\mathbf{E}_{\substack{y\in [\pm MW^{-1}]}}f_{2,j+P(k)}^{(2)}(x + P_k(y))f_{3,j+2P(k)}^{(2)}(x + 2P_k(y))\bigg|\\
&\qquad\qquad\ge \poly_{\delta^{-1}}(\delta^{O_k(1)})|Q|,
\end{align*}
where $\eps_Q$ is a $\on{poly}_{\delta^{-1}}(\delta^{-O_k(1)})$-bounded element and $\gamma_Q$ is $\on{poly}_{\delta^{-1}}(\delta^{-O_k(1)})$-bounded and rational element. (A similar argument appears in the proof of \cref{lem:poly-sequence}.) Let $\wt{F}(x) = F(\eps_Q\gamma_Q x)$ and $g^{(2)} = \gamma_Q^{-1}g'\gamma_Q$, so that  
\begin{align*}
&\bigg|\sum_{x\in Q}e(-6\beta^{\ast}x)\wt{F}(g^{(2)}(6x)\Gamma)\mathbf{E}_{\substack{y\in [\pm MW^{-1}]}}f_{2,j+P(k)}^{(2)}(x + P_k(y))f_{3,j+2P(k)}^{(2)}(x + 2P_k(y))\bigg|\\
&\qquad\qquad\ge \poly_{\delta^{-1}}(\delta^{O_k(1)})|Q|. 
\end{align*}

Note that $g^{(2)}$ is a polynomial sequence with respect to the filtration $\gamma_Q^{-1}G_{\bullet}'\gamma_Q$. We now claim that $\xi$ is trivial on $\gamma_Q^{-1}G_{2}'\gamma_Q$. Indeed, $G_2'\subseteq G_2$ and $[G,G_2] = \mr{Id}_G$, and thus $G_{2}'\subseteq Z(G)$ (the center of $G$). It follows that  $\gamma_Q^{-1}G_{2}'\gamma_Q = G_2'$ and thus $\xi$ is trivial on $\gamma_Q^{-1}G_{2}'\gamma_Q$. Furthermore, note that $\wt{F}$ has vertical frequency $\xi$, is $\on{poly}_{\delta^{-1}}(\delta^{-O_k(1)})$-Lipschitz (with respect to a suitable Mal'cev basis on $\gamma_Q^{-1}G'\gamma_Q$), and that $\gamma_Q^{-1}G'\gamma_Q$ is $\on{poly}_{\delta^{-1}}(\delta^{-O_k(1)})$-rational (see \cite[Lemma~A.13]{GT12}) with respect to $G$.

Let $\Gamma' = (\gamma_Q^{-1}G'\gamma_Q)\cap \Gamma$. Since $(\gamma_Q^{-1}G'\gamma_Q)\Gamma/\Gamma \cong (\gamma_Q^{-1}G'\gamma_Q)/\Gamma'$, we have
\begin{align*}
&\bigg|\sum_{x\in Q}e(-6\beta^{\ast}x)\wt{F}(g^{(2)}(6x)\Gamma')\mathbf{E}_{\substack{y\in [\pm MW^{-1}]}}f_{2,j+P(k)}^{(2)}(x + P_k(y))f_{3,j+2P(k)}^{(2)}(x + 2P_k(y))\bigg|\\
&\qquad\qquad\ge \poly_{\delta^{-1}}(\delta^{O_k(1)})|Q|. 
\end{align*}
As $\gamma_Q^{-1}G'\gamma_Q$ is a sufficiently rational subgroup, one may put a Mal'cev basis $\mc{X}'$ on $\Gamma'$ such that the Lipschitz bounds on $F$ transfer to $\mc{X}'$. We note that until this point, we have been operating under the assumption that $\xi$ is nonzero. When $\xi$ is zero, by taking $\beta^{\ast} = 0$, we can immediately find ourselves in the same situation by taking $\eps_Q = \gamma_Q = \mr{Id}_G$, $G' = G$, and $\Gamma' = \Gamma$.

From these last couple steps, the key extra property we have guaranteed compared to \eqref{eq:correl} is that we know $g^{(2)}$ lives in $\gamma_Q^{-1}G'\gamma_Q$ and also $\xi$ is trivial on $G_2\cap(\gamma_Q^{-1}G'\gamma_Q)$.

\noindent\textbf{Step 6: Degree-reduction.}
We are finally in a position to obtain the necessary degree reduction. Given the above setup, we define $G^{\ast} := \gamma_Q^{-1}G'\gamma_Q/(\gamma_Q^{-1}G_2'\gamma_Q)$ and take $g^{(3)} \equiv g^{(2)} \mod \gamma_Q^{-1}G_2'\gamma_Q$ to be a polynomial sequence in $G^{\ast}$. Furthermore, let $\Gamma^{\ast} = \Gamma'/(\Gamma'\cap \gamma_Q^{-1}G_2'\gamma_Q)$ and $F^{\ast}$ be the projection of $\wt{F}$ from the domain $G'/\Gamma'$ to the domain $G^{\ast}/\Gamma^{\ast}$ (which is well defined, as $\wt{F}$ is invariant under $\gamma_Q^{-1}G_2'\gamma_Q$). We have
\begin{align*}
&\bigg|\sum_{x\in Q}e(-6\beta^{\ast}x)F^{\ast}(g^{(3)}(6x)\Gamma^{\ast})\mathbf{E}_{\substack{y\in [\pm MW^{-1}]}}f_{2,j+P(k)}^{(2)}(x + P_k(y))f_{3,j+2P(k)}^{(2)}(x + 2P_k(y))\bigg|\\
&\qquad\qquad\ge \poly_{\delta^{-1}}(\delta^{O_k(1)})|Q|.
\end{align*}
Note, however, that now $g^{(3)}$ is a polynomial sequence of degree $1$. Combining \cref{lem:vertical-expansion}, the fact that the functions $f_{2,j+P(k)}^{(2)}$, $f_{3,j+2P(k)}^{(2)}$ are $1$-bounded, and the fact that $Q$ is an arithmetic progression of appropriate length and common difference, it follows using \cref{lem:interval-corr} that
\begin{align*}
&\sup_{\alpha\in \mathbf{T}}\bigg|\sum_{x\in \mathbf{Z}}e(-\alpha x)\mathbf{E}_{\substack{y\in [\pm MW^{-1}]}}f_{2,j+P(k)}^{(2)}(x + P_k(y))f_{3,j+2P(k)}^{(2)}(x + 2P_k(y))\bigg|\ge \poly_{\delta^{-1}}(\delta^{O_k(1)})NW^{-1}.
\end{align*}

\noindent\textbf{Step 7: $U^2$-control of S\'{a}rk\"{o}zy-type configurations}
Fix $\alpha$ such that 
\begin{align*}
\bigg|\sum_{x\in \mathbf{Z}}e(-\alpha x)\mathbf{E}_{\substack{y\in [\pm MW^{-1}]}}f_{2,j+P(k)}^{(2)}(x + P_k(y))f_{3,j+2P(k)}^{(2)}(x + 2P_k(y))\bigg|\ge \poly_{\delta^{-1}}(\delta^{O_k(1)})NW^{-1}.
\end{align*}
This is equivalent to 
\begin{align*}
&\bigg|\sum_{x\in \mathbf{Z}}\mathbf{E}_{\substack{y\in [\pm MW^{-1}]}}f_{2,j+P(k)}^{(2)}(x + P_k(y)) e(-2\alpha(x+P_k(y)))\\
&\qquad\qquad\qquad\qquad\cdot f_{3,j+2P(k)}^{(2)}(x + 2P_k(y))e(\alpha(x+2P_k(y))\bigg|\ge \poly_{\delta^{-1}}(\delta^{O_k(1)})NW^{-1}.
\end{align*}
This immediately implies, by \cref{lem:sarkozy}, that 
\[\min_{i\in \{2,3\}}\sup_{\alpha\in \mathbf{T}}\bigg|\sum_{x\in \mathbf{Z}}e(\alpha x)f_{i, j + (i-1)P(k)}^{(2)}(x)\bigg|\ge \poly_{\delta^{-1}}(\delta^{O_k(1)})NW^{-1}\]
for our original choice of $(j,k)\in[W]^2$.

\noindent\textbf{Step 8: Unwinding and deducing the final result.}
Note that if one samples $j\in [W]$ and $k\in [W]$ uniformly, then $j + (i-1)P(k)$ is uniformly distributed modulo $W$ for each $i$. Also, recall that the correlation was deduced for a positive portion of $j$ and $k$. So, we can deduce
\[\min_{i\in \{2,3\}}\sum_{j\in [W]}\sup_{\alpha\in \mathbf{T}}\bigg|\sum_{x\in \mathbf{Z}}e(\alpha x)f_{i, j}^{(2)}(x)\bigg|\ge \poly_{\delta^{-1}}(\delta^{O_k(1)})N.\]
By the converse to the $U^2$-inverse theorem (see, e.g., \cref{lem:converse}), it follows that 
\[\min_{i\in \{2,3\}}\sum_{j\in [W]}\snorm{f_{i, j}^{(2)}(x)}_{U^{2}_{[N/W]}}^{4}\ge \poly_{\delta^{-1}}(\delta^{O_k(1)})N.\]
Inserting the definition of $f_{i,j}^{(2)}$ yields
\[\min_{i\in \{2,3\}}\snorm{f_{i}^{(1)}(x)}_{U^{2}_{W\cdot[N/W]}}^{4}\ge \poly_{\delta^{-1}}(\delta^{O_k(1)})N.\]
We now unwind the definition of $f_{i}^{(1)}$. Recall from Step 1 that a positive proportion of shifts $(h_i,h_i')_{i=1}^{k-3}\in (W\cdot[N/W])^{k-3}$ were satisfied conditions sufficient for the analysis in Step 2 (and thus subsequent steps) to follow. Therefore, using that the box-norm is always nonnegative, we obtain 
\[\min_{i\in \{2,3\}}\mathbf{E}_{\substack{h_j,h_j'\in W\cdot[N/W]\\ 1\le j\le k-3}}\snorm{\Delta_{(h_j,h_j')_{i=1}^{k-3}}f_{i}(x)}_{U^{2}_{W\cdot[N/W]}}^{4}\ge \poly_{\delta^{-1}}(\delta^{O_k(1)})N.\]
By the definition of the box-norm, this is equivalent to
\[\min_{i\in \{2,3\}}\snorm{f_{i}(x)}_{U^{k-1}_{W\cdot[\pm N/W]}}^{2^{k-1}}\ge \poly_{\delta^{-1}}(\delta^{O_k(1)})N.\]
This (finally) completes the proof. 
\end{proof}

\section{Proof of \texorpdfstring{\cref{thm:main}}{Theorem 1.1}}\label{sec:main}

\subsection{Initial \texorpdfstring{$U^s$}{Us}-norm control and degree-lowering output}
To obtain our initial $U^s$-norm control for the counting operator $\Lambda^{W}$, we can, essentially, apply \cite[Theorem~6.1]{Pel20} as a black-box.
\begin{proposition}\label{prop:PET+Quant-output}
There exists a positive integer $s = s_{\ref{prop:PET+Quant-output}}$ such that the following holds. Fix $1$-bounded functions $f_1,f_2,f_3\colon\mathbf{Z}\to\mathbf{C}$ with $\on{supp}(f_i)\subseteq [\pm CN]$ for $i=1,2,3$, $W$, $M$, and $P$ as in \eqref{eq:Wdef}, and $N\ge (W\delta^{-1})^{\Omega(1)}$. If
\[\bigg|\Lambda^{W}(f_1,f_2,f_3)\bigg|\ge \delta MN,\]
then
\[\min_{i\in [3]}\snorm{f_i}_{U^{s}_{W\cdot[N/W]}}^{2^{s}}\gg_{C} \delta^{O(1)}N.\]
\end{proposition}
\begin{proof}
By shifting the $f_i$, we may assume that they are supported in $[2CN]$ instead. The result is then, essentially, an immediate consequence of \cite[Theorem~6.1]{Pel20}. For $f_1$, apply the result with $P_1(y) = 2Wy^2 + y$ and $P_2(y) = 4Wy^2 + 2y$; for $f_2$, apply the result with $P_1(y) = -2Wy^2 - y$ and $P_2(y) = 2Wy^2 + y$; and 
for $f_3$, apply the result with $P_1(y) = -2Wy^2 - y$ and $P_2(y) = -4Wy^2 - 2y$. In each case, we take $M = \sqrt{N/W}$ and the desired result follows, except that the box-norm may have shift parameters lying in $qW\cdot[\delta^{O(1)}N/W]$ with $q\ll 1$. By applying \cref{lem:mod-up,lem:rescale-down}, we may assume that the shift parameters are the same and thus instead a Gowers norm with parameter $qW\cdot[\delta^{O(1)}N/W]$ with $q\ll 1$. This Gowers norm can be upgraded to the one in the conclusion of the proposition using \cref{cor:rescale-loss,cor:mod-loss}.
\end{proof}

By combining \cref{prop:PET+Quant-output} with our key degree-lowering result, we can deduce that $\Lambda^W$ is controlled by the $U^2$-norm. For the statements below, we let $\exp^k$ denote the $k$-fold iterated exponential. 
\begin{proposition}\label{prop:degree-lower-output}
There exists a positive integer $K = K_{\ref{prop:degree-lower-output}}$ such that the following holds. Suppose that $f_1,f_2,f_3\colon\mathbf{Z}\to\mathbf{C}$ are $1$-bounded functions with $\on{supp}(f_i)\subseteq [CN]$ for $i=1,2,3$, $W$, $M$, and $P$ are as in \eqref{eq:Wdef}, and $N\ge W^{\Omega(1)} \exp^{K}(\delta^{-1})$. If
\[\bigg|\Lambda^{W}(f_1,f_2,f_3)\bigg|\ge \delta MN,\]
then
\[\min_{i\in [3]}\snorm{f_i}_{U^{2}_{W\cdot[N/W]}}^{4}\gg_{C} (\exp^{K}(\delta^{-1}))^{-1}N.\]
\end{proposition}
\begin{proof}
Let $s=s_{\ref{prop:PET+Quant-output}}$. We prove by downwards induction on $k\in\{2,\ldots,s\}$ that given appropriate support and boundedness conditions on functions $h_i\colon\mathbf{Z}\to\mathbf{C}$, we have that $|\Lambda^W(h_1,h_2,h_3)|\ge\delta MN$ implies
\[\min_{i\in[3]}\snorm{h_i}_{U^2_{W\cdot[N/W}}^4\gg_C\exp^{2(s-k)}(\delta^{-O(1)})^{-1}N.\]
For $k=s$, this is \cref{prop:PET+Quant-output}. The result for $k=2$ with $h_i=f_i$ is the desired.

Now suppose that we have established the result for $k\ge 3$ and wish to prove it for $k-1$. Note that
\begin{align*}
\delta MN\ll\Lambda^{W}(h_1,h_2,h_3) &= (2M+1)\sum_{x\in \mathbf{Z}}h_i(x)\mc{D}^{1}(h_{2},h_{3})(x)\\
&\le (2M+1)\bigg(\sum_{x\in \mathbf{Z}}|h_i(x)|^2\bigg)^{1/2}\bigg(\sum_{x\in \mathbf{Z}}|\mc{D}^{1}(h_{2},h_{3})(x)|^2\bigg)^{1/2}\\
&\ll M\cdot N^{1/2}\cdot \Lambda^{W}(\ol{\mc{D}^{1}(h_{2},h_{3})},h_2,h_3)^{1/2}
\end{align*}
and, therefore,
\[\Lambda^{W}(\ol{\mc{D}^{1}(h_{2},h_{3})},h_2,h_3)\gg\delta^{2}MN.\]
Now apply the inductive hypothesis with $h_1$ replaced by $\ol{\mc{D}^{1}(h_{2},h_{3})}$ (which still is bounded and with appropriate support) and $\delta$ replaced by $\Omega(\delta^2)$. We deduce
\[\snorm{\mc{D}^{1}(h_{2},h_{3})}_{U^k_{W\cdot[N/W]}}^{2^k}\gg_C\exp^{2(s-k)}(\delta^{-O(1)})^{-1}N.\]
Similarly, we have
\[\snorm{\mc{D}^{3}(h_{1},h_{2})}_{U^k_{W\cdot[N/W]}}^{2^k}\gg_C\exp^{2(s-k)}(\delta^{-O(1)})^{-1}N.\]
(Note that the $O(1)$ exponents here may decay with each induction step, but $s=s_{\ref{prop:PET+Quant-output}}$ is an absolute constant so this will remain bounded at the end.)

Now using \cref{prop:degree-lower,prop:degree-lower-2}, it follows that
\[\min_{i\in[3]}\snorm{h_i}_{U^{k-1}_{W\cdot[N/W]}}^{2^{k-1}}\gg_C\exp^{2(s-k) + 2}(\delta^{-O(1)})^{-1}N,\]
using that $O_k(1)=O(1)$ as $k$ is bounded, which completes the induction.
\end{proof}

\subsection{Completing the proof}
We are now in position to complete the proof. The following result states that, for $1$-bounded functions, the counting operators $\Lambda^{W}$ and $\Lambda^{\mr{Model}}$ agree up to a universal scaling factor. 
\begin{proposition}\label{prop:transference}
There exists an integer $K = K_{\ref{prop:transference}}>0$ such that the following holds. Suppose $f_1,f_2,f_3\colon\mathbf{Z}\to\mathbf{C}$ are $1$-bounded functions such that $\on{supp}(f_i)\subseteq [N]$ for $i=1,2,3$, $W$, $M$, $w$, and $P$ are as in \eqref{eq:Wdef}, and $N\gg W^{\Omega(1)}$ and $W\gg\exp^{K}(\delta^{-1})$. Then, 
\[\bigg|(NW)^{1/2}\Lambda^{W}(f_1,f_2,f_3) - \Lambda^{\mr{Model}}(f_1,f_2,f_3)\bigg|\le\delta N^2.\]
\end{proposition}
\begin{proof}
Assume for the sake of contradiction that 
\[\bigg|(NW)^{1/2}\Lambda^{W}(f_1,f_2,f_3) - \Lambda^{\mr{Model}}(f_1,f_2,f_3)\bigg|\ge\delta N^2,\]
and define
\[\wt{\Lambda}(f_1,f_2,f_3) = (NW)^{1/2}\Lambda^{W}(f_1,f_2,f_3) - \Lambda^{\mr{Model}}(f_1,f_2,f_3).\]
For this proof, define the modified dual functions
\begin{align*}
D^{1,\ast}(f_2,f_3)(x) &= N^{-1}((NW)^{1/2}\sum_{|k|\le M}f_2(x+P(k))f_3(x+2P(k)) - \sum_{d\in \mathbf{Z}}f_2(z+d)f_3(z+2d)\nu(d)),\\
D^{2,\ast}(f_3,f_1)(x) &= N^{-1}((NW)^{1/2}\sum_{|k|\le M}f_1(x-P(k))f_3(x+P(k)) - \sum_{d\in \mathbf{Z}}f_1(z-d)f_3(z+d)\nu(d)),\\
D^{3,\ast}(f_1,f_2)(x) &= N^{-1}((NW)^{1/2}\sum_{|k|\le M}f_1(x-2P(k))f_2(x-P(k)) - \sum_{d\in \mathbf{Z}}f_1(z-2d)f_2(z-d)\nu(d)).
\end{align*}
By an application of the Cauchy--Schwarz inequality analogous to that used in \cref{prop:degree-lower-output} (and at the end of \cref{sec:notation}), we have
\[\bigg|\wt{\Lambda}(D^{1,\ast}(f_2,f_3),f_2,f_3)\bigg| \gg \delta^2N^2.\]
By the triangle inequality, we have that 
\[\bigg|\Lambda^{\mr{Model}}(D^{1,\ast}(f_2,f_3),f_2,f_3)\bigg|\gg \delta^2N^2\quad\text{or}\quad\bigg|(NW)^{1/2}\Lambda^{W}(D^{1,\ast}(f_2,f_3),f_2,f_3)\bigg|\gg \delta^2N^2.\]
Therefore, by \cref{lem:model-control}, \cref{prop:degree-lower-output}, and \cref{lem:mod-up}, we have 
\[\snorm{D^{1,\ast}(f_2,f_3)}_{U^{2}_{W\cdot[ N/W]}}^{4}\gg \exp^{K_{\ref{prop:degree-lower-output}}}(\delta^{-O(1)})N.\]
Applying the $U^2$-inverse theorem \cref{lem:U2-inver} to each progression of spacing $W$ and passing to the interval $[\pm C_1 N]$ (which contains the support of $D^{1,\ast}(f_2,f_3)$), there exist constants $\alpha_{j,1},\beta_{j,1}$ for each $j\in [W]$ such that 
\[\wt{f}_1(x) := \mbm{1}_{x\in [\pm C_1N]}\sum_{j\in[W]}\mbm{1}_{W\mid(x-j)}e(\alpha_{j,1}x + \beta_{j,1})\]
satisfies
\[\sum_{x\in \mathbf{Z}} \wt{f}_1(x) D^{1,\ast}(f_2,f_3)(x)\gg \exp^{K_{\ref{prop:degree-lower-output}}}(\delta^{-O(1)})^{-1}N.\]
By construction, this implies that 
\[|\wt{\Lambda}(\wt{f}_1,f_2,f_3)|\ge \exp^{K_{\ref{prop:degree-lower-output}}}(\delta^{-O(1)})^{-1}N^2.\]
Repeating this procedure, we find $\alpha_{j,i},\beta_{j,i}$ for each $j\in [W]$ such that defining
\[\wt{f}_i(x) := \mbm{1}_{x\in[\pm C_i N]}\sum_{j\in [W]}\mbm{1}_{W\mid(x-j)}e(\alpha_{j,i}x + \beta_{j,i})\]
for $i\in\{2,3\}$ (with appropriate absolute constants $C_i$), we have 
\[|\wt{\Lambda}(\wt{f}_1,\wt{f}_2,\wt{f}_3)|\gg \exp^{3K_{\ref{prop:degree-lower-output}}}(\delta^{-O(1)})^{-1}N^2.\]
Unwinding the definition of $\Lambda^{W}$ and $\Lambda^{\mr{Model}}$ (recall $\nu$ is supported on $[1,N]$), we have 
\begin{align*}
&\bigg|\sum_{\substack{x\in \mathbf{Z}\\d\in[N]}}\left((NW)^{1/2}\wt{f}_1(x)\wt{f}_2(x+d)\wt{f}_3(x+2d)\mbm{1}_{d\in \{P(k):k\in \mathbf{Z}\}\cap [N]} - \wt{f}_1(x)\wt{f}_2(x+d)\wt{f}_3(x+2d)\nu(d)\right)\bigg|\\
&\qquad\gg\exp^{3K_{\ref{prop:degree-lower-output}}}(\delta^{-O(1)})^{-1}N^2.
\end{align*}
Define $\nu^{\ast}(d) = (NW)^{1/2}\cdot |\{d = P(k)\colon k\in \mathbf{Z}\text{ and } |k|\le M\}|$. Since $P$ is injective on $\mathbf{Z}$ so this set is only size $0$ or $1$. We have
\begin{align*}
&\bigg|\sum_{\substack{x\in \mathbf{Z}\\d\in \mathbf{Z}}}\wt{f}_1(x)\wt{f}_2(x+d)\wt{f}_3(x+2d)(\nu^{\ast}(d) - \nu(d))\bigg|\gg \exp^{3K_{\ref{prop:degree-lower-output}}}(\delta^{-O(1)})^{-1}N^2.
\end{align*}
Note that 
\begin{align*}
&\bigg|\sum_{\substack{x\in \mathbf{Z}\\d\in \mathbf{Z}}}\wt{f}_1(x)\wt{f}_2(x+d)\wt{f}_3(x+2d)(\nu^{\ast}(d) - \nu(d))\bigg|\\
&\le \sum_{k,\ell\in [W]}\bigg|\sum_{\substack{x\in \mathbf{Z}\\d\in \mathbf{Z}\\W\mid(x-\ell)\\W\mid(d-k)}}\wt{f}_1(x)\wt{f}_2(x+d)\wt{f}_3(x+2d)(\nu^{\ast}(d) - \nu(d))\bigg|\\
&\le \sum_{k,\ell\in [W]}\bigg|\sum_{\substack{x\in \mathbf{Z}\\d\in \mathbf{Z}}}\wt{f}_1(Wx + \ell)\wt{f}_2(W(x+d) + \ell + k )\wt{f}_3(W(x+d) + \ell + 2k)(\nu^{\ast}(Wd + k) - \nu(Wd + k))\bigg|\\
&\le W^2\sup_{\substack{\alpha_1,\alpha_2,\alpha_3\in \mathbf{T}\\k,\ell\in[W]}}\bigg|\sum_{\substack{x\in \mathbf{Z}\\d\in \mathbf{Z}}}e(\alpha_1x)\mbm{1}_{|Wx + \ell|\le C_1N}e(\alpha_2(x+d))\mbm{1}_{|W(x+d) + \ell + k|\le C_2N}\\
&\qquad\qquad\qquad\qquad e(\alpha_3(x+2d))\mbm{1}_{|W(x+2d) + \ell + 2k|\le C_3N}(\nu^{\ast}(Wd + k) - \nu(Wd + k))\bigg|.
\end{align*}
Letting $\tau_{i,\alpha_i}(x) = e(\alpha_ix)\mbm{1}_{|x|\le C_iNW^{-1}}$, we have 
\[\sup_{\substack{\alpha_1,\alpha_2,\alpha_3\in \mathbf{T}\\ k\in [W]}}\bigg|\sum_{\substack{x\in \mathbf{Z}\\d\in \mathbf{Z}}}\tau_{1,\alpha_1}(x)\tau_{2,\alpha_2}(x+d)\tau_{3,\alpha_3}(x+2d)(\nu^{\ast}(Wd + k) - \nu(Wd + k))\bigg|\gg\exp^{3K_{\ref{prop:degree-lower-output}}}(\delta^{-O(1)})^{-1}N^2W^{-2}. \]
We now take a Fourier transform. Defining $\wt{\nu}_k(d) = (\nu^{\ast}(Wd + k) - \nu(Wd + k))$ we have
\begin{align*}
\bigg|\sum_{\substack{x\in \mathbf{Z}\\d\in \mathbf{Z}}}&\tau_{1,\alpha_1}(x)\tau_{2,\alpha_2}(x+d)\tau_{3,\alpha_3}(x+2d)(\nu^{\ast}(Wd + k) - \nu(Wd + k))\bigg|\\
&=\bigg|\int_{\mathbf{T}^2}\wh{\tau_{1,\alpha_1}}(\Theta_1)\wh{\tau_{2,\alpha_2}}(\Theta_2)\wh{\tau_{3,\alpha_3}}(-\Theta_1-\Theta_2)\wh{\wt{\nu}}_k(\Theta_2)~d\Theta_1d\Theta_2\bigg|\\
&\le \snorm{\wh{\wt{\nu}}_k}_{\infty} \cdot \int_{\mathbf{T}^2}|\wh{\tau_{1,\alpha_1}}(\Theta_1)|\cdot |\wh{\tau_{2,\alpha_2}}(\Theta_2)|\cdot |\wh{\tau_{3,\alpha_3}}(-\Theta_1-\Theta_2)|~d\Theta_1d\Theta_2\\
&\le \snorm{\wh{\wt{\nu}}_k}_{\infty} \bigg(\int_{\mathbf{T}^2}|\wh{\tau_{1,\alpha_1}}(\Theta_1)|^{3/2}\cdot |\wh{\tau_{2,\alpha_2}}(\Theta_2)|^{3/2}~d\Theta_1d\Theta_2\bigg)^{1/3}\\
&\qquad\qquad\bigg(\int_{\mathbf{T}^2}|\wh{\tau_{1,\alpha_1}}(\Theta_1)|^{3/2}\cdot |\wh{\tau_{3,\alpha_3}}(-\Theta_1-\Theta_2)|^{3/2}~d\Theta_1d\Theta_2\bigg)^{1/3}\\
&\qquad\qquad\bigg(\int_{\mathbf{T}^2}|\wh{\tau_{2,\alpha_2}}(\Theta_2)|^{3/2}\cdot |\wh{\tau_{3,\alpha_3}}(-\Theta_1-\Theta_2)|^{3/2}~d\Theta_1d\Theta_2\bigg)^{1/3} \\
&=\snorm{\wh{\wt{\nu}}_k}_{\infty} \prod_{i\in [3]}\bigg(\int_{\mathbf{T}}|\wh{\tau_{i,\alpha_i}}(\Theta)|^{3/2}d\Theta\bigg)^{2/3}\ll\snorm{\wh{\wt{\nu}}_k}_{\infty}\cdot (N/W),
\end{align*}
where in the final line we have used standard fact that the $L^p$-norm of the Fourier transform of an interval of length $N$ is $\ll_p N^{(p-1)/p}$ for $p>1$. However, by \cref{lem:Linfinity}, we have 
\[\sup_{k\in [W]}\snorm{\wh{\wt{\nu}}_k}_{\infty}\ll \frac{N}{W}\cdot \frac{1}{\sqrt{w}}.\]
We have our desired contradiction if $w$ (i.e., $W$) is sufficiently large with respect to $\delta^{-1}$. 
\end{proof}

The main result now follows in a straightforward manner. 
\begin{proof}[{Proof of \cref{thm:main}}]
Let $S$ be a subset of density $\delta$ in $[N]$ and $W$ be a sufficiently large parameter to be chosen at the end of the proof. By the pigeonhole principle, there exists $j\in[4W]$ such that $S_j = S\cap (4W\mathbf{Z} + j)$ has size at least $\delta N/(4W)$. Set $S_{j}^{\ast} = (S_j-j)/(4W) \subseteq[N/(4W)]$. Note that, since $4W^2y^2+4Wy=(2Wy+1)^2-(2Wy+1)$, differences of the form $((2Wy+1)^2-1)/(4W) = Wy^2 + y$ in the set $S_{j}^{\ast}$ lift to differences of the form $z^2-1$ in $S$. By \cref{lemma:lower-bound}, we have 
\[\Lambda^{\mr{Model}}(\mbm{1}_{S_{j}^{\ast}},\mbm{1}_{S_{j}^{\ast}},\mbm{1}_{S_{j}^{\ast}})\ge \exp(-\log(2/\delta)^{O(1)})N^2W^{-2}.\]
Taking $N\ge W^{\Omega(1)}\exp^{K+1}(\delta^{-1})$ and $W\ge \exp^{K+1}(\delta^{-1})$ with $K = K_{\ref{prop:transference}}$, \cref{prop:transference} with $N$ replaced by $N/W$ (note this alters the value of $M$) and $\delta$ appropriately changed implies
\[\Lambda^W(\mbm{1}_{S_{j}^{\ast}},\mbm{1}_{S_{j}^{\ast}},\mbm{1}_{S_{j}^{\ast}})\ge\exp(-\log(2/\delta)^{O(1)})N^{3/2}W^{-2}.\]
However, if $S$ is free of nontrivial progressions of the form $x, x+y^2-1, x+2(y^2-1)$, we have
\[\Lambda^{W}(\mbm{1}_{S_{j}^{\ast}},\mbm{1}_{S_{j}^{\ast}},\mbm{1}_{S_{j}^{\ast}})\ll N/W.\]
Therefore, if $\delta \ge \log_{K+2}(N)^{-1}$, taking $N=W^{\Omega(1)}\exp^{K+1}(\delta^{-1})$ and $W=\exp^{K+1}(\delta^{-1})$, we obtain a nontrivial progression of the form $x,x+y^2-1,x+2(y^2-1)$ in $S$, as desired.
\end{proof}

\bibliographystyle{amsplain0-full.bst}
\bibliography{main.bib}

\appendix
\section{Conventions regarding nilsequences and effective equidistribution}\label{sec:nilmanifold}

We begin this appendix by giving the precise definition of the complexity of a nilmanifold; this definition is exactly as in \cite[Definition~6.1]{TT21}.
\begin{definition}\label{def:nilman}
Let $s\ge 1$ be an integer and let $K>0$. A \emph{filtered nilmanifold $G/\Gamma$ of degree $s$ and complexity at most $K$} consists of the following:
\begin{itemize}
    \item a nilpotent, connected, and simply connected Lie group $G$ of dimension $m$, which can be identified with its Lie algebra $\log G$ via the exponential map $\exp\colon\log G\to G$;
    \item a filtration $G_{\bullet} = (G_i)_{i\ge 0}$ of closed connected subgroups $G_i$ of $G$ with 
    \[G = G_0 = G_1\geqslant G_1\geqslant \cdots\geqslant G_s\geqslant G_{s+1} = \mr{Id}_G\]
    such that $[G_i,G_j]\subseteq G_{i+j}$ for all $i,j\ge 0$;
    \item a discrete cocompact subgroup $\Gamma$ of $G$; and
    \item a linear basis $\mathcal{X}=\{X_1,\ldots, X_{m}\}$ of $\log G$, known as a \emph{Mal'cev basis}.
\end{itemize}
We, furthermore, require that this data obeys the following conditions:
\begin{enumerate}
    \item for $1\le i,j\le m$, one has Lie algebra relations
    \[[X_i,X_j] = \sum_{i,j<k\le m}c_{ijk}X_k\]
    for rational numbers $c_{ijk}$ of height at most $K$;
    \item for each $1\le i\le s$, the Lie algebra $\log G_i$ is spanned by $\{X_j\colon m-\dim(G_i)<j\le m\}$; and
    \item the subgroup $\Gamma$ consists of all elements of the form $\exp(t_1X_1)\cdots\exp(t_{m}X_{m})$ with $t_i\in \mathbf{Z}$.
\end{enumerate}
\end{definition}
We note that the conditions imply $[G,G_s]=\mr{Id}_G$, i.e., $G_s$ is contained in the center of $G$ (commutes with every element).

Next, we will define polynomial sequences in filtered nilpotent groups. This concrete definition is equivalent (by~\cite[Lemma 6.7]{GT12}) to the one given in~\cite{GT12}.
\begin{definition}\label{def:polyseq}
We adopt the conventions of \cref{def:nilman}. Let $G$ be a filtered nilpotent group of degree $s$. A function $g\colon\mathbf{Z}\to G$ is a \emph{polynomial sequence} if there exist elements $g_i\in G_{i}$ for $i=0,\ldots,s$ such that
    \begin{equation*}
        g(n)=g_0g_1^{\binom{n}{1}}\cdots g_s^{\binom{n}{s}},
    \end{equation*}
    where $\binom{n}{i}=\frac{1}{i!}\prod_{j=0}^{i-1}(n-j)$, for all $n\in\mathbf{Z}$
\end{definition}
We will denote the set of polynomial sequences $g\colon\mathbf{Z}\to G$ relative to the filtration $G_\bullet$ of $G$ by $\on{Poly}(\mathbf{Z},G_\bullet)$. It turns out that $\on{Poly}(\mathbf{Z},G_\bullet)$ is a group under the natural multiplication of sequences--this is due to Lazard~\cite{Lazard54} and Leibman~\cite{Leibman98,Leibman02}.

We will also require the definition of rational points, sequences, and subgroups. 
\begin{definition}\label{def:rational}
We adopt the conventions of \cref{def:nilman}. We say that $\gamma\in G$ is \emph{$Q$-rational} if there exists an integer $0<r\le Q$ such that $\gamma^{r}\in \Gamma$. A \emph{$Q$-rational point in $G/\Gamma$} is any point of the form $\gamma \Gamma$ for some $\gamma\in G$ that is $Q$-rational. A sequence $(\gamma(n))_{n=1}^\infty$ in $G$ is \emph{$Q$-rational} if all elements in the sequence are $Q$-rational.

Finally, we say a closed connected subgroup $G'$ of $G$ is \emph{$Q$-rational relative to $\mc{X}$} if its Lie algebra $\mf{g}'$ is spanned by linear combinations of the form $\sum_{i\in[m]}a_iX_i$ with $a_1,\ldots,a_m\in\mathbf{Q}$ all of height at most $Q$.
\end{definition}

Now we can define Mal'cev coordinates, the explicit metrics on $G$ and $G/\Gamma$ used in our work, and the precise definition of the Lipschitz norm of functions on $G/\Gamma$. These definitions are exactly as in \cite[Appendix~A]{GT12}.
\begin{definition}\label{def:Lip}
We adopt the conventions of \cref{def:nilman}. Given a Mal'cev basis $\mc{X}$ and $g\in G$, there exists $(u_1,\ldots,u_m)\in \mathbf{R}^m$ such that 
\[g = \exp(u_1X_1)\cdots\exp(u_mX_m),\]
and we define the \emph{Mal'cev coordinates $\psi=\psi_{\mathcal{X}}\colon G\to\mathbf{R}^m$ for $g$ relative to $\mc{X}$} by
\[\psi(g) := (u_1,\ldots, u_m).\]
We then define a metric $d = d_{\mc{X}}$ on $G$ by
\[d(x,y) := \bigg\{\sum_{i=1}^{n}\min(|\psi(x_ix_{i+1}^{-1})|,|\psi(x_{i+1}x_{i}^{-1})|)\colon n\in\mathbf{N}, x_1,\ldots,x_{n+1}\in G, x_1 = x, x_{n+1} = y\bigg\},\]
where $|\cdot|$ denotes the $\ell^\infty$-norm on $\mathbf{R}^m$, and define a metric on $G/\Gamma$ by
\[d(x\Gamma,y\Gamma) = \inf_{\gamma,\gamma'\in\Gamma}d(x\gamma,y\gamma').\]
Furthermore, for any function $F\colon G/\Gamma\to\mathbf{C}$, we define 
\[\snorm{F}_{\mr{Lip}} := \snorm{F}_{\infty} + \on{sup}_{\substack{x,y\in G/\Gamma \\ x\neq y}}\frac{|F(x)-F(y)|}{d(x,y)}.\]
\end{definition}
We now define the notion of equidistribution of a sequence on $G/\Gamma$ which we will require.
\begin{definition}\label{def:equidistributed}
Given a length $N$, a sequence $(g(n)\Gamma)_{n\in\Gamma}$ is $\delta$-equidistributed if for all Lipschitz functions $F\colon G/\Gamma\to\mbf{C}$ we have that 
\[\bigg|\mathbf{E}_{n\in[N]}F(g(n)\Gamma) - \int_{G/\Gamma}F\bigg|\le \delta \snorm{F}_{\mr{Lip}}.\]
\end{definition}

We will require the notion of a horizontal character and the notion of a function $F$ having a vertical frequency; our definitions are exactly as in \cite[Definitions~1.5,~3.3,~3.4,~3.5]{GT12}.
\begin{definition}\label{def:characters}
Given a filtered nilmanifold $G/\Gamma$, the \emph{horizontal torus} is defined to be \[(G/\Gamma)_{\mr{ab}}:=G/[G,G]\Gamma.\] A \emph{horizontal character} is a continuous homomorphism $\eta\colon G\to\mathbf{T}$ that annihilates $\Gamma$; such characters may be equivalently viewed as characters on the horizontal torus. A horizontal character is \emph{nontrivial} if it is not identically zero. 

Furthermore, if the nilmanifold $G/\Gamma$ has degree $s$, the vertical torus is defined to be \[G_s/(G_s\cap \Gamma).\] A \emph{vertical character} is a continuous homomorphism $\xi\colon G_s\to\mathbf{T}$ that annihilates $\Gamma\cap G_s$. Setting $m_s = \dim G_s$, one may use the last $m_s$ coordinates of the Mal'cev coordinate map to identify $G_s$ and $G_s/(G_s\cap \Gamma)$ with $\mathbf{R}^{m_s}$ and $\mathbf{R}^{m_s}/\mathbf{Z}^{m_s}$, respectively. Thus, we may identify any vertical character $\xi$ with a unique $k\in \mathbf{Z}^{m_s}$ such that $\xi(x) = k\cdot x$ under this identification $G_s/(\Gamma\cap G_s) \cong \mathbf{R}^{m_s}/\mathbf{Z}^{m_s}$. We refer to $k$ as the \emph{frequency} of the character $\xi$, we write $|\xi| :=\snorm{k}_{\infty}$ to denote the magnitude of the frequency $\xi$, and say that a function $F\colon G/\Gamma\to\mathbf{C}$ \emph{has a vertical frequency $\xi$} if 
\[F(g_s\cdot x) = e(\xi(g_s))F(x)\]
for all $g_s\in G_s$ and $x\in G/\Gamma$.
\end{definition}

Finally, we will require the definition of the smoothness norm of a polynomial sequence $\mathbf{Z}\to\mathbf{T}$.
\begin{definition}\label{def:smooth}
Any polynomial sequence $g\colon\mathbf{Z}\to\mathbf{T}$ can be expressed uniquely as \[g(n) = \sum_{i=0}^{d}\alpha_i\binom{n}{i}\]
with $\alpha_0,\ldots,\alpha_d\in\mathbf{T}$~\cite[Exercise 1.6.11]{Tao12}.
We then define 
\[\snorm{g}_{C^{\infty}[N]} := \max_{1\le j\le d}N^j\snorm{\alpha_j}_{\mathbf{T}}.\]
\end{definition}

We will need the fact that any Lipschitz function of a nilsequence can be well-approximated by a sum of vertical characters. The statement we require is, essentially, \cite[Lemma~A.6]{Len23}; our proof closely follows \cite[Lemma~3.7]{GT12}, given a sufficiently explicit estimate for approximating functions on the torus as a sum of characters. We provide a proof below, as the statement in \cite[Lemma~A.6]{Len23} has several typos. To give the proof, we require a version of Fourier expansion on the torus, which we will obtain by quantifying the proof of \cite[Proposition~1.1.13]{Tao12} (or \cite[Lemma~A.9]{GT08}).
\begin{lemma}\label{lem:torus-expand}
Fix $0<\eps<1/2$, and let $F\colon\mathbf{T}^d\to\mathbf{C}$ with $\snorm{F}_{\mr{Lip}}\le L$, where, for $x,y\in \mathbf{T}^d$, we have $d(x,y) = \max_{1\le i\le d}\snorm{x_i-y_i}_{\mathbf{T}}$. There exists an absolute constant $C = C_{\ref{lem:torus-expand}}>0$ such that we can write
\[F(x) = \sum_{|\xi|\le (CLd\eps^{-1})^{2}}c_{\xi}e(\xi\cdot x) + \wt{F}(x)\]
for a choice of $\wt{F},c_\xi$ with $\snorm{\wt{F}}_{\infty}\le \eps$ and $\sum_{\xi}|c_{\xi}|\le (3CLd\eps^{-1})^{5d}$.
\end{lemma}
\begin{proof}
Let $R\ge 1$ be a integer cutoff parameter to be chosen later and define 
\[F_R(x) := \sum_{k\in \mathbf{Z}^d}R^d\mu_R(k)\wh{F}(k)e(k\cdot x),\]
recalling the definition~\eqref{eq:fejer} of $\mu_R$. It is a basic fact from Fourier analysis that
\[F_R(x) = \int_{\mathbf{T}^d}F(y)K_{R}(x-y)dy,\]
where 
\[K_R(y) = \prod_{i=1}^{d}\frac{1}{R}\bigg(\frac{\sin(\pi R y_i)}{\sin(\pi y_i)}\bigg)^2=\prod_{i=1}^d\left(\sum_{|h|\le R}\left(1-\frac{|h|}{R}\right)e(hy_i)\right).\]
Noting that $\int_{\mathbf{T}^d}K_R(y)dy=1$,
\[
\frac{1}{R}\left(\frac{\sin(\pi R y_i)}{\sin(\pi y_i)}\right)^{2}\le C_0 \frac{R}{(1 + R\snorm{y_i}_{\mathbf{T}})^{2}}
\]
for some absolute constant $C_0>0$,
and $F$ is $L$-Lipschitz, we get
\begin{align*}
\snorm{F-F_{R}}_{\infty}&\le \sup_{x\in \mathbf{T}^d}\int_{\mathbf{T}^d}|F(x)-F(y)|\cdot K_{R}(x-y)~dy\\
&\le L\sup_{x\in \mathbf{T}^d}\int_{\mathbf{T}^d}\max_{1\le i\le d}\snorm{x_i-y_i}_{\mathbf{T}}\cdot K_{R}(x-y)~dy\\
&\le L\sup_{x\in \mathbf{T}^d}\int_{\mathbf{T}^d}\sum_{1\le i\le d}\snorm{x_i-y_i}_{\mathbf{T}}\cdot K_{R}(x-y)~dy\\
&\leq Ld\int_{\mathbf{T}}\|y\|_{\mathbf{T}}\cdot \frac{1}{R}\bigg(\frac{\sin(\pi R y)}{\sin(\pi y)}\bigg)^2~dy\leq C \frac{Ld}{\sqrt{R}}
\end{align*}
for some absolute constant $C>0$. The result follows by taking $R=\lceil(CLd\eps^{-1})^2\rceil$, noting that $\snorm{\wh{F}}_{\infty}\le L$, $\mu_{R}(k)\le R^{-d}$, and that $\mu_{R}(k)$ is supported on $k\in \mathbf{Z}^d$ such that $\snorm{k}_{\infty}\le R$.
\end{proof}

We now extend this result to general filtered nilmanifolds by using Fourier analysis on the final nontrivial group of the filtration, $G_s$.
\begin{lemma}\label{lem:vertical-expansion}
Fix $0<\eps<1/2$, let $G/\Gamma$ be a filtered nilmanifold of dimension $m$, degree $s$, and complexity at most $K$, and let $F\colon G/\Gamma\to\mathbf{C}$ satisfy $\snorm{F}_{\mr{Lip}}\le L$. Then one may represent 
\[F(x) = \sum_{|\xi|\le \on{poly}_{m}(LK\eps^{-1})} F_{\xi}(x) + G(x),\]
with
\begin{enumerate}
    \item $\snorm{G}_{\infty}\le \eps$;
    \item $F_{\xi}$ has vertical frequency $\xi$; and
    \item $F_{\xi}$ has Lipschitz norm bounded by $\on{poly}_{m}(LK\eps^{-1})$.
\end{enumerate}
\end{lemma}
\begin{remark*}
The bounds in this specific lemma could likely be substantially improved with a more careful treatment.
\end{remark*}
\begin{proof}
The proof follows exactly as in \cite[Lemma~3.7]{GT12} (with the quantification as suggested by \cite{TT21}) so we will be brief with details. Let $R\ge1$ be an integer cutoff and let $K_R$ denote the same kernel as in the proof of \cref{lem:torus-expand}. Define 
\[F_1(y) :=\int_{\mathbf{R}^{m_s}/\mathbf{Z}^{m_s}}F(\Theta y)K_R(\Theta)~d\Theta\]
where we have identified the last group in the filtration with $\mathbf{R}^{m_s}/\mathbf{Z}^{m_s}$ for the appropriate integer $m_s\le m$ (and therefore $\Theta y$ makes sense for $y\in G/\Gamma$, explicitly defined as $\psi_{\mc{X}}^{-1}(\Theta^\ast)y$ where $\Theta^\ast$ is $0$ in the first $m-m_s$ coordinates and $\Theta$ in the final $m_s$). Fourier expansion in $\mathbf{R}^{m_s}/\mathbf{Z}^{m_s}$ gives that 
\[F_1(y) :=\sum_{k\in \mathbf{Z}^{m_s}}F^{\wedge}(y;k)(R^{m_s}\mu_{R}(k))\]
where 
\[F^{\wedge}(y;k) := \int_{\mathbf{R}^{m_s}/\mathbf{Z}^{m_s}}F(\Theta y)e(-k\cdot \Theta)~d\Theta.\]
The estimates from \cref{lem:torus-expand} now complete the proof, noting that metric on $G/\Gamma$ when descended to the torus is $\on{poly}_{m}(K)$-equivalent to the standard metric on $\mathbf{R}^{m_s}/\mathbf{Z}^{m_s}$.
\end{proof}

\section{Circle method estimates}\label{sec:circle}
The material within this appendix consists of standard circle method computations, aside from proving an $L^{\infty}$-comparison estimate between certain $W$-tricked quadratic Gauss sums and the Fourier transform of an interval. This comparison is essentially contained within the work of Browning and Prendiville \cite{BP17}.

\subsection{\texorpdfstring{$L^6$}{L6}-bound on the Fourier transform}
We first require a log-free variant of Weyl's inequality (see \cite[Lemma~A.11]{GT08}).
\begin{lemma}[Weyl's inequality]\label{lem:Weyl}
There exists an absolute constant $C = C_{\ref{lem:Weyl}}>0$ such that the following holds. Let $\alpha,\beta\in \mathbf{T}$, $\delta\in (0,1)$, and let $I$ be an interval in $\mathbf{Z}$. If 
\[\bigg|\sum_{y\in I} e(\alpha y^2+\beta y)\bigg|\ge \delta|I|,\]
then either $|I|\le C\delta^{-C}$ or there is a positive integer $q\le C\delta^{-C}$ such that
\[\snorm{q\alpha}\le C\delta^{-C}|I|^{-2}.\]
\end{lemma}

We next require the following basic estimate regarding exponential sum estimates, which is based on \cite[Chapter~4]{Nat96}.
\begin{lemma}\label{lem:L4-bound}
Let $W$ and $P$ be as in \eqref{eq:Wdef} with $W\le N$ and, for $r\in [W]$, define 
\[P_r(y) = \frac{P(Wy + r) - P(r)}{W}.\]
We have that
\[\sup_{r\in [W]}\int_{0}^1\bigg|\sum_{x\in [\pm N]}e(\Theta P_r(x))\bigg|^4~d\Theta\ll N^2 \cdot \exp\bigg(O\bigg(\frac{\log N}{\log\log N}\bigg)\bigg).\]
\end{lemma}
\begin{proof}
Fix $r\in [W]$; the proof will trivially give a bound uniform in $r$. Let $s(d) = \{(x,y)\in J\times J\colon P_r(x)-P_r(y) = d\}$. Since $|P_r(x) - P_r(y)|\le N^4$ (say) for $x,y\in[\pm N]$, we have $s(d) = 0$ for $|d|\ge N^4$. Furthermore, note that $P_r(x)-P_r(y) = (x-y)(W^2(x+y) + (2Wr+1))$. Therefore, $s(0)\le 2(2N+1)$ and, for $d\neq 0$, the divisor bound implies 
\[|s(d)|\le \exp\bigg(O\bigg(\frac{\log N}{\log\log N}\bigg)\bigg).\]
By definition,
\[\sum_{d\in\mathbf{Z}}s(d) = (2N+1)^2.\] Therefore,
\begin{align*}
\int_{0}^1\bigg|\sum_{x\in [\pm N]}e(\Theta P_r(x))\bigg|^4~d\Theta &= \int_{0}^1\bigg(\sum_{y\in \mathbf{Z}}s(y)e(y\Theta)\bigg)^2~d\Theta=\sum_{d\in \mathbf{Z}}s(d)s(-d)\\
&\le s(0)^2 + \Big(\max_{d\in \mathbf{Z}\setminus\{0\}}s(d)\Big)\sum_{d\in\mathbf{Z}}s(d)\ll N^2\cdot \exp\bigg(O\bigg(\frac{\log N}{\log\log N}\bigg)\bigg).\qedhere
\end{align*}
\end{proof}

We next record various basic properties of (generalized composite) Gauss sums. Several of these properties are recorded in \cite[Exercise~12,23]{BEW98}.
\begin{lemma}\label{lem:gauss}
Define
\[G(a,b,c) = \sum_{n=0}^{c-1}e\bigg(\frac{an^2+bn}{c}\bigg).\]
We have the following set of properties:
\begin{itemize}
    \item If $\gcd(c,d) = 1$ then \[G(a,b,cd) = G(ac,b,d)G(ad,b,c);\]
    \item If $\gcd(a,c) > 1$ then $G(a,b,c) = 0$ unless $\gcd(a,c)|b$. In this case, it follows that 
    \[G(a,b,c) = G\bigg(\frac{a}{\gcd(a,c)},\frac{b}{\gcd(a,c)},\frac{c}{\gcd(a,c)}\bigg);\]
    \item If $\gcd(a,c) = 1$ and $\gcd(c,2) = 1$ then 
    \[|G(a,b,c)|=\sqrt{c};\]
    \item If $c=2^k$ for $k\ge 1$, $\gcd(a,c)=1$, and $b$ is even then
    \[|G(a,b,c)|\le 2\sqrt{c}.\]
\end{itemize}
\end{lemma}
Note that these relations can be used to determine a bound on the magnitude of any composite Gauss sum. We can use the first relation to decompose into prime power moduli $c$ and the second relation to reduce to $\gcd(a,c)=1$; the third deals with odd prime powers and the final one with even prime powers.

Now for the remainder of this appendix, we say $\Theta$ is in the major arcs $\mf{M}$ if there exists $0\le q_1<q_2\le N^{\eps}$ such that 
\[\bigg|\Theta - \frac{q_1}{q_2}\bigg|\le N^{-2+\eps}\] for a small constant $\eps > 0$ to be chosen later. Set $\mf{m} = \mathbf{T}\setminus \mf{M}$ to be the minor arcs.
\begin{lemma}\label{lem:minor-arc-bound}
There exists $\eps_{\ref{lem:minor-arc-bound}} > 0$ such that the following holds. Let $0<\eps < \eps_{\ref{lem:minor-arc-bound}}$ and $0<\delta \le \delta_{\ref{lem:minor-arc-bound}}(\eps)$. Furthermore, suppose that $W$ and $P_r(\cdot)$ are as in \cref{lem:L4-bound} and with $W\le N^{\delta}$. Then, we have
\[\sup_{r\in [W]}\int_{\mf{m}}\bigg|\sum_{x\in[\pm N]}e(\Theta P_r(x))\bigg|^{6}~d\Theta \ll_{\eps,\delta} N^{4-\delta}\]
and 
\[\sup_{\substack{\Theta\in \mf{m}\\ r\in [W]}}\bigg|\sum_{x\in [\pm T]}e(\Theta P_r(x))\bigg|\ll_{\eps,\delta}N^{1-\delta}.\]
\end{lemma}
\begin{proof}
We take $0<\delta \ll \eps$ to be chosen later. Let $f(\Theta)$ be the expression inside the supremum. By \cref{lem:Weyl}, for $\Theta\in\mf{m}$ we have $|f(\Theta)|\ll_{\eps,\delta}N^{1-\delta}$. Indeed, if not then we must have $\snorm{qW^2\alpha}\ll N^{C\delta}N^{-2}$ for some $q\ll N^{C\delta}$ (and appropriate $C$), noting that the first coefficient of $P_r$ is $W^2$. If $\delta$ is small enough, this violates the definition of the minor arcs. This proves the second desired inequality.

For the first, applying \cref{lem:L4-bound} and using the above bound we find
\begin{align*}
\int_{\mf{m}}|f(\Theta)|^{6}~d\Theta &\le \int_{\mathbf{T}}|f(\Theta)|^{4}~d\Theta \cdot \sup_{\Theta\in \mf{m}}|f(\Theta)|^2\\
&\ll_{\eps,\delta} N^2\cdot \exp\bigg(O\bigg(\frac{\log N}{\log\log N}\bigg)\bigg) \cdot (N^{1-\delta})^2\\
&\ll N^{4-\delta/2}.\qedhere
\end{align*}
\end{proof}

We now handle the major arcs. Note that, without loss of generality, either $(q_1,q_2) = (0,1)$ or $\gcd(q_1,q_2) = 1$ and $1\le q_1<q_2\le N^{\eps}$. Furthermore, given that $\eps<1/4$ (say), the arcs
\[\mf{M}_{q_1,q_2} := \bigg\{\Theta\colon\bigg|\Theta - \frac{q_1}{q_2}\bigg|\le N^{-2+\eps}\bigg\}\]
for such $(q_1,q_2)$ are disjoint. We now state the major arc asymptotic for exponential sums of $P_r(x)$; as this material is completely standard, we omit the proof. 

\begin{lemma}\label{lem:major-arc-asym}
There exists $\eps = \eps_{\ref{lem:major-arc-asym}}$ such that the following holds. If $0<\eps<\eps_{\ref{lem:major-arc-asym}}$, $\Theta\in \mf{M}_{q_1,q_2}$, and $W$, $P_r(\cdot)$ are as in \cref{lem:L4-bound} with $W\le N^{\eps}$, and $\Theta^{\ast} = \Theta-\frac{q_1}{q_2}$, then 
\[\sum_{x\in[\pm N]}e(\Theta P_r(x)) = q_2^{-1}G(W^2q_1,(2Wr+1)q_1,q_2) \int_{-N}^N e(\Theta^{\ast}\cdot W^2x^2)~dx +O\left(N^{1/2}\right).\]
\end{lemma}

We also need the following elementary fact proven via integration by parts.  
\begin{lemma}\label{lem:interval}
We have
\begin{equation*}
\bigg|\int_{-\gamma}^\gamma e(x^2)~dx\bigg|\ll\min(|\gamma|, 1).
\end{equation*}
\end{lemma}
\begin{proof}
By negation symmetry and the triangle inequality it suffices to assume that $\gamma \ge 2$. Now
\begin{align*}
\bigg|\int_{-\gamma}^{\gamma}e(x^2)~dx\bigg| &\le  \bigg|\int_{-1}^{1}e(x^2)~dx\bigg| + 2\bigg|\int_{1}^{\gamma}e(x^2)~dx\bigg|\\
&\le 2 + 4 \sup_{t\ge 1}\bigg|\int_{x\ge t}e(x^2)~dx\bigg|= 2 + 4 \sup_{t\ge 1}\bigg|\int_{t^2}^{\infty}\frac{e(x)}{2\sqrt{x}}~dx\bigg|\\
&\le 2 + 4 \sup_{t\ge 1}\bigg|\frac{e(t^2)}{4\pi it}\bigg| + 4\bigg|\int_{t^2}^{\infty}\frac{e(x)}{8\pi i x^{3/2}}~dx\bigg| \ll 1.\qedhere
\end{align*}
\end{proof}

We now in position to derive the necessary $L^{6}$-bound; this is essentially an exercise in bounding certain integrals and quadratic Gauss sums. 
\begin{lemma}\label{lem:L6-bound}
There exists $\eps_{\ref{lem:L6-bound}} > 0$ such that the following holds. Let $0<\eps < \eps_{\ref{lem:L6-bound}}$ and $0<\delta \le \delta_{\ref{lem:L6-bound}}(\eps)$. Furthermore, suppose that $W$ and $P_r(\cdot)$ are as in \cref{lem:L4-bound} and with $W\le N^{\delta}$. We have 
\[\sup_{r\in [W]}\int_{\mathbf{T}}\bigg|\sum_{x\in [\pm N]}e(\Theta P_r(x))\bigg|^{6}~d\Theta \ll N^{4}W^{-2},\]
where the implied constant is absolute. 
\end{lemma}
\begin{proof}
Choosing $\eps_{\ref{lem:L6-bound}}$ sufficiently small and using \cref{lem:minor-arc-bound,lem:major-arc-asym}, it suffices to prove that 
\[\bigg(G(0,0,1)^6 + \sum_{\substack{1\le q_1<q_2\le n^{\eps}\\\gcd(q_1,q_2) = 1}}(q_2^{-1}G(W^2q_1,(2Wr+1)q_1,q_2))^6\bigg)\cdot \int_{-N^{-2+\eps}}^{N^{-2+\eps}}\bigg|\int_{-N}^Ne(\Theta^{\ast}\cdot W^2x^2)~dx\bigg|^6d\Theta^{\ast}\]
is bounded as in the statement of the lemma. Note that, as $\gcd(W^2,2Wr+1) = 1$ and $\gcd(q_1,q_2) = 1$, we have $|G(W^2q_1,(2Wr+1)q_1,q_2)|\le 4\sqrt{q_2}$ using \cref{lem:gauss}, and thus the first term is seen to be bounded by a constant. The required bound on the other term, the integral, follows from \cref{lem:interval}: we obtain after change of variables the bound
\[\bigg|\int_{-N}^Ne(\Theta^\ast W^2x^2)~dx\bigg|\ll\min\{N,(\Theta^\ast)^{-1/2}W^{-1}\}\]
whose $6$-th power integrates to $O(N^4W^{-2})$, as desired.
\end{proof}

\subsection{\texorpdfstring{$L^{\infty}$}{Linfinity}-comparison estimate}
\begin{lemma}\label{lem:Linfinity}
There exists $\eps = \eps_{\ref{lem:Linfinity}}>0$ such that the following holds. Let $N\ge 1$ and $W, w$ be as in \eqref{eq:Wdef} with $|W|\le N^{\eps}$. Furthermore, define
\[\nu^{\ast}(d) = (NW)^{1/2}\cdot \mbm{1}[d\in \{P(k)\colon k\in \mathbf{Z}\}\cap [N]]\] and
\[\nu(d) = \sqrt{\frac{N}{d}}\mbm{1}_{1\le d\le N}.\]
Then, we have 
\[\sup_{\substack{\Theta\in\mathbf{T}\\k\in [W]}}\bigg|\sum_{d\in \mathbf{Z}}e(d\Theta)(\nu^{\ast}(Wd + k)-\nu(Wd + k))\bigg|\ll \frac{N}{W\sqrt{w}}.\]
\end{lemma}
\begin{proof}
Unwinding the definitions and noting that $W$ is sufficiently small, it suffices to prove that
\begin{align*}
&\sup_{\substack{\Theta\in\mathbf{T}\\k\in [W]}}\bigg|\sum_{|d|\le N^{1/2}W^{-3/2}}(NW)^{1/2}e(P(Wd+k^{\ast})\Theta)-\sum_{1\le d\le NW^{-1}}\sqrt{\frac{N}{Wd}}e((Wd+k)\Theta)\bigg|\ll \frac{N}{W\sqrt{w}}
\end{align*}
with $k^{\ast}\in [W]$ the unique choice (due to \cref{prop:Hensel}) such that $P(k^{\ast}) \equiv k \mod W$. Replacing $N$ by $NW$, it suffices to prove that 
\[\sup_{\substack{\Theta\in \mathbf{Z}\\k\in [W]}}\bigg|\sum_{|d|\le N^{1/2}W^{-1}}N^{1/2}We(P(Wd+k^{\ast})\Theta)-\sum_{1\le d\le N}\sqrt{\frac{N}{d}}e((Wd+k)\Theta)\bigg|\ll \frac{N}{\sqrt{w}}.\]

Consider potential $\Theta,k$ which do not satisfy this. We first consider the second summation. By summing a geometric series, we find that 
\begin{align*}
\bigg|\sum_{1\le d\le N}\sqrt{\frac{N}{d}}\exp((Wd+k)\Theta)\bigg| &= \bigg|\sum_{1\le d\le N}\sqrt{\frac{N}{d}}\exp((Wd)\Theta)\bigg|\\
&\ll\sum_{1\le t\le N-1}\bigg(\sqrt{\frac{N}{t}}-\sqrt{\frac{N}{t+1}}\bigg)\bigg|\sum_{1\le d\le t}e(dW\Theta)\bigg|  + \bigg|\sum_{1\le d\le N}\exp(dW\Theta)\bigg|\\
&\ll\frac{1}{\snorm{W\Theta}_{\mathbf{T}}}+ \sum_{1\le t\le N}N^{1/2}t^{-3/2}\min\bigg(t,\frac{1}{\snorm{W\Theta}_{\mathbf{T}}}\bigg)\\
&\ll\frac{1}{\snorm{W\Theta}_{\mathbf{T}}}+ \frac{N^{1/2}}{\snorm{W\Theta}_{\mathbf{T}}^{1/2}}.
\end{align*}
This is smaller than $\frac{N}{\sqrt{w}}$ unless $\snorm{W\Theta}_{\mathbf{T}}\le w/N$. Now we consider the first summation. Using a version of Weyl's inequality that accounts for the linear coefficient (see, e.g., \cite[Proposition~4.3]{GT12}), we have that if
\[\sum_{|d|\le N^{1/2}W^{-1}}N^{1/2}We(P(Wd+k^{\ast})\Theta)\gg \frac{N}{\sqrt{w}},\]
then there exists $q\le w^{O(1)}$ such that 
\begin{equation*}
\snorm{q\cdot W^3\Theta}_{\mathbf{T}}\le \frac{w^{O(1)}W^{2}}{N},\qquad
\snorm{q\cdot (2W^2k^{\ast} + W)\Theta}_{\mathbf{T}}\le \frac{w^{O(1)}W}{N^{1/2}}.
\end{equation*}
As $W$ is sufficiently small and $\gcd(2Wk^{\ast} + 1, W) = 1$, we can see that it suffices to handle $\Theta$ such that there exists $q'\le w^{O(1)}$ for which
\[\snorm{q'\cdot W\Theta}_{\mathbf{T}}\le \frac{w^{O(1)}}{N}.\]
Indeed, one can use argumentation similar to that appearing in the proof of \cref{prop:Malcev-bash}. (Note that both the first and second sums being large implies this condition.)

For such $\Theta$, we may write $W\Theta = \frac{q_1}{q_2} + \Theta^{\ast}$ with $|\Theta^{\ast}|\le\frac{w^{O(1)}}{N}$ where $q_2\le w^{O(1)}$ and $\gcd(q_1,q_2)=1$. We now have
\begin{align*}
&\sum_{|d|\le N^{1/2}W^{-1}}N^{1/2}We(P(Wd+k^{\ast})\Theta)\\
&=\exp(k\Theta)\cdot \sum_{|d|\le N^{1/2}W^{-1}}N^{1/2}We\bigg((W\Theta)\bigg(\frac{P(Wd + k^{\ast})-P(k^{\ast})}{W}\bigg) + \frac{P(k^{\ast}) - k}{W}\cdot W\Theta\bigg)\\
&=\exp\bigg(k\Theta + \frac{q_1\cdot (P(k^{\ast})-P(k))}{q_2W}\bigg) \sum_{|d|\le N^{1/2}W^{-1}}N^{1/2}We((W\Theta)(W^2d^2 + (2Wk^{\ast}+1)d)) +O\left(N^{1/2}\right).
\end{align*}

Now using the major arc bounds in \cref{lem:major-arc-asym}, and noting that $q_2^{-1}G(W^2q_1,(2Wk^{\ast}+1)q_1,q_2)=0$ if $\gcd(q_2,W)>1$, the above sum is bounded by $N^{1-\Omega(1)}$ in this case. Furthermore, if $q_2\neq 1$ and $\gcd(q_2,W)=1$, then $q_2$ has a prime factor larger than $w$, and therefore $q_2^{-1}G(W^2q_1, (2Wk^{\ast}+1)q_1,q_2)\ll w^{-1/2}$ and the sum becomes bounded by $O((N^{1/2}W)(N^{1/2}W^{-1})w^{-1/2})=O(Nw^{-1/2})$.

Thus, it suffices to focus on the case where $q_2=1$ and thus $q_1=0$. Note that
\[\sup_{|\Theta^{\ast}|\le \frac{w^{O(1)}}{N}}\bigg|\sum_{|d|\le N^{1/2}W^{-1}}N^{1/2}W\bigg(e(\Theta^\ast(W^2d^2 + (2Wk^{\ast}+1)d)) - e(\Theta^\ast(W^2d^2))\bigg)\bigg|\le N^{3/4}.\]

Now that we have significantly reduced our initial situation of general $\Theta,k$. To prove the lemma, it now simply suffices to prove that
\begin{align*}
\sup_{|\Theta^{\ast}|\le\frac{w^{O(1)}}{N}}\bigg|\sum_{|d|\le N^{1/2}W^{-1}}N^{1/2}We(\Theta^\ast(W^2d^2)) - \sum_{1\le d\le N}\sqrt{\frac{N}{d}}e(d\Theta^\ast)\bigg|\ll \frac{N}{\sqrt{w}}.
\end{align*}
For $|\Theta^{\ast}|\le w^{O(1)}/N$, we have
\begin{align*}
&\bigg|\sum_{|d|\le N^{1/2}W^{-1}}N^{1/2}We(\Theta^\ast(W^2d^2)) - \sum_{1\le d\le N}\sqrt{\frac{N}{d}}e(d\Theta^\ast)\bigg|\\
&\le\bigg|\sum_{1\le d \le N^{1/2}W^{-1}}2N^{1/2}We(\Theta^\ast(W^2d^2)) - \sum_{1\le d\le N}\sqrt{\frac{N}{d}}e(d\Theta^\ast)\bigg| + N^{3/4}\\
&\le\sum_{1\le d \le N^{1/2}W^{-1}}\bigg|2N^{1/2}We(\Theta^\ast(W^2d^2)) - \sum_{W^2d^2\le t\le W^2d^2 + 2W^2d}\sqrt{\frac{N}{t}}e(t\Theta^\ast)\bigg| + O\left(N^{4/5}\right)\\
&\le\sum_{1\le d \le N^{1/2}W^{-1}}\bigg|2N^{1/2}We(\Theta^\ast(W^2d^2)) - \sum_{W^2d^2\le t\le W^2d^2 + 2W^2d}\sqrt{\frac{N}{W^2d^2}}e(W^2d^2\Theta^\ast)\bigg| +O\left( N^{4/5}\right)\\
&\le\sum_{1\le d \le N^{1/2}W^{-1}}\bigg|2N^{1/2}W-(2W^2d+1)\sqrt{\frac{N}{W^2d^2}}\bigg| +O\left( N^{4/5}\right)\ll N^{4/5},
\end{align*}
as desired.
\end{proof}

\section{Miscellaneous estimates}\label{sec:random}
In this appendix, we prove a variety of miscellaneous estimates largely concerning changing the parameters of the $U^k$-norms. We first require the following elementary inequality. 
\begin{fact}\label{ft:trivial}
For all positive integers $k\ge 1$, we have that
\[|\sin(kx)|\le k|\sin x|.\]
\end{fact}
\begin{proof}
We proceed by induction; the result is trivial for $k = 1$. For the inductive step, note that
\[|\sin((j+1)x)|\le |\sin(jx)\cos x + \sin x\cos jx|\le |\sin(jx)| + |\sin(x)|\le (j+1)|\sin x|.\qedhere\]
\end{proof}

We now prove a lemma saying that the $U^k$-norms behave well with respect to rescaling the width. 

\begin{lemma}\label{lem:rescale-down}
Given a function $f\colon\mathbf{Z}\to\mathbf{C}$ with finite support, subsets of integers $Q_1,\ldots,Q_{k-1}$, and positive integers $L_1, L_2$ such that $L_2\mid L_1$, we have 
\[\snorm{f}_{\Box_{Q_1,\ldots,Q_{k-1},[L_1]}^{k}}^{2^k}\le \snorm{f}_{\Box_{Q_1,\ldots,Q_{k-1},[L_2]}^{k}}^{2^k}.\]
\end{lemma} 
\begin{proof}
We first reduce to the case $k = 1$. Consider expanding the box-norm; note that 
\[\snorm{f}_{\square_{\Box_{Q_1,\ldots,Q_{k-1},[L]}}^{k}}^{2^k} = \mbf{E}_{h_i,h_i'\in Q_i} \snorm{\Delta_{(h_i,h_i')_{i\in [k-1]}}'f}_{U^1_{[L]}}^{2}\]
for $L\in\{L_1,L_2\}$. Therefore it suffices to prove that
\[\sum_{x\in \mathbf{Z}}\mathbf{E}_{h,h'\in[L_1]}\Delta_{h,h'}'f(x)\le \sum_{x\in \mathbf{Z}}\mathbf{E}_{h,h'\in[L_2]}\Delta_{h,h'}'f(x).\]

Set $W_{L_i}(x) = \mbm{1}_{0<x\le L_i}/L_i$, and note that 
\begin{align*}
&\sum_{x\in \mathbf{Z}}\mathbf{E}_{h,h'\in[L_i]}\Delta_{h,h'} f(x)= \sum_{x,h,h'\in \mathbf{Z}}f(x+h)\ol{f(x+h')}W_{L_i}(h)W_{L_i}(h')\\
&= \int_{0}^1|\wh{f}(\Theta)|^2|\wh{W_{L_i}}(\Theta)|^2~d\Theta= \int_{-1/2}^{1/2}|\wh{f}(\Theta)|^2\bigg(\frac{\sin(L_i\pi\Theta)}{L_i\sin(\pi\Theta)}\bigg)^2~d\Theta.
\end{align*}
Applying \cref{ft:trivial} with $k = L_1/L_2$ yields
\begin{align*}
\int_{-1/2}^{1/2}|\wh{f}(\Theta)|^2\bigg(\frac{\sin(L_1\pi\Theta)}{L_1\sin(\pi\Theta)}\bigg)^2~d\Theta &\le \int_{-1/2}^{1/2}|\wh{f}(\Theta)|^2\bigg(\frac{\sin(L_2\pi\Theta)}{L_2\sin(\pi\Theta)}\bigg)^2~d\Theta,
\end{align*}
as desired. 
\end{proof}

We next prove the analogous inequality with respect to rescaling the difference parameters within the $U^k$-norm.

\begin{lemma}\label{lem:mod-up}
Given an integer $k\ge 1$, there exists $C_k = C_{\ref{lem:mod-up}}(k)>0$ such that the following holds. Given a $1$-bounded function $f\colon\mathbf{Z}\to\mathbf{C}$ such that $\on{supp}(f)\subseteq [\pm N]$, subsets of integers $Q_1,\ldots,Q_{k-1}$ each contained in $[\pm N]$, and positive integers $L_1$, $L_2$ such that $N\ge L_1\ge 2 L_2$, we have 
\[\snorm{f(x)}_{\Box_{Q_1,\ldots,Q_{k-1},[L_1]}^{k}}^{2^k}\le \snorm{f(x)}_{\Box_{Q_1,\ldots,Q_{k-1},L_2\cdot [L_1/L_2]}^{k}}^{2^k} + O\left(\frac{C_k N\cdot L_2}{L_1}\right)\]
\end{lemma} 
\begin{proof}
One can reduce to the case $k=1$ as in \cref{lem:rescale-down}. It, therefore, suffices to prove that
\[\sum_{x\in \mathbf{Z}}\mathbf{E}_{h,h'\in[L_1]}\Delta_{h,h'} f(x)\le \sum_{x\in \mathbf{Z}}\mathbf{E}_{h,h'\in L_2\cdot [L_1/L_2]}\Delta_{h,h'} f(x) + O\bigg(\frac{N\cdot L_2}{L_1}\bigg).\]
Via a direct Fourier-analytic computation, we have 
\begin{equation*}
\sum_{x\in \mathbf{Z}}\mathbf{E}_{h,h'\in[L_1]}\Delta_{h,h'}' f(x)= \int_{-1/2}^{1/2}|\wh{f}(\Theta)|^2\bigg(\frac{\sin(L_1\pi\Theta)}{L_1\sin(\pi\Theta)}\bigg)^2~d\Theta
\end{equation*}
and
\begin{equation*}
\sum_{x\in \mathbf{Z}}\mathbf{E}_{h,h'\in L_2\cdot [L_1/L_2]}\Delta_{h,h'}' f(x) = \int_{-1/2}^{1/2}|\wh{f}(\Theta)|^2\bigg(\frac{\sin(\lfloor L_1/L_2\rfloor \pi(L_2\Theta))}{\lfloor L_1/L_2\rfloor\sin(L_2\pi\Theta)}\bigg)^2~d\Theta.
\end{equation*}
Using \cref{ft:trivial} in the denominator, we have
\[
\bigg(\frac{\sin(L_1\pi\Theta)}{L_1\sin(\pi\Theta)}\bigg)^2\le\bigg(\frac{L_2\sin(L_1\pi\Theta)}{L_1\sin(L_2\pi\Theta)}\bigg)^2.\]
Next, note that 
\begin{align*}
\bigg|\bigg(\frac{L_2\sin(L_1\pi\Theta)}{L_1\sin(L_2\pi\Theta)}\bigg)^2&- \bigg(\frac{\sin(\lfloor L_1/L_2\rfloor \pi(L_2\Theta))}{\lfloor L_1/L_2\rfloor\sin(L_2\pi\Theta)}\bigg)^2\bigg|\\
&\le \frac{2L_2|\pi \Theta|}{|\sin(L_2\pi \Theta)|^2}\cdot \bigg|\frac{L_2\sin(L_1\pi\Theta)}{L_1}-\frac{\sin(\lfloor L_1/L_2\rfloor \pi(L_2\Theta))}{\lfloor L_1/L_2\rfloor}\bigg|\\
&\ll \frac{L_2|\pi \Theta|}{|\sin(L_2\pi \Theta)|^2}\cdot\bigg(\frac{L_2^2|\Theta|}{L_1}\bigg)\ll\frac{L_2^3|\Theta|^2}{L_1|\sin(L_2\pi\Theta)|^2}.
\end{align*}
Therefore, it follows that 
\begin{align*}
&\sum_{x\in \mathbf{Z}}\mathbf{E}_{h,h'\in[L_1]}\Delta_{h,h'}' f(x)= \int_{-1/2}^{1/2}|\wh{f}(\Theta)|^2\bigg(\frac{\sin(L_1\pi\Theta)}{L_1\sin(\pi\Theta)}\bigg)^2~d\Theta\\
&= \int_{-1/(4L_2)}^{1/(4L_2)}|\wh{f}(\Theta)|^2\bigg(\frac{\sin(L_1\pi\Theta)}{L_1\sin(\pi\Theta)}\bigg)^2~d\Theta + O\bigg(\frac{N\cdot L_2^2}{L_1^2}\bigg)\\
&\le \int_{-1/(4L_2)}^{1/(4L_2)}|\wh{f}(\Theta)|^2\bigg(\bigg(\frac{\sin(\lfloor L_1/L_2\rfloor \pi(L_2\Theta))}{\lfloor L_1/L_2\rfloor\sin(L_2\pi\Theta)}\bigg)^2 + O\bigg(\frac{L_2^3|\Theta|^2}{L_1|\sin(L_2\pi\Theta)|^2}\bigg)\bigg)~d\Theta+ O\bigg(\frac{N\cdot L_2^2}{L_1^2}\bigg)\\
&\le\sum_{x\in \mathbf{Z}}\mathbf{E}_{h,h'\in L_2\cdot [L_1/L_2]}\Delta_{h,h'}' f(x) +O\bigg(\frac{N\cdot L_2}{L_1}\bigg),
\end{align*}
as desired.
\end{proof}

We will also require the following version of $U^2$-inverse theorem, which appears as \cite[Lemma~2.4]{Pel20}.
\begin{lemma}[{\cite[Lemma~2.4]{Pel20}}]\label{lem:U2-inver}
Let $N\ge 1$ and $f\colon\mathbf{Z}\to\mathbf{C}$ be $1$-bounded such that $\on{supp}(f)\subseteq [N]$. If 
\[\snorm{f}_{U_{[\delta'N]}^{2}}^{4}\ge \delta N\]
then 
\[\sup_{\beta\in \mathbf{T}} \bigg|\sum_{x\in \mathbf{Z}}e(\beta x)f(x)\bigg|\gg (\delta \delta')^{O(1)}N.\]
\end{lemma}

We also have the following well-known converse to the $U^2$-inverse theorem. We include the proof, as our definition of the $U^2$-norm is slightly nonstandard.
\begin{lemma}\label{lem:converse}
Let $f\colon\mathbf{Z}\to\mathbf{C}$ be a $1$-bounded function with $\on{supp}(f)\subseteq [\delta^{-1} N]$. If $N\gg \delta^{-O(1)}$ and
\[\sup_{\beta\in \mathbf{T}} \bigg|\sum_{x\in \mathbf{Z}}e(\beta x)f(x)\bigg|\ge \delta N\]
then
\[\snorm{f}_{U^{2}_{[N]}}^{4}\gg \delta^{O(1)}N.\]
\end{lemma}
\begin{proof}
By adjusting implicit constants, we may assume that $\delta$ is smaller than an absolute constant throughout. Let $\beta$ be such that 
\[\bigg|\sum_{x\in \mathbf{Z}}e(\beta x)f(x)\bigg|\ge \delta N\]
and define $f^{(1)}(x) = f(x) e(\beta x)$. Note that
\begin{align*}
\snorm{f}_{U^{2}_{[N]}}^{4} &= \sum_{x\in \mathbf{Z}}\mathbf{E}_{\substack{h_1,h_1'\in [N]\\h_2,h_2'\in [N]}} f(x+h_1+h_2)\ol{f(x+h_1+h_2')}\ol{f(x+h_1'+h_2)}f(x+h_1'+h_2')\\
&= \sum_{x\in \mathbf{Z}}\mathbf{E}_{\substack{h_1,h_1'\in [N]\\h_2,h_2'\in [N]}} f^{(1)}(x+h_1+h_2)\ol{f^{(1)}(x+h_1+h_2')}\ol{f^{(1)}(x+h_1'+h_2)}f^{(1)}(x+h_1'+h_2')\\
&= \sum_{x\in \mathbf{Z}}\mbm{1}_{|x|\le 5\delta^{-1}N}\mathbf{E}_{\substack{h_1,h_1'\in [N]}} \bigg|\mathbf{E}_{h_2\in [N]}f^{(1)}(x+h_1+h_2)\ol{f^{(1)}(x+h_1'+h_2)}\bigg|^2\\
&\ge \bigg(\sum_{x\in \mathbf{Z}}\mbm{1}_{|x|\le 5\delta^{-1}N}\mathbf{E}_{\substack{h_1,h_1'\in [N]}} \bigg|\mathbf{E}_{h_2\in [N]}f^{(1)}(x+h_1+h_2)\ol{f^{(1)}(x+h_1'+h_2)}\bigg|\bigg)^{2}\\
&\qquad\cdot\bigg(\sum_{x\in \mathbf{Z}}\mbm{1}_{|x|\le 5\delta^{-1}N}\mathbf{E}_{\substack{h_1,h_1'\in [N]}} 1\bigg)^{-1} \\
&\gg \delta N^{-1}\bigg|\sum_{x\in \mathbf{Z}}\mbm{1}_{|x|\le 5\delta^{-1}N}\mathbf{E}_{\substack{h_1,h_1'\in [N]}} \mathbf{E}_{h_2\in [N]}f^{(1)}(x+h_1+h_2)\ol{f^{(1)}(x+h_1'+h_2)}\bigg|^{2}\\
& = \delta N^{-1}\bigg|\sum_{x\in\mathbf{Z}}\mathbf{E}_{\substack{h_1,h_1'\in [N]}}f^{(1)}(x+h_1)\ol{f^{(1)}(x+h_1')}\bigg|^{2}\\
& = \delta N^{-1}\bigg(\int_{\mathbf{T}}|\wh{f^{(1)}}(\Theta)|^2\cdot \bigg(\frac{\sin(N\Theta/2)}{N\sin(\Theta/2)}\bigg)^2~d\Theta\bigg)^2\gg \delta^{O(1)}N,
\end{align*}
where, by construction, $|\wh{f^{(1)}}(0)|\ge \delta N$, and therefore $|\wh{f^{(1)}}(\Theta)|\ge \delta N/2$ for $|\Theta|\le \delta^{4}N^{-1}$.
\end{proof}

By writing the $2^k$-th power of the $U^k$-norm for $k\geq 2$ as the sum of the $4$-th powers of $U^2$-norms of differenced functions and applying \cref{lem:U2-inver,lem:converse,lem:rescale-down}, and then iterating, we thus deduce the following rescaling inequality for the $U^k$-norm. 
\begin{corollary}\label{cor:rescale-loss}
Fix an integer $k\ge 2$. Let $f\colon\mathbf{Z}\to\mathbf{C}$ is $1$-bounded such that $\on{supp}(f)\subseteq [N]$, $N\ge \delta^{-O_k(1)}$, and
\[\snorm{f}_{U_{[\delta N]}^{k}}^{2^{k}}\ge \delta N.\]
Then if $\delta'\in [\delta,\delta^{-1}]$, we have
\[\snorm{f}_{U_{[\delta' N]}^{k}}^{2^{k}}\gg \delta^{O_k(1)} N.\]
\end{corollary}

We next require the elementary fact that if a $1$-bounded function correlates with an exponential phase on a arithmetic progression of a positive density, this may be extended to the full interval with only polynomial loss. This is essentially \cite[Lemma~3.5(ii)]{GTZ11} or \cite[Proposition~A.4]{Alt22}; we provide a proof for completeness. 
\begin{lemma}\label{lem:interval-corr}
Suppose that $f\colon\mathbf{Z}\to\mathbf{C}$ is a $1$-bounded function such that $\on{supp}(f)\subseteq [\pm N]$, $N\ge \delta^{-O(1)}$, and there exists an arithmetic progression $P$ contained in $[N]$ such that 
\[\sup_{\beta\in \mathbf{T}}\bigg|\sum_{x\in P}e(\beta x) f(x)\bigg|\ge \delta N.\]
Then, 
\[\sup_{\beta\in \mathbf{T}}\bigg|\sum_{x\in \mathbf{Z}}e(\beta x) f(x)\bigg|\gg \delta^{O(1)} N.\]
\end{lemma}
\begin{proof}
Since $f$ is $1$-bounded, $P$ must have length at least $\delta N$. Therefore, $\mbm{1}_{P}(x) = \mbm{1}_{q|(x-a)}\mbm{1}_{I}(x)$ for an interval $I$ of length at least $\delta N$ and $0\le a<q\le \delta^{-1}$. Let 
\[P_2 = \mbm{1}_{P}\ast\bigg(\frac{\mbm{1}_{q|x}\mbm{1}_{[\delta^{3}N]}(x)}{q^{-1}\cdot \delta^{3}N}\bigg).\]
By construction, there exists $\beta\in \mathbf{T}$ such that 
\[\bigg|\sum_{x\in \mathbf{Z}}P_2(x) e(\beta x)f(x)\bigg|\ge 2^{-1}\delta N.\]
Therefore, letting $f_\beta(x) = e(\beta x)f(x)$ and taking the Fourier transform, we have 
\begin{align*}
2^{-1}\delta N &\le \sup_{\beta\in \mathbf{T}}\bigg|\sum_{x\in \mathbf{Z}}P_2(x) e(\beta x)f(x)\bigg|=\sup_{\beta\in \mathbf{T}}\bigg|\int_{\mathbf{T}}\wh{P_2}(\Theta)\cdot \wh{f_\beta}(\Theta)~d\Theta\bigg|\\
&\le \sup_{\substack{\beta\in \mathbf{T}\\ \Theta\in \mathbf{T}}}|\wh{f_\beta}(\Theta)|\cdot \int_{\mathbf{T}}|\wh{P_2}(\Theta)|~d\Theta\ll \delta^{-O(1)} \sup_{\substack{\beta\in \mathbf{T}\\ \Theta\in \mathbf{T}}}|\wh{f_\beta}(\Theta)|\\
&\ll \delta^{-O(1)} \sup_{\beta\in \mathbf{T}}\Big|\sum_{x\in \mathbf{Z}}e(\beta x)f(x)\Big|,
\end{align*}
where we bound the $L^{1}$-norm of $\wh{P_2}(\Theta)$ by using the Cauchy--Schwarz inequality.
\end{proof}

Analogously to~\cref{cor:rescale-loss}, by writing the $2^k$-th power of the $U^k$-norm for $k\geq 2$ as the sum of $4$-th powers of the $U^2$-norms of differenced functions, and applying \cref{lem:U2-inver,lem:converse,lem:rescale-down,lem:interval-corr} and then iterating, we thus deduce another rescaling inequality for the $U^k$-norm. 
\begin{corollary}\label{cor:mod-loss}
Fix an integer $k\ge 2$. Let $L\le \delta^{-1}$, $f\colon\mathbf{Z}\to\mathbf{C}$ be $1$-bounded such that $\on{supp}(f)\subseteq [\pm N]$, and $N\ge \delta^{-O_k(1)}$. If 
\[\snorm{f}_{U_{L\cdot [\delta N/L]}^{k}}^{2^{k}}\ge \delta N,\]
then
\[\snorm{f}_{U_{[\delta N]}^{k}}^{2^{k}}\gg \delta^{O_k(1)} N.\]
\end{corollary}

\end{document}